\documentclass[preprint,12pt]{elsarticle}

 \usepackage{amsthm}
 \usepackage{enumitem}
 \usepackage{hyperref}

\newcommand{\changetwo}[1]{{\color{black}#1}}  

\newcommand{\revisionthree}[1]{{\color{red}#1}}  

 \usepackage{amssymb}
 \usepackage{amsmath}
 \usepackage{dsfont}
 \usepackage{mathtools}
 \usepackage{xcolor}
 \usepackage{mathrsfs}
 \usepackage{amsmath}
\usepackage[uline]{hhtensor}
 \usepackage{float}
\usepackage{graphicx,graphics,subfigure}
\usepackage[titlenumbered,ruled]{algorithm2e}
 \usepackage{comment}
\usepackage{booktabs}

 \usepackage[left=3cm,right=3cm,top=3cm,bottom=3cm]{geometry}

 \newcommand{\npoints}{ N }
 \newcommand{\pointdim}{ n }

\newtheorem{conjecture}{Conjecture}

\theoremstyle{definition}
\newtheorem{definition}{Definition}[section]

 \newtheorem{remarkex}{Remark}
\newenvironment{remark}
  {\remarkex}
  {\endremarkex}

\newtheorem{prop}{Proposition}

\theoremstyle{remark}

\theoremstyle{definition}

\theoremstyle{remark}

\theoremstyle{remark}

\usepackage{ulem}
\usepackage{soul}
\usepackage{xcolor}

 \journal{Applied Mathematics and Computation}
    
\begin{document}
   \begin{frontmatter}
   
    \title{Learning a robust shape parameter for RBF approximation}

    \author[label1]{Maria Han Veiga}
   
 \affiliation[label1]{organization={Department of Mathematics, Ohio State University},
             country={United States of America}}

 \author[label2]{Faezeh Nassajian Mojarrad}
 \affiliation[label2]{organization={Max Planck Institute for Informatics},
             country={Germany}}

\author[label3]{Fatemeh Nassajian Mojarrad}
             \affiliation[label3]{organization={Department of Computer Science, University of Geneva},
             country={Switzerland}}

    \date{}

    \begin{abstract}
    Radial basis functions (RBFs) play an important role in function interpolation, in particular, when considering an arbitrary set of interpolation nodes. The accuracy of the interpolation depends on a parameter called the \textit{shape parameter}. There are many approaches in the literature on how to appropriately choose it to increase the accuracy of interpolation while avoiding stability issues. However, finding the optimal shape parameter remains a challenge in general. We introduce a new method for determining the shape parameter in RBFs. First, we construct an optimization problem to obtain a shape parameter that leads to a bounded condition number for the interpolation matrix, then, we introduce a data-driven method that controls the condition number of the interpolation matrix to avoid numerically unstable interpolations, while keeping good accuracy. In addition, a fallback procedure is proposed to enforce a strict upper bound on the condition number, as well as a learning strategy to improve the performance of the data-driven method by learning from previously run simulations. Several numerical results are presented to demonstrate the robustness of our strategy in both 1- and 2-dimensional spaces.
\end{abstract}

\end{frontmatter}

\section{Introduction}
\label{sec:introduction}
Radial basis functions (RBFs) have a variety of applications; in function approximation, function interpolation, and numerical representation of solutions of partial differential equations (PDEs). The main advantages of using RBFs for interpolation are:
\begin{itemize}
    \item RBFs\footnote{Assuming positive definiteness of the RBF, as in Definition \ref{def:positive_definite}} lead to a well-posed problem when interpolating on an arbitrary set of points of dimension $\pointdim \in\mathbb{N}$ (e.g., Chapter 6 \cite{Wendland2004});
    \item Treating multi-dimensional problems is simple, as many RBFs are uni-dimensional functions of distance.
\end{itemize}

The interpolation problem reads: given data $\{x_i,f(x_i)\}_{i=1}^N$, with $x_i\in\mathbb{R}^\pointdim$ and $f(x_i)\in\mathbb{R}$, find an interpolation function $s$ such that $s(x_i)=f(x_i)$ for all $i$. Using RBFs, we seek the approximating function $s$ of the form:
\begin{align}
\label{eq:rbf_interpolation}
    &s(x) = \sum_{i=1}^N \lambda_i \phi(||x-x_i||_2;\varepsilon) \mbox{ subject to } s(x_i) = f(x_i), \quad i = 1,...,\npoints,
\end{align}
where $\lambda_i$ are unknown coefficients, $\phi$ is a radial function, $x_i$ are the interpolation nodes and $\varepsilon$ represents the \textit{shape parameter} of $\phi$, which determines the shape of the RBF and significantly influences the accuracy of the approximation (e.g., for multiquadric basis functions \cite{Carlson1991}).

We can write \eqref{eq:rbf_interpolation} as a system of equations
\begin{equation}
\label{eq:system_of_eq}
    \matr{A}(\vec{x},\varepsilon)\vec{\lambda} = \vec{f}(\vec{x}) ,
\end{equation} 
where $\matr{A}(\vec{x},\varepsilon)\in \mathbb{R}^{\npoints \times \npoints}$ is the interpolation matrix with entries $a_{ij}=\phi(||x_i-x_j||,\varepsilon)$, $\vec{x}$\footnote{Throughout the paper, \vec{x} denotes both the set of points  $\vec{x}=\{x_1,...,x_N\}\subset\mathbb{R}^n$ and the matrix that contains each interpolation point as a column vector $\vec{x}=(x_1,...,x_N)\in\mathbb{R}^{n\times N}$ when the ordering of the points matters.}  denotes a set of interpolation points $\{x_1,...,x_N\}\subset \mathbb{R}^n$ and $\vec{f}(\vec{x}) = (f(x_1),...,f(x_N))$ is the unknown function evaluated at the interpolation nodes.

\begin{table}[h]
\centering
\caption{Common positive definite RBFs. \label{tab:rbfs}}
\begin{tabular}{|c| c |} 
\hline
Name of RBF & $\phi(r)$\\
\hline
Gaussian &$\exp(-(\varepsilon r)^2)$ \\
Inverse multiquadric & $\left(1+(\varepsilon r)^2\right)^{-\beta/2}$,$\quad \beta\geq 0$ \\
Mat\'{e}rn & $K_\nu (\varepsilon r) (\varepsilon r)^\nu, \nu > 0$ \\
\hline
\end{tabular}
\label{tab:common rbf}
\end{table}

Some examples of commonly used RBFs can be found in Table \ref{tab:common rbf}, all of  which contain a \textit{shape parameter} $\varepsilon$ multiplying $r$ in their definition. How to correctly choose the shape parameter on a RBF has been a topic of much study; it can be treated as a hyper-parameter that is problem dependent, fixed by trial and error. Another approach is to consider an adaptive strategy that depends on the configuration of the interpolation nodes and on the function values at the interpolation nodes. An early contribution was by Hardy \cite{Hardy1971}, where he suggested an adaptive shape parameter given by
\begin{equation}
\label{eq:hardy}
\varepsilon = \frac{1}{0.815d}, 
\end{equation}
with $d=\frac{\sum_{i=1}^N d_i}{N}$, where $d_i$ represents the distance of the interpolation node  $x_i$ to its nearest neighbor. Later, Franke \cite{Franke1982} proposed the following formulation 
\begin{equation}
\label{eq:franke}
    \varepsilon = \frac{0.8\sqrt{\npoints}}{d_{min}},
\end{equation}
with $d_{min}$ being a diameter of the smallest circle that contains all the given points. The modified Franke method modifies the square-root to read $\varepsilon = 0.8\sqrt[4]{N}/d_{min}$ \cite{modifiedfranke}. Foley \cite{Foley1987} used a similar
value for the shape parameter based on the area of the bounding rectangle to the data. These proposed methods take into consideration the placement of the interpolation nodes $\{x_i\}_{i=1}^\npoints$ but ignore the function values $\{f(x_i)\}_{i=1}^\npoints$.

Foley and Carlson proposed a method by examining statistical properties such as the root-mean-square (RMS) error, taking into account not only the nodal placement but also function values \cite{Carlson1991}. Other methods focusing on reducing the RMS have been proposed \cite{Kansa1992,Rippa1999}. The leave-one-out cross validation (LOOCV) method selects the shape parameter $\varepsilon$ by minimizing the norm of an error vector $\vec{E}(\varepsilon)=(E_1(\varepsilon),...,E_N(\varepsilon))$:
\begin{equation}
\label{eq:loo-cv}
    \arg\min_\varepsilon ||\left(E_1(\varepsilon), ..., E_\npoints(\varepsilon)\right)||_2, \quad \mbox{ where } E_k(\varepsilon) = f(x_k)-s^{(k)}(x_k,\varepsilon),
\end{equation}
where $s^{(k)}$ is the interpolant for a subset of the data that excludes the $(x_k,f(x_k))$ pair. In \cite{goldberg1996}, the LOOCV strategy is adopted to choose the shape parameter for multiquadric basis in the context of approximating PDEs. Later, Rippa \cite{Rippa1999} proposed another algorithm based on LOOCV which does not require the computation of $N$ distinct interpolants $s^{(k)}$, $k=1,...,N$ but rather computes each entry in the error vector through an entry-wise division:
\begin{equation}
    \label{eq:rippa}
    E = \matr{A}^{-1}(\vec{x},\varepsilon)\vec{f}/\mbox{diag}(\matr{A}^{-1}(\vec{x},\varepsilon)).
\end{equation}
This idea has been further developed: \cite{Marchetti2021} considers a more general $k$-fold cross-validation setting, \cite{marchetti2022} extends the method to a stochastic framework. In \cite{Scheuerer2011}, a connection to traditional spatial statistics methods \cite{stein99} is established. The shape parameter is analogous to the kernel's bandwidth, and this can be attained through solving a maximum likelihood estimation (MLE) problem. By assuming the unknown function $f$ is sampled from a Gaussian Process $GP(0,k(x,x';\varepsilon))$,  the vector of the observed data $Z(\vec{x}) = (f(x_1),...,f(x_N))\in\mathbb{R}^N$ is distributed as a multi-dimensional Gaussian distribution with vector mean $\vec{0}$ and covariance matrix $
\matr{K}_{ij}= k(x_i,x_j;\varepsilon)$. Then, the shape parameter $\varepsilon$ can be attained through maximizing the conditional probability $p(Z(\vec{x})|\varepsilon)$ with respect to $\varepsilon$: 
\begin{equation}
\label{eq:loglikelihood}
   \arg\max_\varepsilon \left[ -\ln(\det(\matr{A}(\vec{x},\varepsilon)) - \ln(\vec{f}^T \matr{A}^{-1} 
\vec{f})\right].
\end{equation} 

Both minimization problems \eqref{eq:loo-cv} and \eqref{eq:loglikelihood} are typically solved by a grid search, where the error is computed over a set of candidate values for $\varepsilon$. Recently, \cite{sergey} defines candidate sets iteratively and introduces novel LOOCV based algorithms.

In RBF, a key challenge is solving the problem without encountering the Runge phenomenon. In \cite{Fornberg2007}, the shape parameter was determined with consideration of the limitations imposed by the Runge phenomenon. They suggested that instead of using a constant parameter, the data should be clustered and a separate shape parameter calculated for each cluster. Other efforts in spatially variable shape parameters in RBF have been explored \cite{Kansa1992,zhang2007}, as they offer adaptability to local variations, enabling dynamic adjustments to capture intricate spatial features. 

More recently, techniques from optimization and machine learning have been considered  to improve approximations using RBFs. For example, \cite{koupaei2019} uses a Particle Swarm Optimization algorithm, combined with the well-known Kansa method \cite{Kansa1992} to find a good shape parameter to solve PDEs. In \cite{sun2023}, a random walk algorithm is employed to iteratively propose a shape parameter. In \cite{cavoretto2024-1, cavoretto2024-2}, a Bayesian optimization approach is proposed to fine-tune the shape parameter for various kernel based RBFs and for the Partition of Unity meshfree method. Various greedy approaches \cite{wenzel2022, dutta2021} have been considered to best choose the interpolation centers. Other approaches have been proposed based on \textit{kernel machines} \cite{wenzel2024}, where the interpolation nodes and shape parameters are optimized together. In our previous work, we introduced a technique for training a neural network (NN) to predict an optimal shape parameter for Gaussian and inverse multiquadric RBFs \cite{Mojarrad2023}.
Numerical experiments demonstrated the accuracy and robustness of the method.
However, a specific structure for the cloud points is needed for $\pointdim=2$. In this work, we present a natural continuation of that work by making the following contributions:
\begin{itemize}
    \item We formulate an easy to solve optimization problem to find the shape parameter $\varepsilon$ for any set of points $\vec{x} \subset \mathbb{R}^\pointdim$ of size $N$ while controlling the condition number of the interpolation matrix $\matr{A}$ (Section \ref{sec:optimization});
    \item Using the optimization problem, we create a dataset that is used to train a NN to predict the shape parameter $\varepsilon$ given any set of points $\vec{x} \subset \mathbb{R}^\pointdim$ of size $N$ (Sections \ref{sec:dataset} and \ref{sec:ml});
    \item We provide a fallback procedure that guarantees that the proposed shape parameter generates a well-posed interpolation matrix (Section \ref{sec:fallback}).
    \item We propose a retraining strategy to improve the performance of the data-driven method after simulations are run.
\end{itemize}
In this work, we do not need to explicitly treat the 1-dimensional and 2-dimensional cases differently, instead, we consider the distances between the nodes, thus the task depends on the number of interpolation points $\npoints$.

The paper is structured in the following way: in Section \ref{sec:rbf} we provide a summary of results about RBFs, in Section \ref{sec:optimization} we describe the 1-dimensional optimization problem with respect to $\varepsilon$ that bounds the condition number of the interpolation matrix in order to balance double precision and interpolation error, then, in Sections \ref{sec:dataset} and \ref{sec:ml} we describe the dataset and data-driven method to predict the shape parameter $\varepsilon$, in Section \ref{sec:fallback} we present the fallback procedure that ensures the generated interpolation matrices remain well-conditioned and in Section \ref{sec:retraining} we describe a retraining strategy which enables the data-driven method to improve as simulations progress. In Section \ref{sec:numerical} we present some numerical experiments and test the performance of our proposed model. Lastly, conclusions and future directions are provided in Section \ref{sec:conclusion}.

\section{Radial basis functions}
\label{sec:rbf}
\begin{definition}
\label{def:positive_definite}
A RBF $\phi$ on $[0,\infty)$ is \textbf{positive definite} on $\Omega \subset \mathbb{R}^\pointdim$, if for all choices of sets with finite number of points $\vec{x}:=\{x_1,...,x_N\} \subset \mathbb{R}^n$ and arbitrary $\npoints\in\mathbb{N}$, the symmetric $\npoints \times \npoints$ matrix $\matr{A}$ with entries $\matr{A}_{ij}=\phi(||x_i-x_j||,\varepsilon)$ is positive definite.
\end{definition}
This is a standard assumption and many commonly used RBFs have this property (e.g., Table \ref{tab:rbfs}).

The formulation \eqref{eq:rbf_interpolation} provides the RBF scheme for interpolating on scattered data. The RBF basis can be augmented by polynomials of a certain maximal degree, leading to the following augmented form of $s$:
\begin{equation}
\label{eq:rbf_poly}
   s(x) = \sum_{i=1}^N \lambda_i \phi(||x-x_i||,\varepsilon) + \sum_{k=1}^m \gamma_k p_k(x),
\end{equation}
with $m = \binom{\text{deg}_m+\left(\pointdim-1\right)}{\pointdim-1}$, where $\text{deg}_m$ denotes the monomial degree, $\gamma_k$ is a linear coefficient and $p_k$ represents a basis for $\mathcal{P}_{\text{deg}_m-1}$, which is the space of polynomials of degree at most $\text{deg}_m-1$. In order to ensure the uniqueness of the RBF interpolant \eqref{eq:rbf_poly} of a function $f$, one requires the following equations to be fulfilled:
\begin{align}
\label{eq:rbf_poly_conditions}
    &s(x_i) = f(x_i), \quad i = 1,...,N, \\
    &\sum_{i=1}^N \lambda_i p_\ell(x_i) = 0, \quad \ell = 1,...,m. \nonumber
\end{align}

The addition of polynomials in the approximation space can be beneficial for several reasons:
\begin{itemize}
    \item Their inclusion leads to a better representation of the constant function and low order polynomials and to a better general accuracy of the approximation \cite{flyer2016},
    \item Some RBFs (e.g., the thin-plate spline) are conditionally positive definite, i.e. they only attain a positive definite interpolation matrix in the subspace that satisfies \eqref{eq:rbf_poly_conditions}. 
\end{itemize}
Throughout this paper, we assume positive definite RBFs.
\subsection{Error and stability}

By Definition \ref{def:positive_definite}, we can define a positive definite kernel 
\[\Phi:\Omega\times\Omega \to \mathbb{R}, \quad \Phi(x,x')=\phi(||x-x'||),\] which is a symmetric real-valued function of two variables. Let us define the function space
\[ H_\Phi(\Omega) = \mbox{span}\{ \Phi(\cdot,x): x \in \Omega \}, \]
and the associated bilinear form
\[  \langle \sum_{i=1}^N c_i \Phi(\cdot,x_i), \sum_{k=1}^N d_k \Phi(\cdot,x_k) \rangle_{\Phi} = \sum_{i=1}^N \sum_{k=1}^N c_i d_k \Phi(x_i,x_k),\]
which defines an inner product on $H_\Phi(\Omega)$.
According to the Moore-Aronszajn theorem, there exists a unique Reproducing Kernel Hilbert space, referred to as the native space $\mathcal{N}_\Phi$ associated with $\Phi$. Its norm is induced by the inner product $\langle\cdot,\cdot \rangle_\Phi$ and it can be constructed as detailed in Chapter 10, \cite{Wendland2004}, expressed as $||\cdot||_{\mathcal{N}_\Phi}$. The pointwise error estimate for any point $x\in\Omega$ between an interpolator defined over interpolation nodes $\vec{x}=\{x_1,...,x_N\}$ and  $f\in \mathcal{N}_\Phi$ can be expressed as \cite{Wendland2004}:
\[ |f(x) - s_{f,X}(x)| \leq P_{\Phi,\vec{x}} ||f||_{\mathcal{N}_\Phi},\]
which shows that the interpolation error can be estimated through two independent quantities:
\begin{itemize}
    \item the native space norm of $f$, which measures the smoothness of $f$ and  is independent of the position of the interpolation nodes;
    \item the \textit{power function} $P_{\Phi,\vec{x}}$ which depends only on the basis function $\Phi$, the placement of the interpolation nodes and $x$.
\end{itemize}
It has been observed numerically and studied theoretically there is a link between the approximation error and the condition number of the interpolation matrix (measured by the smallest eigenvalue of $\matr{A}(\vec{x},\varepsilon)$, denoted by $\mu_{\min}$) for many standard RBFs \cite{schaback1995}. It is known as the \textit{trade-off principle} (or also uncertainty principle) in RBF interpolation, that states that ``\textit{there is no case known where the error $P_\Phi,\vec{x}$ and the sensitivity $\mu_{\min}^{-1}$ are both reasonably small}".

For example, considering the Gaussian or inverse multiquadric basis, one can fix the number of interpolation nodes  and increase the value of $\varepsilon$, this will result in improving the condition number of $\matr{A}$ but will reduce the accuracy of the interpolant. Similarly, by decreasing $\varepsilon$ we can increase the  accuracy of a RBF interpolant, but the matrix $\matr{A}$ will become ill-conditioned \cite{Fasshauer2007}.

In this part, we present some existing results on the asymptotic behavior of the interpolation matrix $\matr{A}(\vec{x},\varepsilon)$ as $\varepsilon\to 0$ or $\varepsilon \to \infty$. In this section, we focus on the Gaussian and the inverse multiquadric RBFs:
\begin{align}
\label{eq:basis_gauss}
     \phi_{\mbox{gauss}}(r) &= \exp(-(\varepsilon r)^2) \\ \phi_{\mbox{imq}}(r) &= \frac{1}{\sqrt{1+(\varepsilon r)^2 }} 
\label{eq:basis_imq}
\end{align}

\begin{prop} Let $\matr{A}(\vec{x},\varepsilon)$ be the interpolation matrix on the ordered distinct points $\vec{x}\in\mathbb{R}^{n\times N}$, using basis functions \eqref{eq:basis_gauss} or \eqref{eq:basis_imq} with shape parameter $\varepsilon$. Then, 
    \[ \lim_{\varepsilon\to\infty} \matr{A}(\vec{x},\varepsilon) = \mathcal{I}_N,\]
    
    \[ \lim_{\varepsilon\to 0} \matr{A}(\vec{x},\varepsilon) = \matr{1}.\]
\end{prop}
\begin{proof}
    We note the structure of $\matr{A}(\vec{x},\varepsilon) = \mathcal{I}_N + Q$, where each entry of $Q_{ij} = \phi(r_{ij};\varepsilon)$ and zero diagonal. As $\varepsilon\to 0$, $\phi(r_{ij};\varepsilon)=1$ for $i\neq j$. Similarly,  as $\varepsilon\to \infty$, $\phi(r_{ij},\varepsilon)=0$ for $i\neq j$.
\end{proof}

\begin{prop} Let $\matr{A}(\vec{x},\varepsilon)$ be the interpolation matrix on the ordered distinct points $\vec{x}\in\mathbb{R}^{n\times N}$, using basis functions \eqref{eq:basis_gauss} or \eqref{eq:basis_imq} with shape parameter $\varepsilon$. Then, the eigenvalues of $ \matr{A}(\vec{x},\varepsilon)$ converge to $1$ as $\varepsilon\to\infty$. 
\end{prop}
\begin{proof}
Let $D(a,r)$ represent the disk with center $a$ and radius $r$ on the complex plane:
\[ D(a,r) = \{ x\in \mathbb{C}: |x-a| \leq r\}. \]
For $\matr{A}\in\mathbb{R}^{N\times N}$, the Gerschgorin circles are defined as $D(1,R_i)$ where $R_i = \sum_{j:i\neq j} |m_{ij}|$. By the Gerschgorin circle theorem, every eigenvalue of $\matr{A}$ lies in some $D(m_{ii},R_i)$.

Namely, $R_i = \sum_{j:j\neq i}^N \phi(r_{ij},\varepsilon) \to 0$ as $\varepsilon\to \infty$, so the eigenvalues lie inside the disk with a shrinking radius centered at 1.
\end{proof}

The other asymptotic regime, $\varepsilon\to 0$, is of special interest due to the connection to Lagrange polynomial interpolation, as shown initially in \cite{driscoll2002} for the univariate case, and subsequently for the multivariate cases. This is called the \textit{flat-limit} regime, which assumes that the set of nodes $\vec{x}$ is arbitrary but fixed. In \cite{barthelme2021}, formulas are derived for the determinant and eigenvalues of the interpolation matrix for small $\varepsilon$ that show the monotonicity of the eigenvalues for kernels with different smoothness and in the univariate and multivariate cases. For example, in the univariate case for the Gaussian kernel, they show that the eigenvalues $\mu_k(\varepsilon)$ are of order $\mathcal{O}(\varepsilon^{2(k-1)})$ (Theorem 4.2).

\section{Controlling the interpolation matrix condition number}
\label{sec:optimization}
Here, we describe a procedure to find a shape parameter $\varepsilon$ for any given set of $N$ points $\vec{x}\subset \Omega \subset \mathbb{R}^n$ that maintains the condition number of the interpolation matrix inside a specified range. 

Consider the logarithm (in base 10) of the condition number of the interpolation matrix:
\[ \mbox{logcond}(\vec{x},\varepsilon):= \log_{10} \left( \mbox{cond}_F(\matr{A}(\vec{x},\varepsilon)) \right),\]
where
\[\mbox{cond}_F(\vec{x},\varepsilon) = ||\matr{A}(\vec{x},\varepsilon)||_F||\matr{A}(\vec{x},\varepsilon)^{-1}||_F,\]
computed with respect to the Frobenius norm:
\[ ||\matr{A}(\vec{x},\varepsilon)||_F = \sqrt{\sum_{i=1}^N \sum_{j=1}^N |a_{ij}|^2}.\]

Let us consider the following minimization problem:
\begin{equation}
\label{eq:varepsilon-optimization}
    \varepsilon^* = \arg\min_{\varepsilon\in \mathbb{R}^+} \ell(\mbox{logcond}(\vec{x},\varepsilon)),
\end{equation}
where $\ell$ is a simple piecewise linear convex loss:
\begin{equation}
\ell(x) =  \left\{
\begin{array}{ll}
    x - b
    &  x > b \\
    0 & b\geq x \geq a \\
    -x+a & x < a \\
\end{array}.
\label{eq:loss}
\right. 
\end{equation}
for some $a,b\in\mathbb{R}$. A zero loss is attained when the logarithm (in base $10$) of the matrix condition number is between the range $\left[a,b\right]$. This range is chosen as a balance between double precision and interpolation error, but can be changed according to numerical accuracy requirements, for example. In the numerical experiments shown throughout the paper, we set $a=11, b=11.5$ unless otherwise specified. This means that any shape parameter that produces an interpolation condition number between $10^{11}$ and $10^{11.5}$ will minimize the defined loss. We choose this interval to be small so that the \textit{optimal} interpolation matrices condition numbers is small. We adopt a gradient descent optimization strategy to solve \eqref{eq:varepsilon-optimization}. The minimization problem above can provide an $\varepsilon$ for points with any dimension. 

To study the convergence to a critical point when using gradient-descent-type optimization, we study the objective function defined above, which is given by:
\[ \mathcal{L}(\varepsilon) = \ell \circ \log_{10} \circ \mbox{cond}\left(\matr{A} (\vec{x},\varepsilon)\right).\]
The function $\mathcal{L}$ is composed of a convex function $\ell$, a monotonically increasing function $\log_{10}$ and the function $\mbox{cond}$. Thus if $\mbox{cond}(\matr{A}(\vec{x},\cdot))$ is strictly monotonically decreasing with respect to $\varepsilon$, then the optimization problem is simple. Namely, if the condition number of the generated matrix is above the desired range, then $\varepsilon$ must be increased to diminish the matrix condition number, and vice-versa. Thus, we investigate  the derivative $d_\varepsilon\mbox{cond}$ numerically:
\begin{equation}
\label{eq:cond_deriv}
\begin{split}
     d_\varepsilon \mbox{cond}(\vec{x},\varepsilon) &= ||\matr{A}(\vec{x},\varepsilon)^{-1}||_F \mbox{Tr}\left( \matr{A}(\vec{x},\varepsilon) \matr{A}'(\vec{x},\varepsilon) \right)/||\matr{A}(\vec{x},\varepsilon)||_F\\ &+ ||\matr{A}(\vec{x},\varepsilon)||_F \mbox{Tr}\left( -\matr{A}'(\vec{x},\varepsilon)\left[\matr{A}^{-1}(\vec{x},\varepsilon))\right]^3 \right)/||\matr{A}(\vec{x},\varepsilon)^{-1}||_F,
     \end{split}
\end{equation}
where $\matr{A}'(\vec{x},\varepsilon)$ is the derivative of the matrix $\matr{A}(\vec{x},\varepsilon)$ with respect to $\varepsilon$. In Figure \ref{fig:derivatives}, we show the derivative value (under the transformation $\exp(\cdot)$ for visualization convenience) for different selection of $\vec{x}$ nodes in one and two dimensions, while varying $\varepsilon$. We observe that the sign of the derivative, far from $\varepsilon\to 0$, is always negative in our numerical investigations, pointing towards the condition number being strictly monotonically decreasing with respect to $\varepsilon$. We were, however, unable to rigorously prove the statement except for the simple case of $N=2$ (through a direct computation of the Frobenius norm) and leave it as an open conjecture:

\begin{conjecture}
For the Gaussian or inverse multiquadric RBF, let $\varepsilon_1 > \varepsilon_2$. For any set of distinct points $\vec{x}\subset \Omega \subset \mathbb{R}^n$, let $\matr{A}_1 = \matr{A}(\vec{x},\varepsilon_1)$ and $\matr{A}_2 = \matr{A}(\vec{x},\varepsilon_2)$. Then, $\mbox{cond}(\matr{A}_1) \leq \mbox{cond}(\matr{A}_2).$
\end{conjecture}

In \cite{barthelme2021} it has been shown that when $\varepsilon\in [0,\varepsilon_0]$ for small positive $\varepsilon_0$, the monotonicity condition is true for the spectral norm, as the smallest eigenvalue is of order $\mathcal{O}^{2(N-1)}$, whereas the largest eigenvalue is of order $\mathcal{O}(1)$ (when considering an infinitely smooth kernel and in the univariate case). 

In this work, we consider the inverse multiquadric RBFs \eqref{eq:basis_imq} plus the constant polynomial (i.e. $m=1$) as our interpolation basis for all tested methods, in order to improve the approximation quality of our interpolator. We verify numerically that the condition number for this interpolation matrix is strictly monotonically decreasing with respect to $\varepsilon$, as shown in Figure \ref{fig:numerical-derivative}, showing enhanced numerical stability near small $\varepsilon$.

\begin{figure}[h]
\begin{center}
\includegraphics[width=1.0\textwidth]{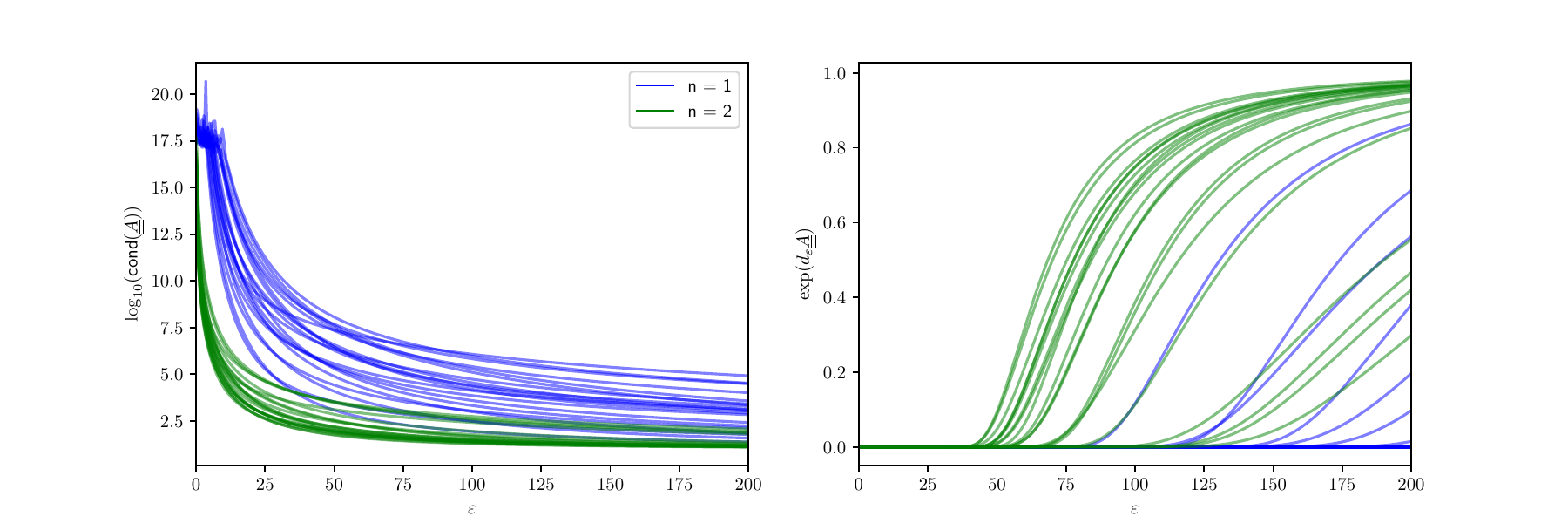}
\caption{Left: Logarithm (base 10) of the interpolation matrix's condition number as a function of the shape parameter $\varepsilon$ for randomly generated sets of points of size 10 in 1-dimension (blue) and 2-dimensions (green). Right: Exponential of the derivative of the matrix condition number as a function of the shape parameter $\varepsilon$ for randomly generated sets of points, which is bounded in $[0,1]$, showing that the derivative is always negative.}
\label{fig:derivatives}
\end{center}
\end{figure}

\begin{figure}[h]
\begin{center}
\includegraphics[width=1.0\textwidth]{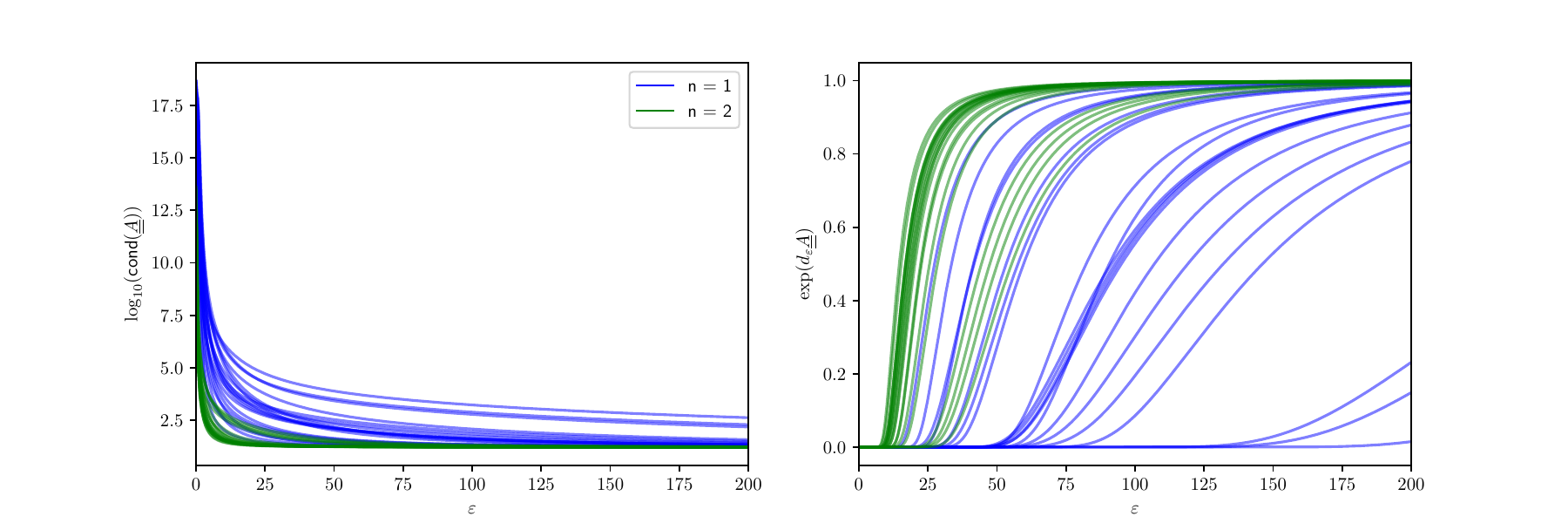}
\caption{Left: Logarithm (base 10) of the interpolation matrix's condition number (including the constant polynomial as basis) as a function of the shape parameter $\varepsilon$ for randomly generated sets of points of size 10 in 1-dimension (blue) and 2-dimensions (green). Right: Exponential of the derivative of the matrix condition number (including the constant polynomial as basis) as a function of the shape parameter $\varepsilon$ for randomly generated sets of points, which is bounded between $[0,1]$, showing that the derivative is always negative.}
\label{fig:numerical-derivative}
\end{center}
\end{figure}

\changetwo{In Figure \ref{fig:sec3_optimization}, we compare the shape parameter attained by this optimization procedure with some traditional methods in Section \ref{sec:introduction}. In particular, we consider the following methods:
\begin{itemize}
    \item The \textbf{Hardy} method, introduced by Hardy in \cite{Hardy1971}, which uses \eqref{eq:hardy} to determine the shape parameter;
    \item The \textbf{Franke} method \cite{Franke1982}, which specifies the shape parameter as in \eqref{eq:franke};
    \item The \textbf{Rippa} method \cite{Rippa1999}, which seeks to minimise the optimization problem \eqref{eq:loo-cv} by expressing the quantity of interest (error) with \eqref{eq:rippa};
    \item The \textbf{MLE} method \cite{Scheuerer2011} that seeks to solve the optimization problem \eqref{eq:loglikelihood}.
\end{itemize}
For the \textbf{Rippa} and\textbf{ MLE} methods, the search interval [0.001, 200] is considered, with 2000 equidistant points, including endpoints.

We consider the 1-dimensional function:
\begin{equation}
    \label{eq:1d-test-fct}
f(x)= \cos(2/\delta \pi x)+x^2+x,
\end{equation}
where $\delta = 0.01, 0.1, 1 , 5$, which also defines the length of the interval where the interpolation nodes are defined. Table \ref{tab:mse-example-1} contains the $L_2$-error between the true function and the interpolator, as computed in \eqref{eq:mse}. We can see that the \textit{Optimization} procedure, as presented above, has led to the smallest $L_2$-error for most of the considered intervals.
}

\begin{table}[h]
\centering
\caption{$L_2$-error of example \eqref{eq:1d-test-fct} (using $M=100$ uniformly sampled points).\label{tab:mse-example-1}}
\begin{tabular}{|c|c|c|c|c|c|} 
\hline
Interval & Rippa & Optimization & MLE & Hardy & Franke \\
\hline
$[0, 5]$ & $5.3000$e-01 & \textbf{$\textbf{7.0316}$e-02} & $3.6670$e+00 & $1.3356$e+01 & $7.9681$e-02 \\
$[0, 1]$ & \textbf{$\textbf{2.6055}$e-07} & $2.1967$e-06 & $1.8776$e-03 & $1.0724$e-01 &$ 2.6927$e-06 \\
$[0, 0.1]$ & $6.2084$e-06 & \textbf{$\textbf{4.3817}$e-09} & $2.6672$e-06 & $1.4293$e-03 & $6.4631$e-06  \\
$[0,0.01]$ & $2.2845$e-04 & \textbf{$\mathbf{8.4440}$e-08} & $2.5863$e-05 & $5.7829$e-04 & $3.0032$e-06 \\
\hline
\end{tabular}
\end{table}

\begin{figure}[h]
\begin{center}
\subfigure[interpolation nodes in $(0,0.01)$]
{\includegraphics[width=0.48\textwidth]{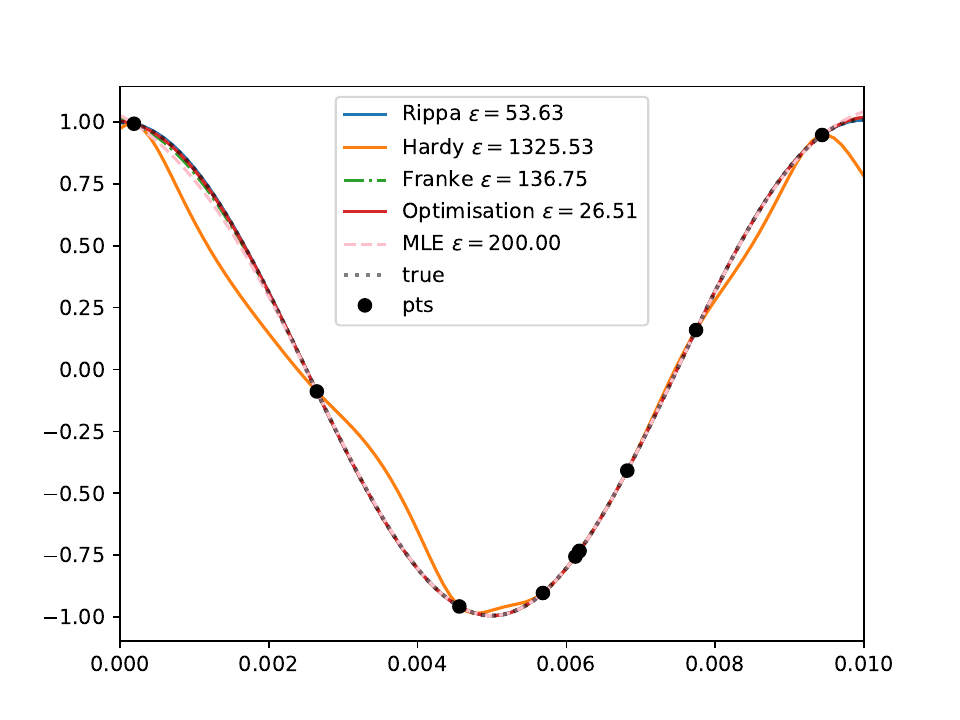}}
\subfigure[interpolation nodes in $(0,0.1)$]
{\includegraphics[width=0.48\textwidth]{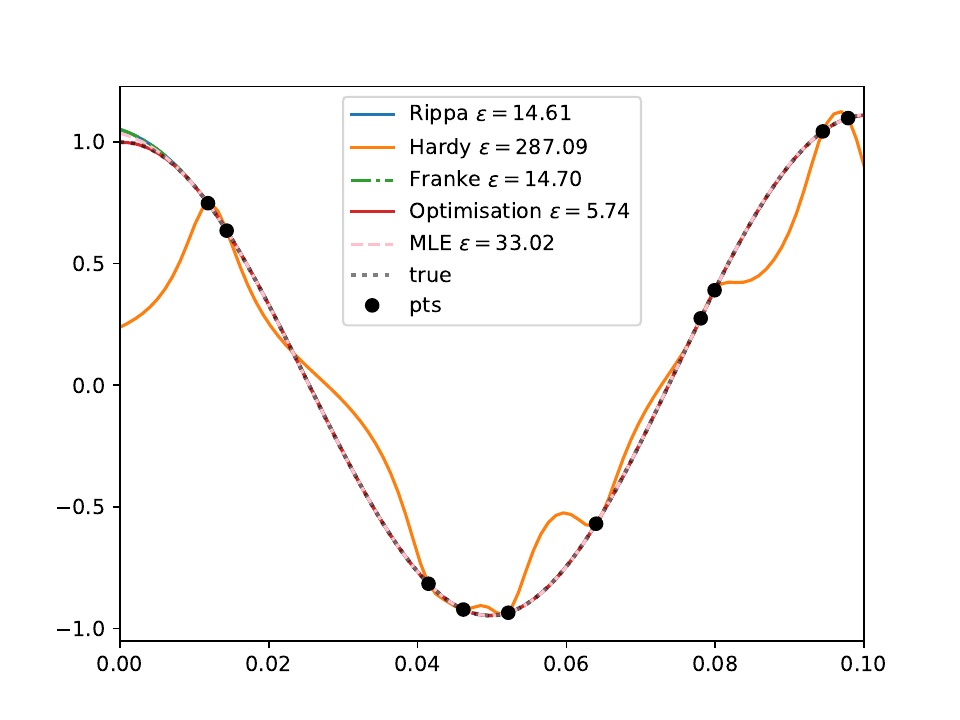}}\\
\subfigure[interpolation nodes in $(0,1)$]
{\includegraphics[width=0.48\textwidth]{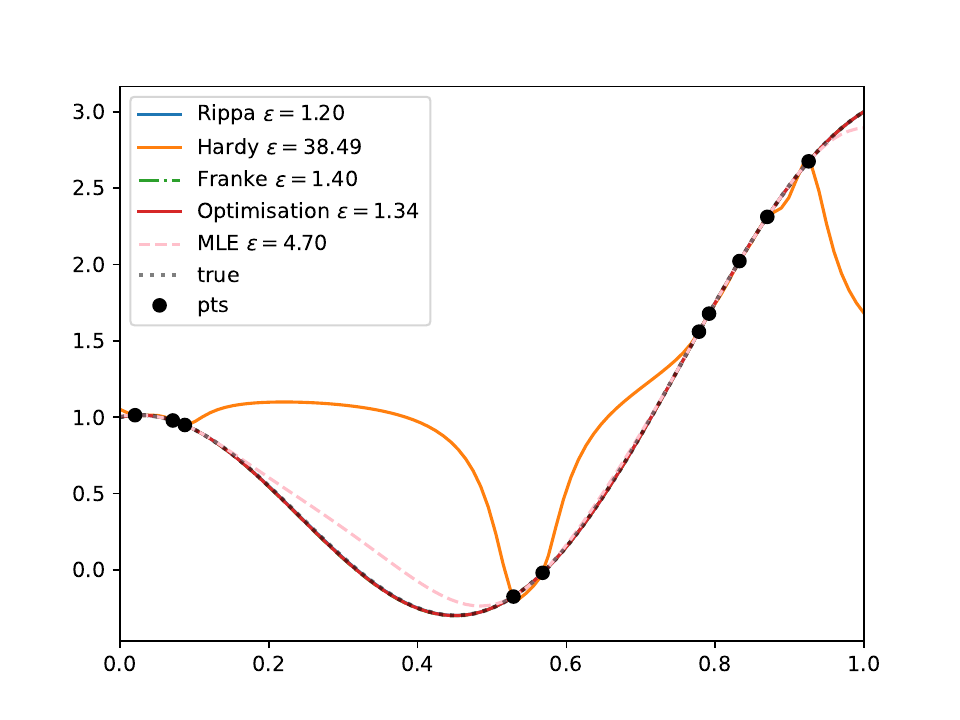}}
\subfigure[interpolation nodes in $(0,5)$]
{\includegraphics[width=0.48\textwidth]{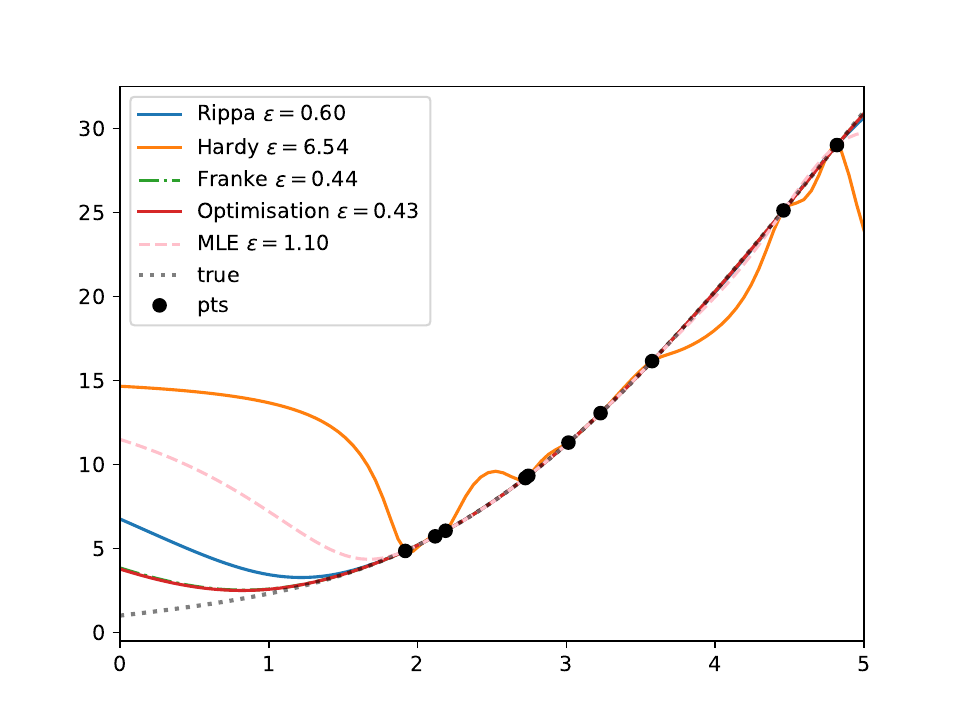}}\\
\end{center}
	\caption{Visual comparison between several adaptive methods: Rippa, Hardy, Franke and MLE methods, and the optimization method presented above on the 1-dimensional function \eqref{eq:1d-test-fct}, varying the interpolation nodes interval length. For the Rippa and MLE methods, we considered the search interval $[0.001,200]$ with 2000 equidistant points.}
	\label{fig:sec3_optimization}
\end{figure}

\section{Learning model}
\subsection{Dataset generation}
\label{sec:dataset}

We generate clouds of points of size $N=10$ embedded in 1-dimension and 2-dimensions randomly by uniformly sampling from a defined interval. For each set of points $\vec{x}$, we use the optimization described in Section  \ref{sec:optimization} to find a suitable $\varepsilon$, generating the pair $(\vec{x},\varepsilon)$, which constitutes a training point. We generate 6300 training points in 1-dimension and 12000 training points in 2-dimensions. We vary the range of the domain that contains a set of points $\vec{x}$, between $[0,0.001]$ to $[0,1]$ to guarantee a good coverage of lengthscales and we also use different interpolation functions to have a sense of the error of the method, although we did not use the interpolation error information for the dataset generation. More details can be found in \ref{app:dataset}, where the datasets are summarized in Tables \ref{table:dataset_1d} and \ref{table:dataset2d} and the algorithm used for generating the dataset is described in Algorithm \ref{alg:data_generation_1d}.

\begin{remark}
While the dimensionality of the interpolation nodes does not change the optimization problem, as it depends on the distance matrix of the points, we still initially distinguished the 1-dimensional from the 2-dimensional setting to understand the learning problem.
\end{remark}

\begin{remark}
The generation of the dataset labels can be replaced. For example, one could use the well-known LOOCV procedure to find optimal $\varepsilon$, at the cost of extra computation and dependency on the function values, which would make the method not agnostic to the function to be interpolated. In particular, this restriction can become relevant in the context of RBF-generated finite difference (RBF-FD) because the solution is not known \textit{a priori} for all times.  
\end{remark}

\subsection{Learning strategy}
\label{sec:ml}

We seek a mapping $f$ from a set of $N$ interpolation nodes $\vec{x}\subset\Omega\subset\mathbb{R}^n$ to an optimal shape parameter $\varepsilon$ such that the interpolation problem is stable and accurate, i.e.   the interpolation matrix $\matr{A}(\vec{x},\varepsilon)$, formed using RBFs with set of interpolation nodes $\vec{x}$ and parameter $\varepsilon$, has the condition number within the specified matrix condition number range. 
\begin{equation}
    \label{eq:approx function}
    f:\vec{x}\in \mathbb{R}^{N\times n} \to \varepsilon \quad \text{s.t. } \text{logcond}(\matr{A}(\vec{x},\varepsilon)) \in [a,b].
\end{equation}
We will approximate \eqref{eq:approx function} through a supervised learning strategy. As a first step, we generate the 1-dimensional and 2-dimensional datasets, see Section \ref{sec:dataset}. Then, given a set of points
$\{x_i \}_{i=1}^N$, we define the distance matrix as follows:
\begin{equation}
\label{eq:distance matrix}
D(\vec{x})=
  \begin{bmatrix}
  ||x_1-x_1|| &   ||x_1-x_2||&  \dots &  ||x_1-x_N||\\
      ||x_2-x_1|| &   ||x_2-x_2||&  \dots &  ||x_2-x_N||\\
    \dots  &  \dots& \dots &  \dots \\
 ||x_N-x_1|| &   ||x_N-x_2||&  \dots &  ||x_N-x_N||    
  \end{bmatrix}.
\end{equation}
The distance matrix does not change with the dimension of the points $x_i$, only with the number of points considered. Furthermore, because the distance matrix is symmetric with zero diagonal, we can consider the features to be given by the inverse of upper part\footnote{Or lower part.} of the matrix $D$, namely, we define the features to be given by $\vec{d}=(\frac{1}{d_{ij}} )_{1\leq i \leq N,i< j\leq N}$ yielding a vector of dimension $N_{in}=N(N-1)/2$. We also sort the cloud of points before generating the distance matrix.

Due to the high-dimension of the input space, we approximate $f$ in \eqref{eq:approx function} with a fully connected neural network (NN).
We assume that the map $f$ 
 can be approximated by $\mathcal{F}$  in the form
\begin{equation}
 \mathcal{F}(\vec d) := \sigma_{k}(...\sigma_1(W_1 \vec{d} + b_1)),
\end{equation}
where $\vec d$ is the input vector, $\{ \sigma_i\}_{i=1}^k$ are the activation functions, $\{W_i\}_{i=1}^k$  are the matrices in $\mathbb{R}^{N_i\times N_{i-1}}$ and $\{b_i\}_{i=1}^k$ are vectors in $\mathbb{R}^{N_i}$. $\{W_i\}_{i=1}^k$ and $\{b_i\}_{i=1}^k$ are learnable parameters, represented generically as $w$. Note that $N_0$ and $N_k$ denote the input and output dimensions, respectively. In this case, $N_0=\frac{N(N-1)}{2}$ and $N_k=1$. 
The architecture of $\mathcal{F}$ is specified in Table \ref{tab:NN_arch}. Each column of Table \ref{tab:NN_arch} represents the following:
\begin{enumerate}
    \item Layer: This column indicates the layer index in the network, starting from the input layer (Layer 1) and moving sequentially through each subsequent layer in the architecture.
    \item Input size: This column specifies the dimensions of the input to each layer.
    \item Output size: This column shows the dimensions of the output from each layer, consistent with the transformations applied by that layer.
    \item Activation: This column describes the activation function applied to the output of the layer.
\end{enumerate}

\begin{table}[h]
\begin{center}
\caption{Architecture of NN, assuming $N=10$. \label{tab:NN_arch}}
\begin{tabular}{|c c c c||} 
 \hline
 Layer & Input size & Output size & Activation \\ [0.5ex] 
 \hline\hline
 Layer 1 & $N_0$ = 45  & $N_1=64$ & ReLU \\
 \hline
 Layer 2 & $N_1=64$ & $N_2=64$ & ReLU \\
 \hline
 Layer 3 & $N_2=64$ & $N_3=64$ & ReLU \\
 \hline
 Layer 4 & $N_3=64$ & $N_4=32$ & ReLU \\
 \hline
 Layer 5 & $N_4=32$ & $N_5=16$ & ReLU \\ 
 \hline
 Layer 6 & $N_5=16$ & $N_6=1$ & Linear \\
 \hline
\end{tabular}
\end{center}
\end{table}

We consider the training paradigm known as supervised learning, in which the desired output value for the points in the training set are known in advance. The goal of the training is to minimize the error between the predictions and the actual values. We use a simple mean squared loss as the loss function:

\begin{equation}
    \label{eq:leastsquareerror}
    \mathcal{L}(w)=\frac{1}{Q}\sum_{i=1}^{Q} (\varepsilon_i - \mathcal{F}(\vec{d}_i;w))^2,
\end{equation}
where $w$ denotes the learnable parameters and $\vec{d}_i$ denotes the feature generated from points $\vec{x}_i$, over a training dataset of size $Q$. The loss function is augmented with a $L_2$ regularization of the networks weights with parameter $\alpha = 0.00005$, which is a standard way to prevent overfitting of the training data \cite{Nowlan1992}
\[
\mathcal{L}(w)=\frac{1}{Q}\sum_{i=1}^Q (\varepsilon_i - \mathcal{F}(\vec{d}_i;w))^2+\alpha ||w||_2 ^2.
\]
Another commonly used approach is the early stopping of the training, which makes use of a validation dataset \cite{Nowlan1992}. After each epoch, the performance of the NN is assessed on the validation dataset, and the loss of its outputs is measured.  Ideally, as training progresses, the loss on the validation set decreases, indicating the NN's ability to generalize. However, beyond a certain point, the validation loss starts increasing after a point due to overfitting. The training is terminated if the loss on the validation dataset increases for $P$ consecutive epochs. We set the patience to $P=200$.
The loss is minimized using the Adam optimizer, with parameters $\beta_0 = 0.9$, $\beta_1 = 0.999$ \cite{Kingma2014} and learning rate $\eta = 10^{-5} $.

\begin{remark}
    Other architectures to approximate the map in \eqref{eq:approx function} have been considered, such as Convolutional Neural Networks and Encoder-Decoder networks but there was no significant benefit to using more complicated architectures, so we omit those results from this work. We also considered the full distance matrix \eqref{eq:distance matrix} as our features, and that lead to worse results.
\end{remark}

\subsubsection{Network training}
We train one NN for both 1-dimensional and 2-dimensional test cases.
The training is performed using a stochastic optimization algorithm that processes mini-batches of size $Q_b$ from the training set to perform a single optimization step. Specifically, the entire training set, containing $Q$ data points, is shuffled, and then mini-batches with $Q_b<Q$ samples are sequentially extracted to complete $Q/Q_b$ optimization steps. Once all the data points in the training set are used, one epoch is completed. The training set is then reshuffled, and the process is repeated for multiple epochs. Shuffling the data introduces stochasticity, which has been observed to accelerate convergence.
The network is trained using  mini-batches of size $Q_b=64$. Training and validation losses are shown in Figure \ref{fig:loss}, indicating no overfitting. The training loss indicates that the model has converged, and no longer significantly improving with more epochs.

\begin{figure}[h]
\begin{center}
{\includegraphics[width=0.45\textwidth]{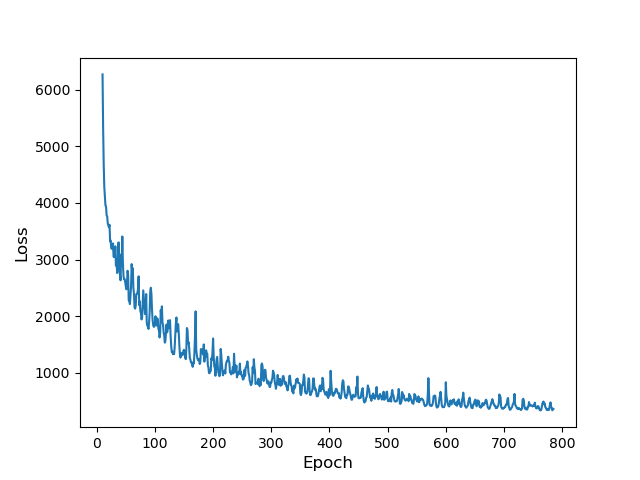}
 \label{fig:train_loss}}
{\includegraphics[width=0.45\textwidth]{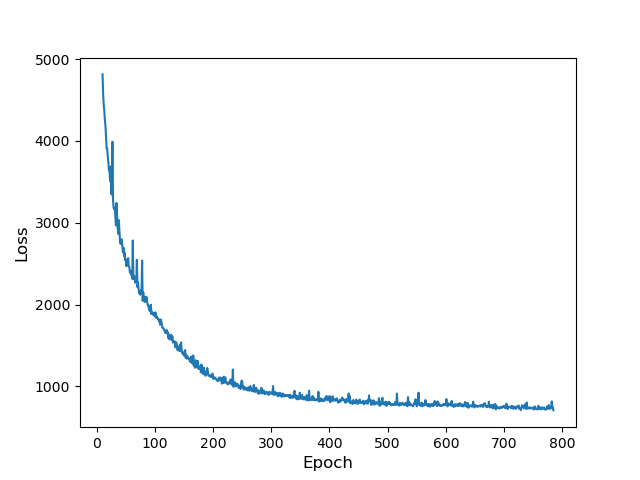}
 \label{fig:validation_loss}}
\end{center}
	\caption{Left: Training loss. Right: Validation loss.}
 \label{fig:loss}
\end{figure}

Once the network has been trained, we will use the same network for all the results shown in Section \ref{sec:numerical}.

\subsection{Performance estimation and fallback procedure}
\label{sec:fallback}
To have an estimate of the performance of the network's output $\hat{\varepsilon}$, we can evaluate the condition number of the interpolation matrix $\matr{A}(\vec{x},\hat{\varepsilon})$ and verify if it is below a certain acceptable threshold. In particular, it is desirable to avoid generating interpolation matrices which are ill-conditioned. Thus, if the logarithm of the condition number of the generated interpolation matrix has a value larger than an acceptable threshold $\theta$ (which is a free parameter), we propose a strategy to correct the shape parameter by calling the optimization procedure defined in Section \ref{sec:optimization} and replacing the predicted $\hat{\varepsilon}$. The full algorithm is shown in Algorithm \ref{alg:fallback}.

\begin{remark}
Other performance indicators could be used. For example, in the case of pure interpolation, the LOOCV error $E$ can be considered, and if it is above a user defined threshold, the LOOCV procedure can be used to replace the predicted $\hat{\varepsilon}$. This was not explored in this work and we leave it to future work.
\end{remark}

\subsection{Learning from simulations}
\label{sec:retraining}
In this section a continual learning setup is described. This allows the data-driven method to improve as simulations are run. Often, the performance of a data-driven method relies on a carefully curated dataset: if test examples are far from the training set examples, it is not guaranteed that the data-driven method will perform well. 

This is an instance of online learning \cite{onlinelearning}, where the data-driven method is retrained on unseen examples that are generated as the method is used in simulation. After the initial training of the NN, the NN is used to solve a variety of tasks, as shown in Sections \ref{sec:interpolation} and \ref{sec:pdes}. Using the fallback procedure, we identify data-points where the network's predictions are  unsatisfactory (based on the condition number of the generated interpolation matrix, as described in \ref{alg:fallback}) and using these data-points, we generate a new dataset that is then used to retrain and improve the NN.

Formally, the setup is the following: let $w_i$ denote the trained parameters of the neural-network and $\mathcal{D}_2 = \{\vec{x}_i, \varepsilon_{opt,i} \}_{i=1}^{m_2}$ denote the newly generated dataset. We consider two strategies to retrain the network:
\begin{itemize}
    \item \textit{Naive} retraining: the weights $w_i$ are updated through minimizing \eqref{eq:leastsquareerror} considering the dataset $\mathcal{D}_2$. This type of update is prone to \textit{catastrophic interference} \cite{Kirkpatrick2017}, where the neural network's performance on previously trained data points can be diminished.
    \item Model merging: a new model is instantiated with weights $w_i$ and trained with through minimizing \eqref{eq:leastsquareerror} considering the dataset $\mathcal{D}_2$. This yields a new set of parameters $w_t$. Then, the weights $w_i$ and $w_t$ are averaged  to generate a new set of parameters $w_f$, which are the new parameters of the data-driven model.
\end{itemize}

\section{Numerical experiments}
\label{sec:numerical}
To test the shape parameter selection method, we evaluate the performance of the method on interpolation tasks and by integrating the method with partial differential equations (PDEs) solvers.  We compare our proposed method with other well-known adaptive shape parameter methods such as Hardy \cite{Hardy1971}, Franke \cite{Franke1982}, modified Franke, Rippa's LOOCV \cite{Rippa1999} and MLE \cite{Scheuerer2011} methods.

Then, we evaluate our approach with the performance indicator, which is a hybrid approach that combines the NN's prediction and the optimization problem defined in Section \ref{sec:optimization}, to guarantee that the proposed method remains robust. Lastly, we  use the retraining strategy, which is computationally faster than the performance indicator.

In the following experiments, we consider the inverse multiquadric kernel with the constant polynomial (i.e. $m=1$) and $N=10$ points to build the interpolation, following the results of \cite{Mojarrad2023}.

We report the $L_2$-error between the approximation via interpolation using a given shape parameter $\varepsilon$ and the exact solution:
\begin{equation}
\label{eq:mse}
     L_2\text{-error}(u_{exact},u_{approx}) = \sqrt{\frac{1}{M}\sum_{i=1}^M|u_{exact}(\tilde{x}_i)-u_{approx}(\tilde{x}_i) |^2},
\end{equation}
where $u_{approx}$ and $u_{exact}$ denote the approximate and exact  solutions, respectively, and $M$ represents the number of evaluation points $\tilde{x}$ considered.

In \cite{Mojarrad2023}, we demonstrated a  clear advantage of the variable shape parameter strategy  using adaptive NN compared to the constant shaped RBF. Also, we observed that the error eventually blow-up in most the cases for constant shape parameter strategy. Hence, our results are only compared with the Hardy, Franke, modified Franke, Rippa and MLE
approaches, which are all adaptive methods. The comparison with the Rippa and MLE approaches will only be done for interpolation tasks, where the function to be interpolated is known \textit{a priori}, as this is a necessity for these methods.

The code developed to generate the datasets as well as the generated datasets, the set up of the NN, and the numerical experiments are in the Github repository \cite{ourgithub}.

\subsection{Interpolation}
\label{sec:interpolation}
\subsubsection{1-dimensional}
\label{sec:interpolation_1d}

We discretize the space domain by the equispaced and non-equispaced centers for 1-dimensional problems, using the   zeros of the Chebyshev polynomial of the first kind with degree $N=10$ remapped to the interval $[0,1]$. The refined mesh is created by adding midpoints. For interpolation task, we  construct the interpolation using sets of
$N$ points, ensuring the boundary node overlaps between clusters.

We examine the performance of the our approach for the function $f_1$, which is the combination of the exponential and trigonometric functions:
\begin{equation}
    \label{eq:test_case_1d:1}
    f_1(x)=\exp(sin(\pi x)), \quad x  \in [0,1].
\end{equation}    

In Figure \ref{fig:inter1d_f1}, we show the error convergence plot with respect to the number of evaluation points when employing shape parameters derived from our NN-based approach and the adaptive methods, like Rippa, Hardy, Franke and modified
Franke   approaches. Our strategy appears to yield a better approximation to interpolate function $f_1$ in most cases, in particular when the interpolation points become very close to each other, i.e. when M is large. Both the Rippa and MLE methods require a list of candidates for $\varepsilon$. We consider the following set of candidate values
\begin{equation}
\label{eq:candidate_epsilons}
    \mathcal{C} := \mathcal{B} \cup \{50,75,100,200,500,1000 \},\footnote{Inspired by the interval considered in the python package PySMO \cite{pysmo}}
\end{equation}where $\mathcal{B}$ contains 200 equidistant points in the interval $[0.001,300]$, starting at 0.001. One can consider a larger grid to search over, at the expense of higher computational cost. We also noted that as the interpolation nodes become closer, the interpolation matrix becomes numerically ill-conditioned for some $\varepsilon$'s, affecting the performance of the method. In \ref{ap:rippa_mle}, we explore the impact of the candidate set and the restriction of the maximum condition number for the interpolation matrix. Based on these findings, we only consider $\varepsilon$'s that lead to a matrix with condition number smaller than $10^{16}$. 

\begin{figure}[h]
\begin{center}
\subfigure[Equispaced]
{\includegraphics[width=0.45\textwidth]{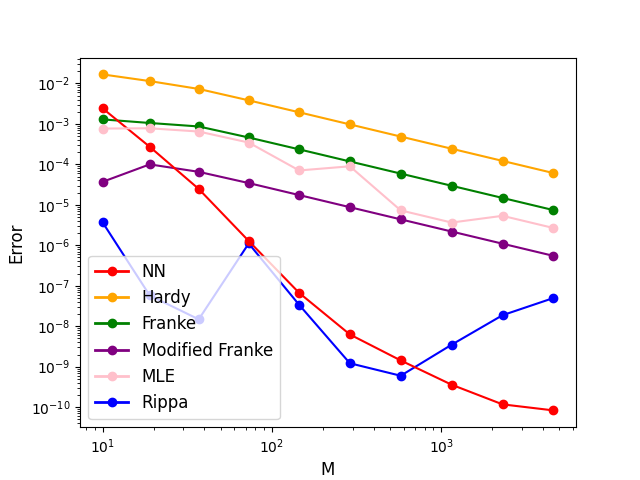}}
\subfigure[Non-equispaced]
{\includegraphics[width=0.45\textwidth]{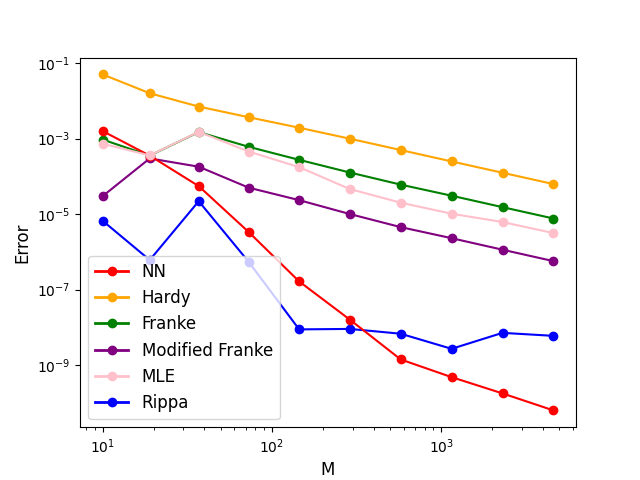}}
\end{center}
	\caption{Plot of error convergence for the 1-dimensional interpolation of $f_1$.}
	\label{fig:inter1d_f1}
\end{figure}

We  consider another test example, which is given by:
\begin{equation}
\label{eq:runge_function}
    f_2(x)=\frac{1}{16x^2+1},~x  \in [0,1].
\end{equation}
We can see in Figure \ref{fig:inter1d_f2} the errors with respect to the number of evaluation points using different approaches. It can be easily seen that in both cases the adaptive NN is able to provide very good approximation for this task.

\begin{figure}[h]
\begin{center}
\subfigure[Equispaced]
{\includegraphics[width=0.45\textwidth]{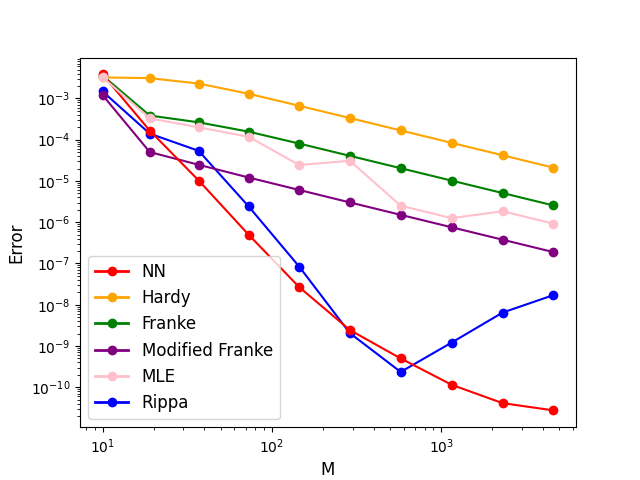}}
\subfigure[Non-equispaced]
{\includegraphics[width=0.45\textwidth]{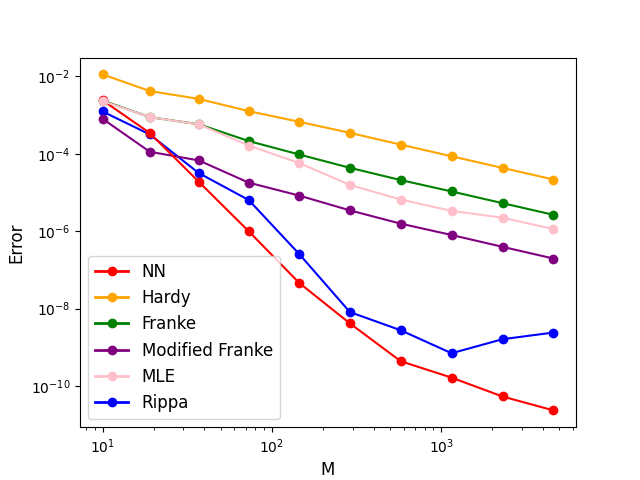}}
\end{center}
	\caption{Plot of error convergence for the 1-dimensional interpolation of $f_2$.}
	\label{fig:inter1d_f2}
\end{figure}

Lastly, we consider a test example that consists of a piecewise constant problem:
\begin{equation}
\label{eq:piecewise_constant}
    f_3(x)=
\begin{cases}
1,&\text{if } x> 0.5, \\
0, &\text{otherwise},
\end{cases},~x  \in [0,1].
\end{equation}
In Figure \ref{fig:inter1d_f3}, we observe how all methods converge similarly, with the NN approach producing slightly lower errors.

\begin{figure}[h]
\begin{center}
\subfigure[Equispaced]
{\includegraphics[width=0.45\textwidth]{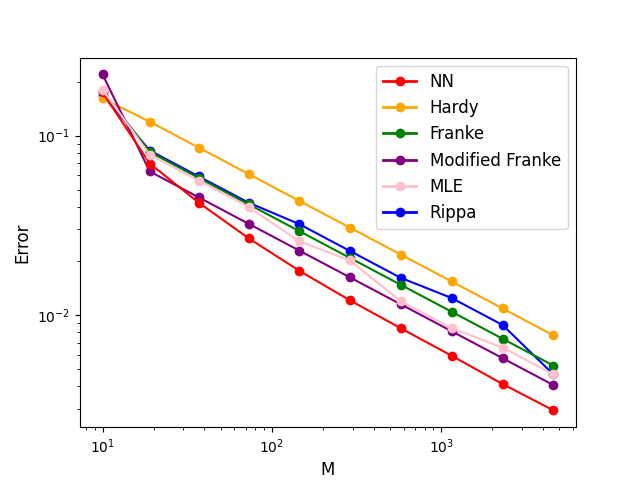}}
\subfigure[Non-equispaced]
{\includegraphics[width=0.45\textwidth]{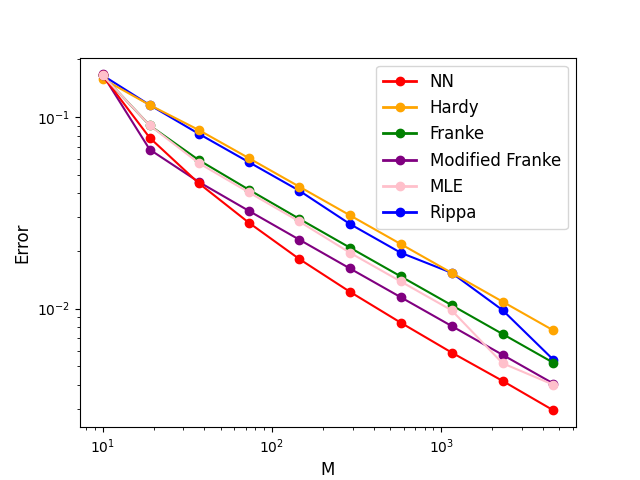}}
\end{center}
	\caption{Plot of error convergence for the 1-dimensional interpolation of $f_3$.}
	\label{fig:inter1d_f3}
\end{figure}

\subsubsection{2-dimensional}
\label{sec:interpolation_2d}
In order to  observe error convergence in two dimensions, we use a regular grid to create  the sets of interpolation and evaluation points. We take the oversampling parameter to be $4$, see \cite{Mojarrad2023} for more details.

We consider the well-known Franke function \cite{Fasshauer2007,Fasshauer2002} (as shown in Figure \ref{fig:inter2d_f4}):
\begin{equation*}
\begin{split}
f_4(x,y)&=\frac{3}{4} \exp \left(-\left( \frac{(9x-2)^2+(9y-2)^2}{4}\right)\right)+\frac{3}{4} \exp\left(-\left( \frac{(9x+1)^2}{49}+\frac{(9y+1)^2}{10}\right)\right)  
 \\&+\frac{1}{2} \exp \left(-\left( \frac{(9x-7)^2+(9y-3)^2}{4}\right)\right)
 -\frac{1}{5} \exp \left(-\left( (9x-4)^2+(9y-7)^2\right)\right) , \\& (x,y)\in [0,1]\times [0,1].
 \end{split}
\end{equation*}

We evaluate the approximation error of the Hardy, Franke, modified Franke, Rippa, MLE and the NN-based approach. For the Rippa and MLE methods, we consider the following set of candidate values for $\varepsilon$:
\begin{align} 
 \label{eq:candidate_epsilons_2d}
A&:= \{0.001, 0.002, 0.005, 0.0075, 0.01, 0.02, 0.05, 0.075, 0.1, 0.2, 0.5, 0.75,1,2, 5, 7.5,  \nonumber \\
& \quad \quad 10, 20.0, 50, 75, 100, 200, 500, 1000\}. 
\end{align}
We do not consider the set $\mathcal{C}$ as it becomes computationally very costly with the increase of the dimension, as the search for the optimal $\varepsilon$ has to be performed for each basis function.  We show in Figure \ref{fig:inter2d_test3}, that while the approximation error decreases in all methods as we refine the mesh, the error using the NN method is significantly smaller than the other considered strategies. Both the Rippa and MLE methods's performance can be improved at the expense of increasing or adapting the candidate set.
\begin{figure}[h]
\begin{center}
\includegraphics[width=0.45\textwidth]{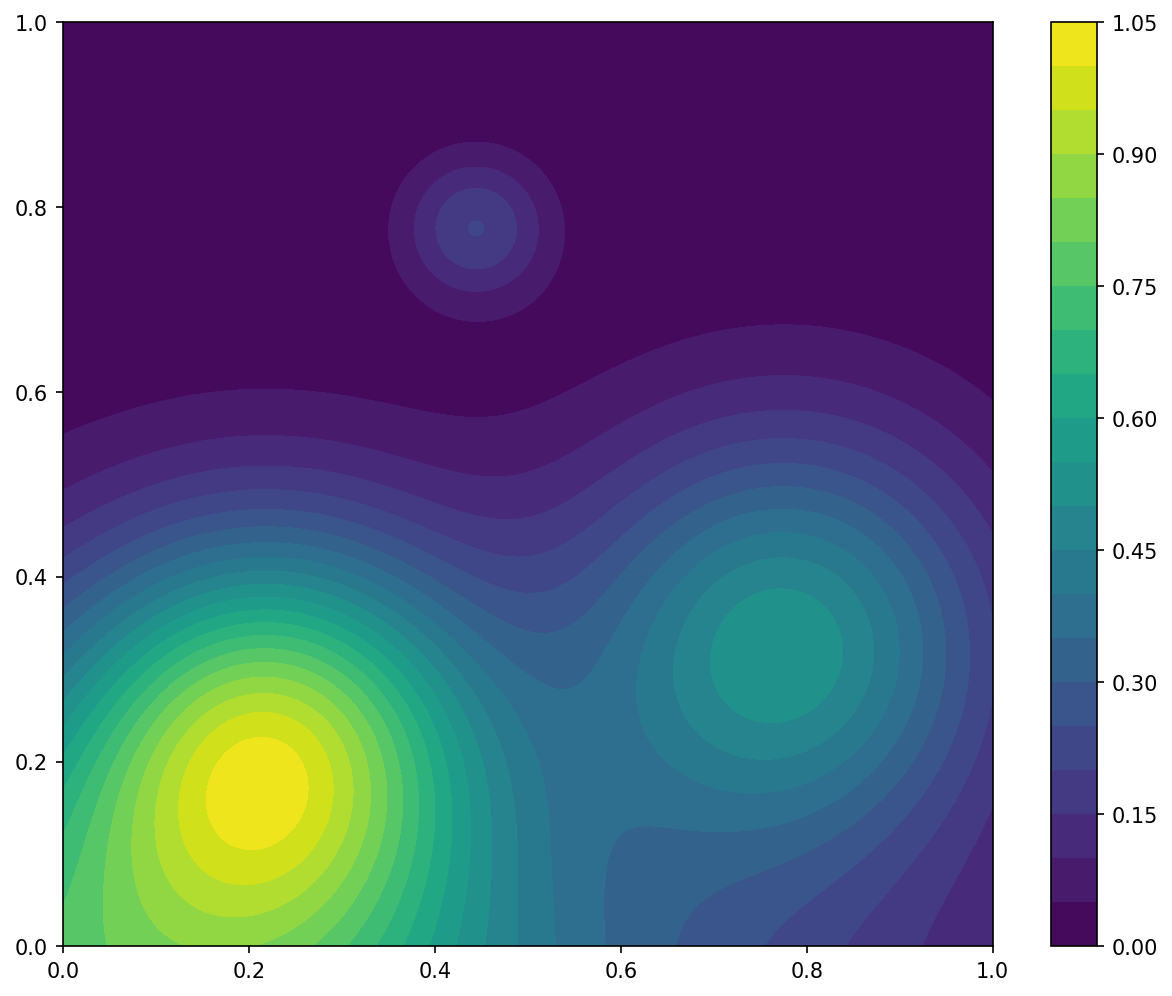}
\end{center}
	\caption{Graphical representation of $f_4$.}
 \label{fig:inter2d_f4}
\end{figure}

\begin{figure}[h]
\begin{center}
\includegraphics[width=0.45\textwidth]{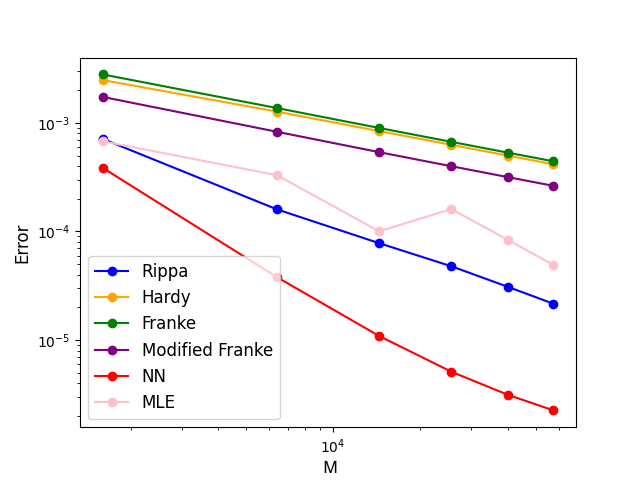}
\end{center}
	\caption{Plot of error for the 2-dimensional interpolation of $f_4$.}
 \label{fig:inter2d_test3}
\end{figure}

Then, we consider the following function, that has been featured in previous works as the initial condition to a nonlinear PDE \cite{Fasshauer2002, Gartland2000,Jankowska2018} (as shown in Figure \ref{fig:inter2d_f5}):
\begin{align}
\label{eq:2d_case1}
  f_5(x,y)=& \left(1+ \exp \left( -\frac{1}{\kappa}\right)-\exp \left(-\frac{x}{\kappa}\right)-\exp\left(\frac{x-1}{\kappa}\right)\right)\times \\
  &\left(1+ \exp\left(-\frac{1}{\kappa}\right)-\exp\left(-\frac{y}{\kappa}\right)-\exp\left(\frac{y-1}{\kappa}\right)\right).
\end{align}

where the computational domain is a square $[0,1]\times [0,1]$. Since the function profile becomes more steep at the boundaries, the complexity of the problem increases for lower $\kappa$. 
In Figure \ref{fig:inter2d}, we plot the error convergence plot considering $\kappa=0.1$ on the left and $\kappa=1$ on the right. The advantages of using the novel method are evident: there is a significant difference in the errors between the adaptive NN and the other strategies. Note that the errors of the Hardy and of the Franke methods are essentially identical. 

\begin{figure}[h]
\begin{center}
{\includegraphics[width=0.45\textwidth]{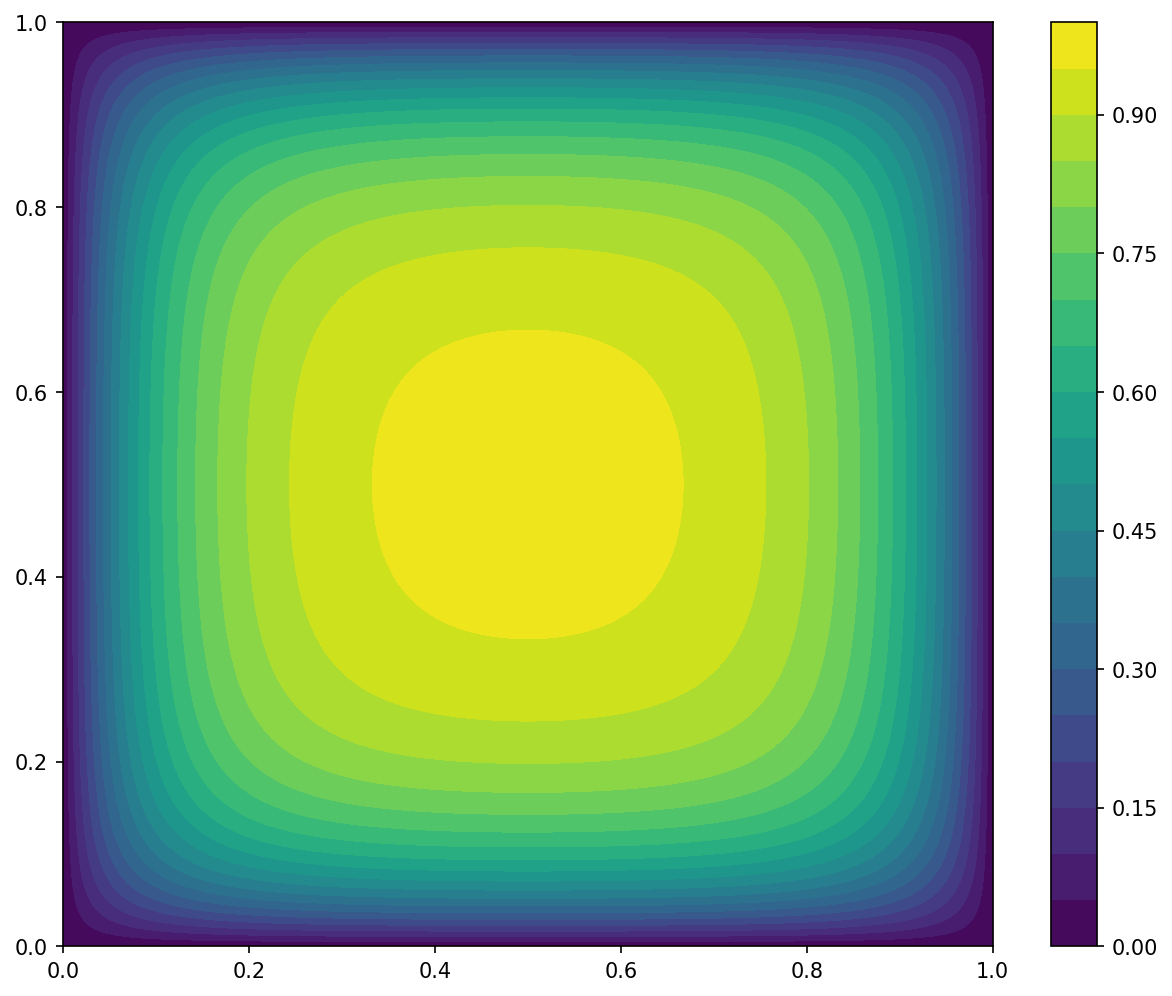}}
{\includegraphics[width=0.45\textwidth]{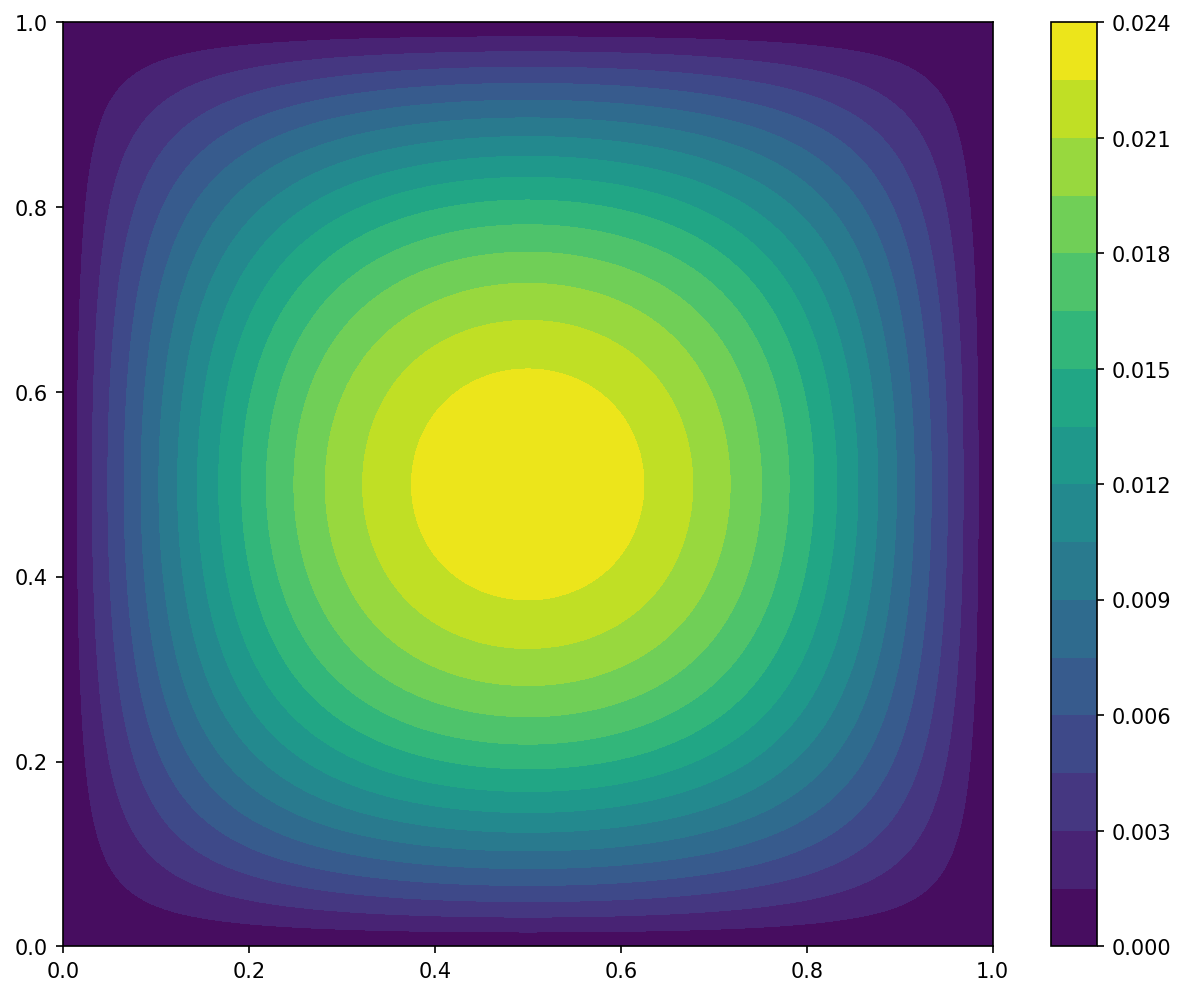}}
\end{center}
	\caption{Graphical representation of $f_5$. $\kappa=0.1$ on the left and $\kappa=1$ on the right.}
 \label{fig:inter2d_f5}
\end{figure}

\begin{figure}[h]
\begin{center}
{\includegraphics[width=0.45\textwidth]{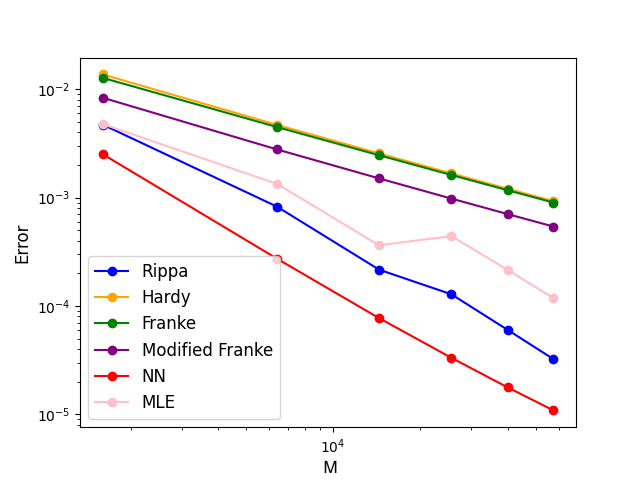}
 \label{fig:inter2d_test1}}
{\includegraphics[width=0.45\textwidth]{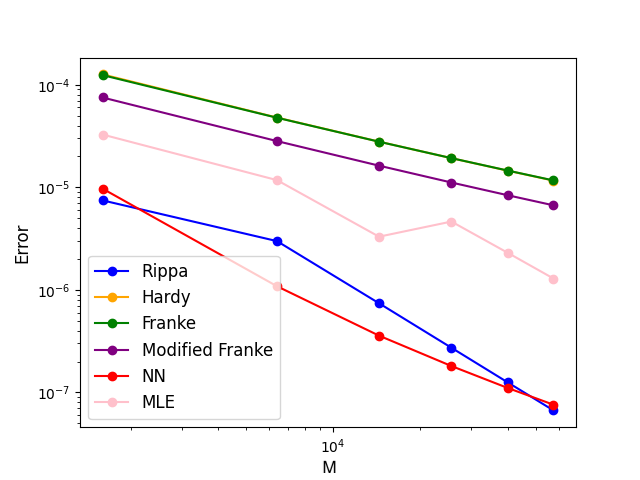}
 \label{fig:inter2d_test2}}
\end{center}
	\caption{Plot of error convergence for the 2-dimensional interpolation of $f_5$. $\kappa=0.1$ on the left and $\kappa=1$ on the right.}
 \label{fig:inter2d}
\end{figure}

\subsubsection{Image reconstruction}
The image reconstruction task can be seen as a 2-dimensional interpolation problem. Let us define an image as a function: 
\[ I: \mathbb{R}^2 \to \mathbb{R}\]
where 
$I(i,j)$ represents the intensity at pixel coordinates 
$(i,j)$.  The objective is to reconstruct the image after performing processing. We use RBF-FD method \cite{Mojarrad2023}. The interpolation point set is the pixel coordinates of the processed image and the evaluation point set is the pixel coordinates of the original image.

We  present the numerical results illustrating the behavior of our scheme on two image reconstruction tasks: \textit{image zoom-in} and \textit{image distortion correction}.

The  steps of the image zoom-in can be summarized as follows: first, downsample the  original image, and then upsample the downsampled image back to its original size by reconstruction using RBFs. The intensity values are upsampled based on the downsampled image's pixel grid.

Here, we measure the error using the mean squared error (MSE) and peak signal to noise ratio (PSNR) \cite{Gonzalez2002}. Let $I$ and $\hat{I}$ be a pixel intensity of the original and approximating image, respectively, with $
m_1$ and $m_2$ number of rows and columns of pixels:
\begin{align*}
\text{MSE}(I,\hat{I})&=\frac{1}{m_1m_2}\sum_{i=1}^{m_1} \sum_{j=1}^{m_2} |I(i,j)-\hat{I}(i,j) |^2,\\
\text{PSNR}&=10\log_{10}\left(\frac{(\max\limits_{i,j} I(i,j))^2}{\text{MSE}(I,\hat{I})}\right).
\end{align*}
The first considered test case is the synthetic checkerboard image (see Figure \ref{fig:image_u1}). The obtained approximating images, along with MSE and PSNR values
are displayed in Figure \ref{fig:image1_upscaled} and Table \ref{tab:mse_upscaled1}, respectively.
One can observe that the image obtained with our approach produce a better quality, when considering the NN strategy, we observe a lower MSE and higher  PSNR  values.

We consider another image, see Figure \ref{fig:image_u2} \cite{Mortensen2000}. We show in Figure \ref{fig:image2_upscaled} and Table \ref{tab:mse_upscaled2} the approximating image and MSE and PSNR values for different strategies, respectively. One can see that the results are similar, demonstrating that the adaptive NN strategy has a superior performance.

\begin{figure}[h]
\begin{center}
{\includegraphics[width=0.35\textwidth]{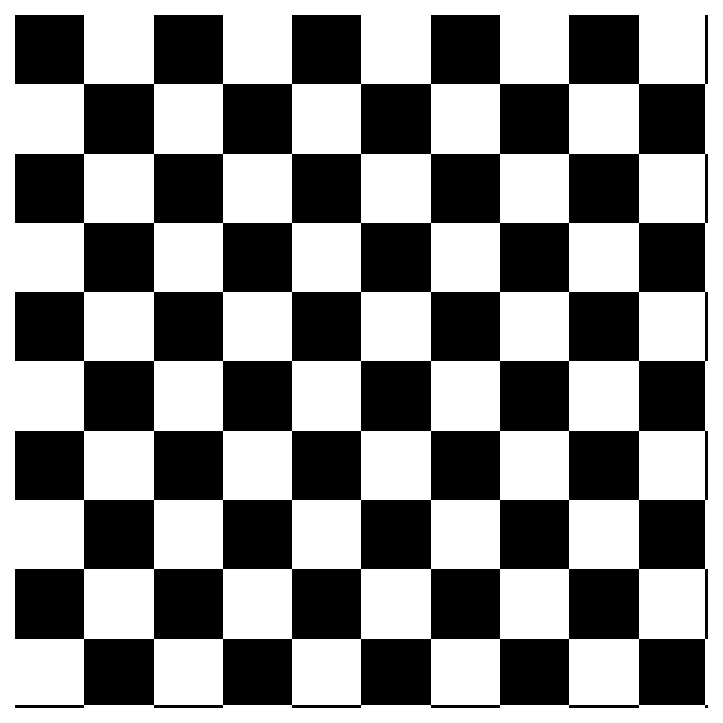}}
{\includegraphics[width=0.175\textwidth]{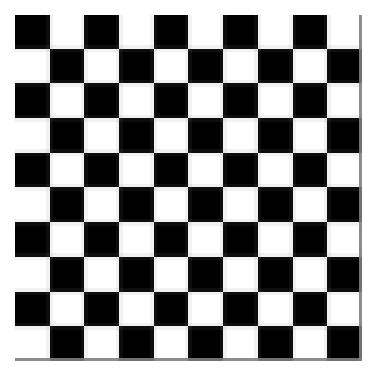}}
\end{center}
	\caption{Left: Original image. Right: Downscaled image.}
 \label{fig:image_u1}
\end{figure}

\begin{figure}[h]
\begin{center}
\subfigure[Rippa]{\includegraphics[width=0.35\textwidth]{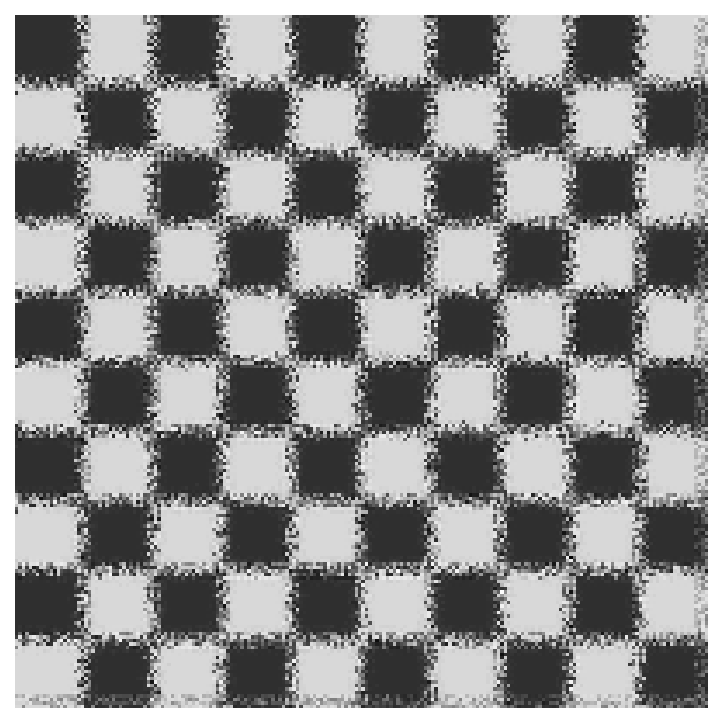}}
\subfigure[MLE]{\includegraphics[width=0.35\textwidth]{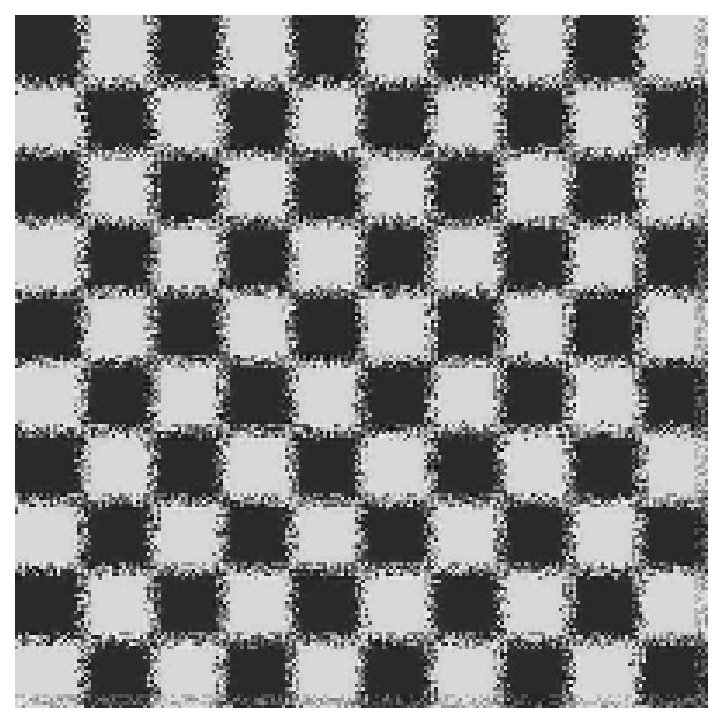}}
\subfigure[Hardy]{\includegraphics[width=0.35\textwidth]{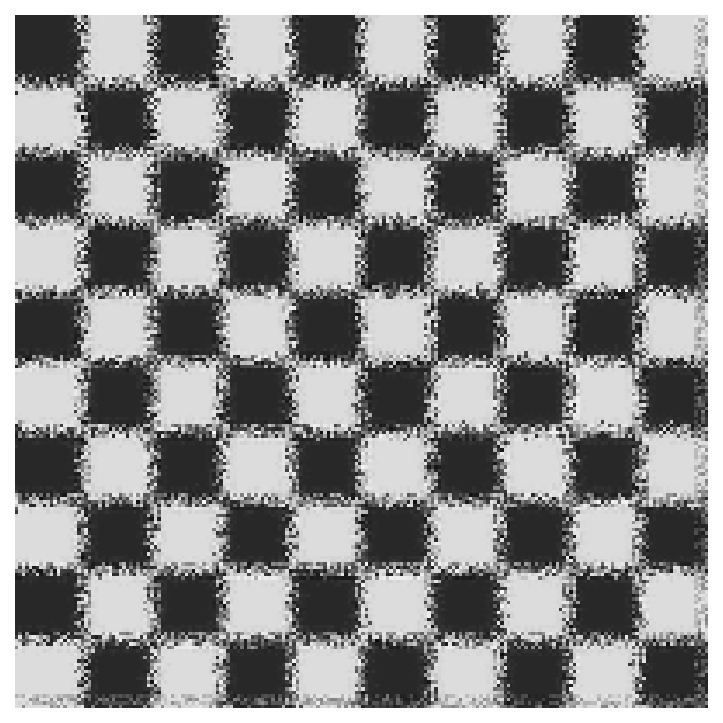}}
\subfigure[Franke]{\includegraphics[width=0.35\textwidth]{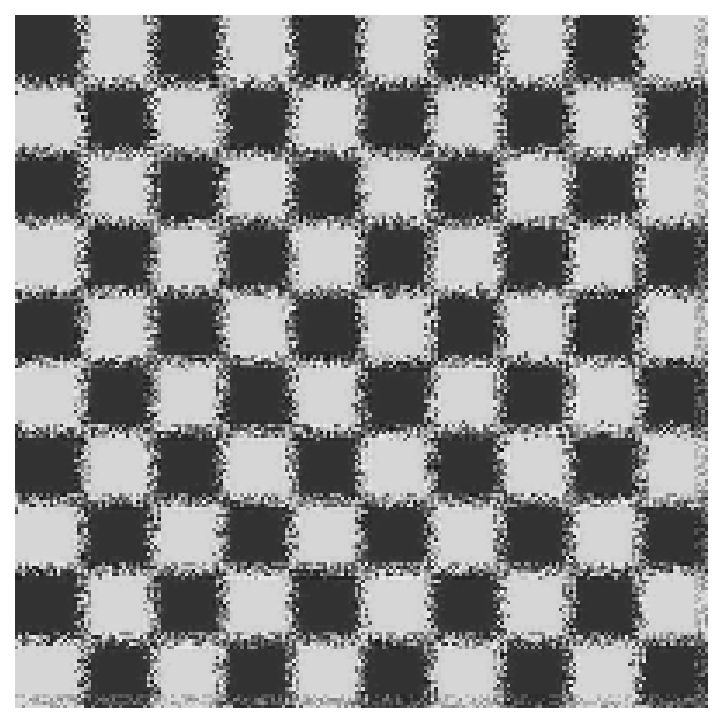}}
\subfigure[Modified Franke]{\includegraphics[width=0.35\textwidth]{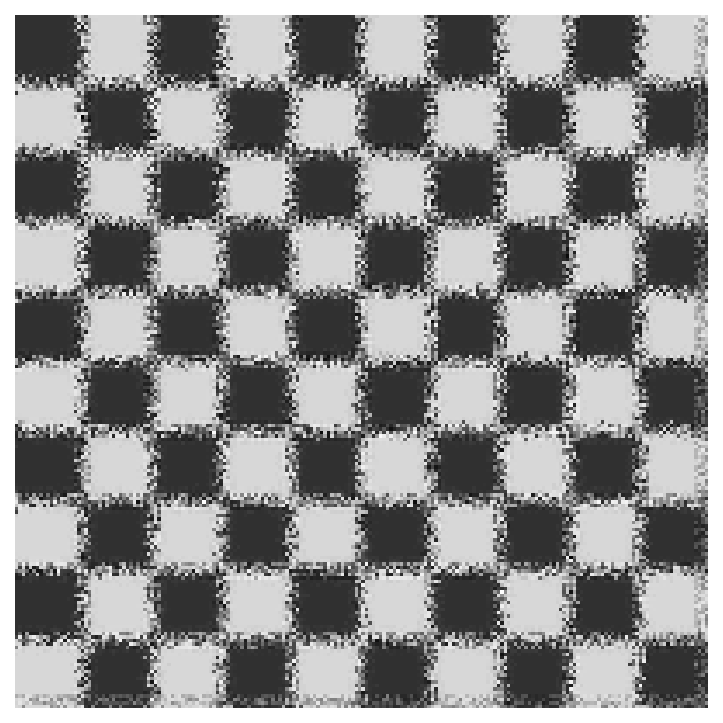}}
\subfigure[NN]{\includegraphics[width=0.35\textwidth]{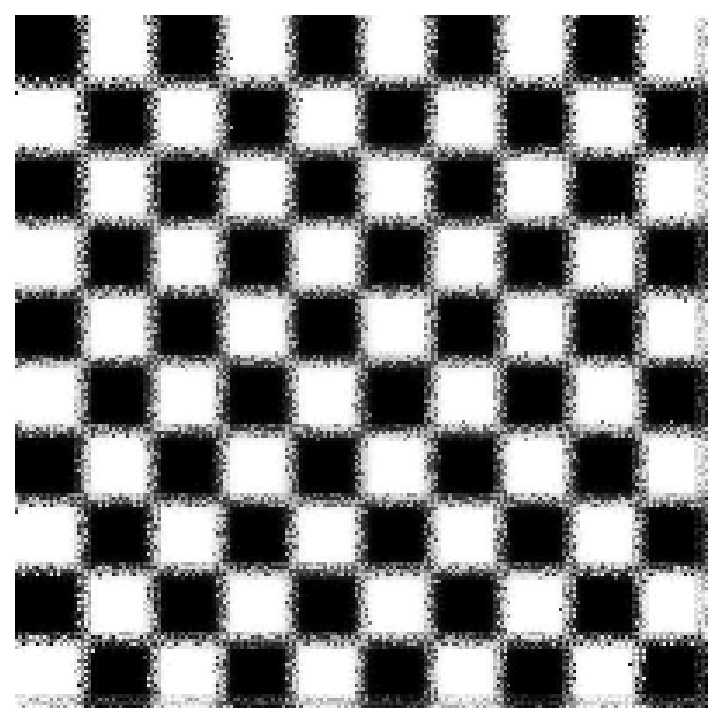}}
\end{center}
\caption{Upscaled images using different strategies.}
\label{fig:image1_upscaled}
\end{figure}

\begin{table}[h]
\centering
\caption{MSE and PSNR with diﬀerent strategies for image zoom-in for Figure \ref{fig:image_u1}.}
\begin{tabular}{|c| c c|} 
\hline
Scheme & MSE & PSNR\\
\hline
Rippa & 8.4151e-02 &10.7494\\
MLE & 8.7580e-02 & 10.5759\\
Hardy & 8.5182e-02 & 10.6965\\
Franke & 8.9298e-02  & 10.4916\\
Modified Franke &8.9302e-02 & 10.4914\\
NN & 6.0895e-02& 12.1542\\
\hline
\end{tabular}
\label{tab:mse_upscaled1}
\end{table}

\begin{figure}[h]
\begin{center}
{\includegraphics[width=0.35\textwidth]{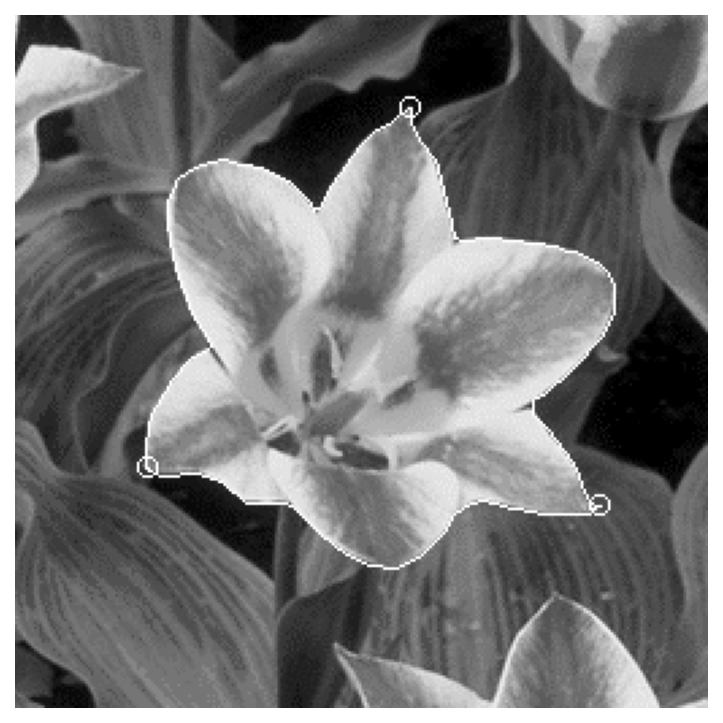}}
{\includegraphics[width=0.175\textwidth]{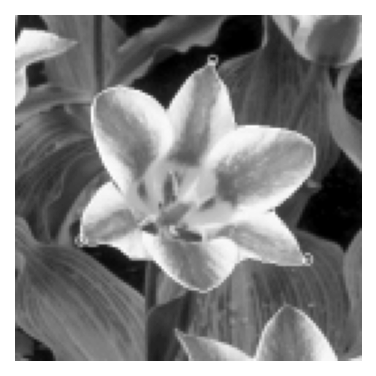}}
\end{center}
	\caption{Left: Original image. Right: Downscaled image.}
 \label{fig:image_u2}
\end{figure}

\begin{figure}[h]
\begin{center}
\subfigure[Rippa]{\includegraphics[width=0.35\textwidth]{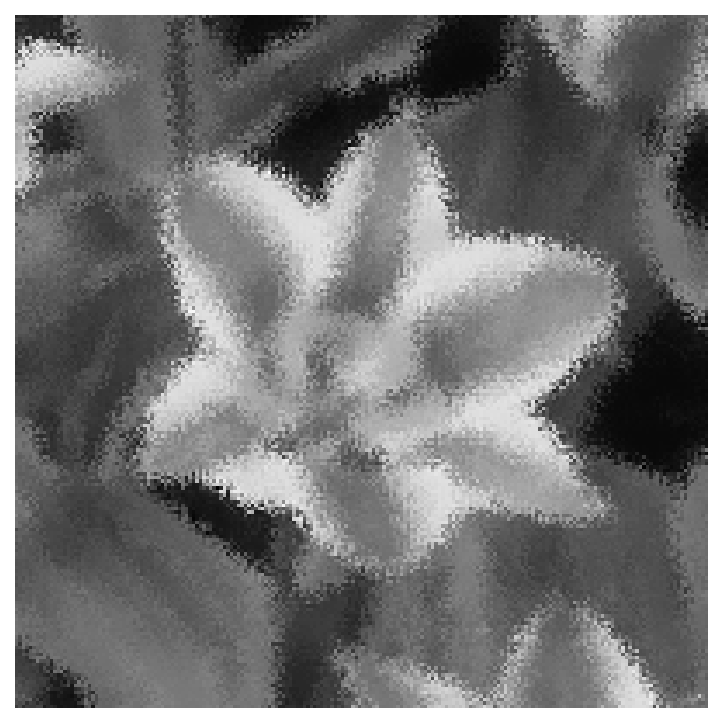}}
\subfigure[MLE]{\includegraphics[width=0.35\textwidth]{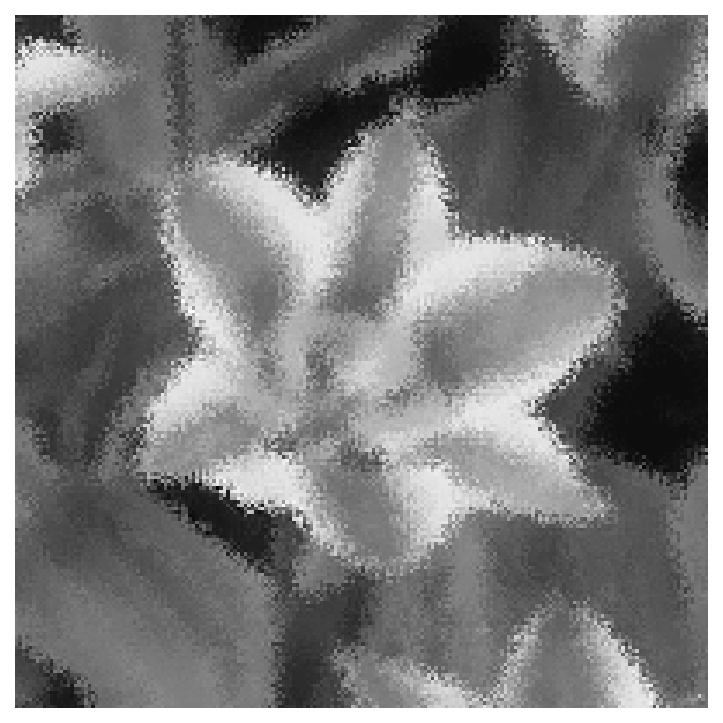}}
\subfigure[Hardy]{\includegraphics[width=0.35\textwidth]{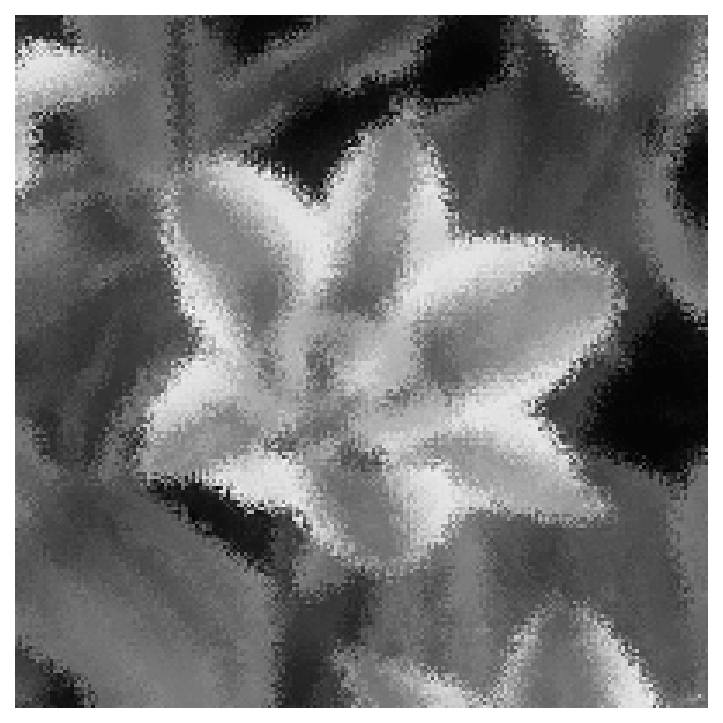}}
\subfigure[Franke]{\includegraphics[width=0.35\textwidth]{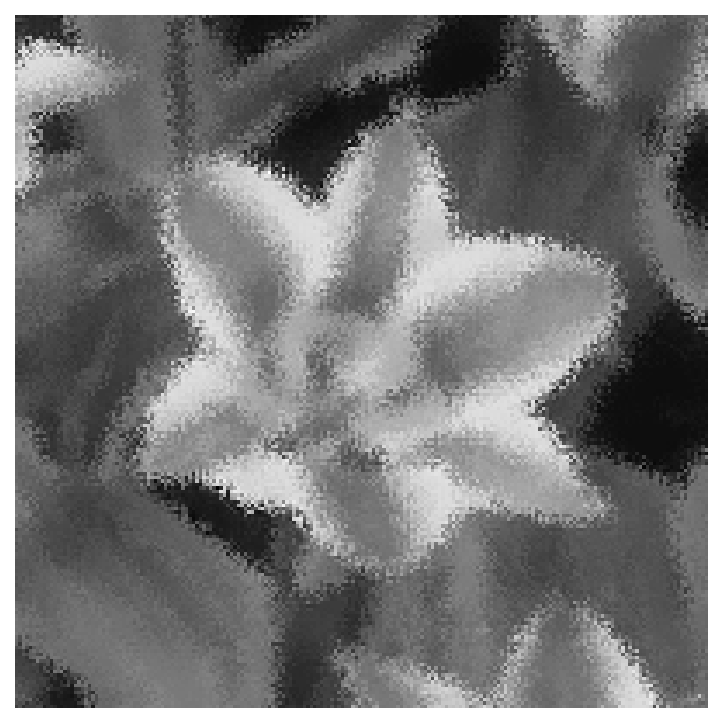}}
\subfigure[Modified Franke]{\includegraphics[width=0.35\textwidth]{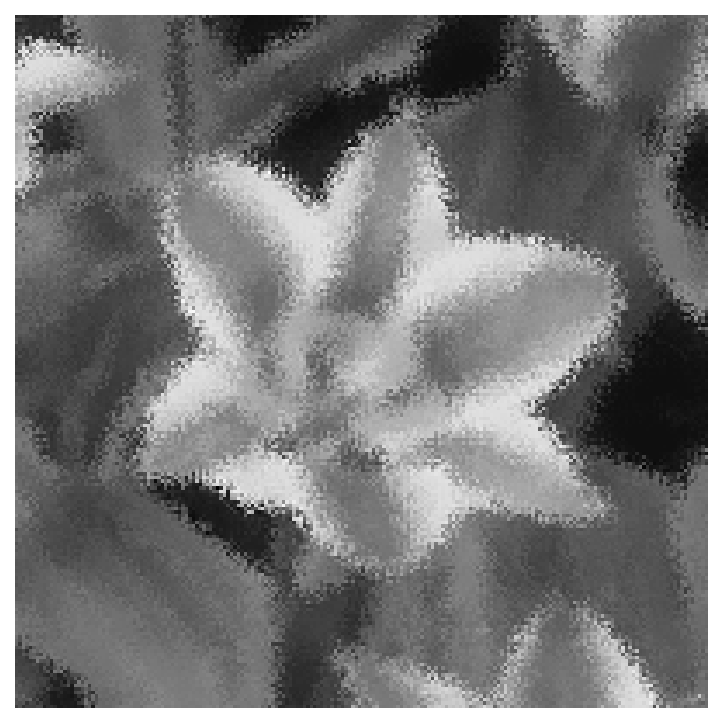}}
\subfigure[NN]{\includegraphics[width=0.35\textwidth]{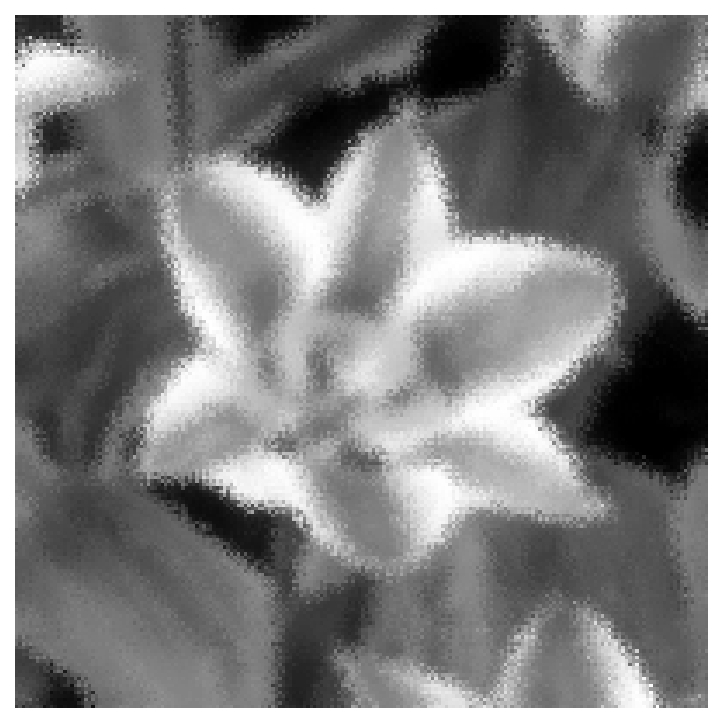}}
\end{center}
\caption{Upscaled images using different strategies.}
\label{fig:image2_upscaled}
\end{figure}

\begin{table}[h]
\centering
\caption{MSE and PSNR with diﬀerent strategies for image zoom-in for Figure \ref{fig:image_u2}.}
\begin{tabular}{|c| c c|} 
\hline
Scheme & MSE & PSNR\\
\hline
Rippa & 7.6708e-03 & 20.9098\\
MLE & 7.6722e-03 & 20.9090\\
Hardy & 7.47227-03 & 21.0234\\
Franke &  7.8612e-03& 20.8034\\
Modified Franke & 7.8544e-03& 20.8071 \\
NN &5.0149e-03 & 22.6253\\
\hline
\end{tabular}
\label{tab:mse_upscaled2}
\end{table}
Image distortion refers to the alteration of the geometry of an image. 
The transformation function $\mathcal{M}:\mathbb{R}^2 \to \mathbb{R}^2$ maps   coordinates $(x,y)$ in the original  image to coordinates $(x^{*},y^{*})$ in the
 distorted image, i.e.
 $(x^{*},y^{*})=\mathcal{M}(x,y)$.
 Here, we consider barrel distortion, which is a common type of distortion. 
 A typical choice is given by
 \begin{equation*}
  \begin{split}
   x^{*}&=x(1+kr^2),\\
   y^{*}&=y(1+kr^2).
  \end{split}   
 \end{equation*}
 where $r=\sqrt{x^2+y^2}$ and $k$ is the distortion coefficient.
 The goal is to correct the distorted image to match its original geometry by mapping the distorted coordinates back to their original undistorted positions.
We need  to correct how pixel positions have been distorted in both  $x$ and $y$ directions, therefore, we need two separate RBFs: one for the x-coordinates and one for the y-coordinates.
Finally, we apply the interpolators   to the original grid to get the corrected coordinates.

 We examine the same test images for this task. The distorted image for the first and second test case is represented in Figures \ref{fig:image_d1} and \ref{fig:image_d2}, respectively. 
 In Figures \ref{fig:image1_corrected} and \ref{fig:image2_corrected}, we depict the corrected images using all the presented methods
 for the first and second test case, respectively.
 The MSE and PSNR for these two images are reported in Tables \ref{tab:mse_corrected1} and \ref{tab:mse_corrected2}. As before, we proved the good performance of the adaptive NN strategy.

\begin{figure}[h]
\begin{center}
{\includegraphics[width=0.35\textwidth]{original1.png}}
{\includegraphics[width=0.35\textwidth]{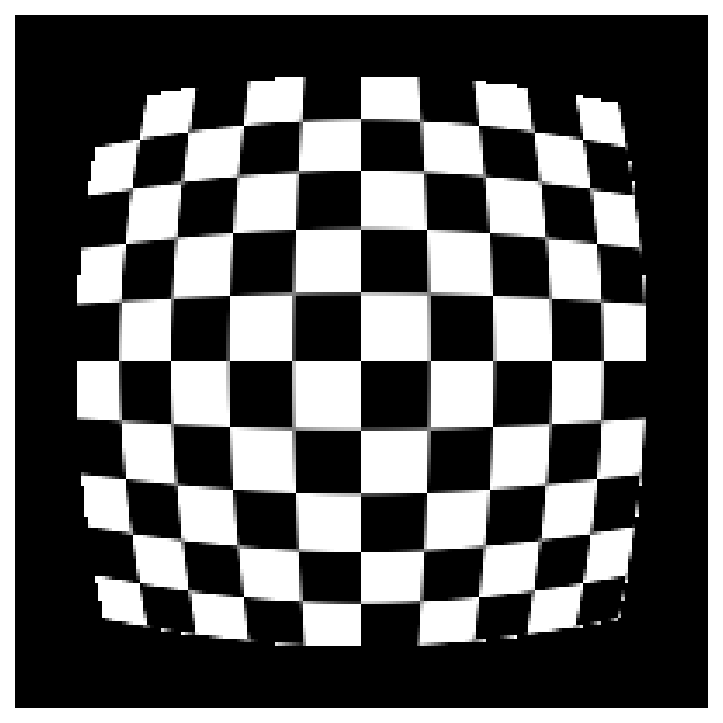}}
\end{center}
	\caption{Left: Original image. Right: Distorted image with $k=0.3$.}
 \label{fig:image_d1}
\end{figure}

\begin{figure}[h]
\begin{center}
\subfigure[Rippa]{\includegraphics[width=0.35\textwidth]{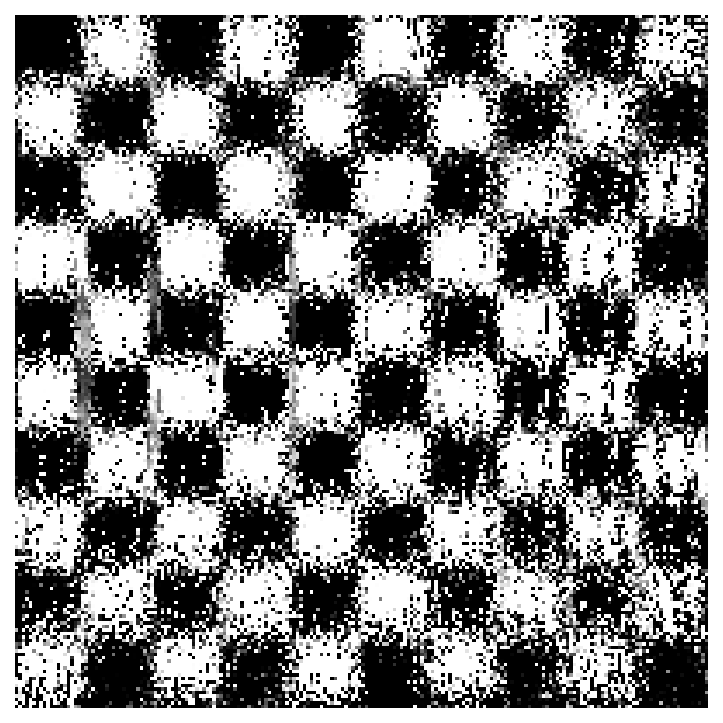}}
\subfigure[MLE]{\includegraphics[width=0.35\textwidth]{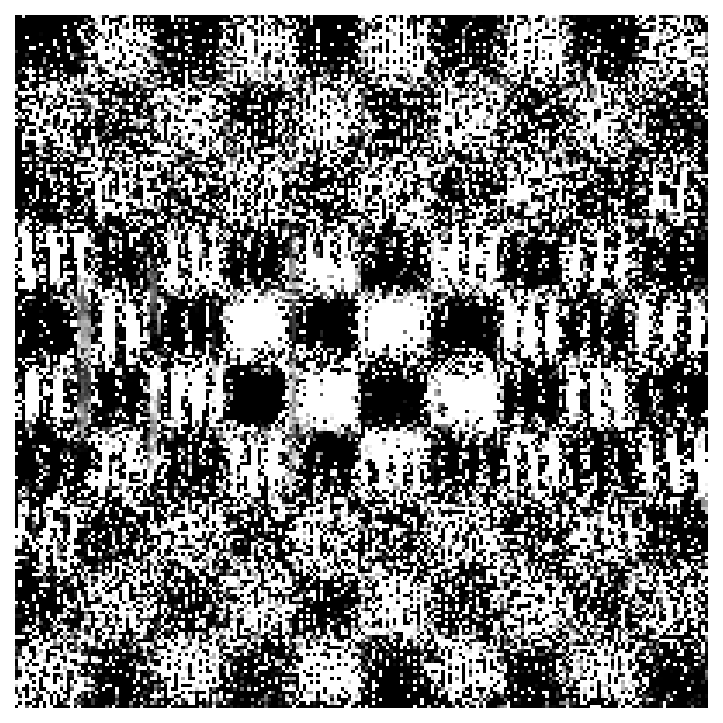}}
\subfigure[Hardy]{\includegraphics[width=0.35\textwidth]{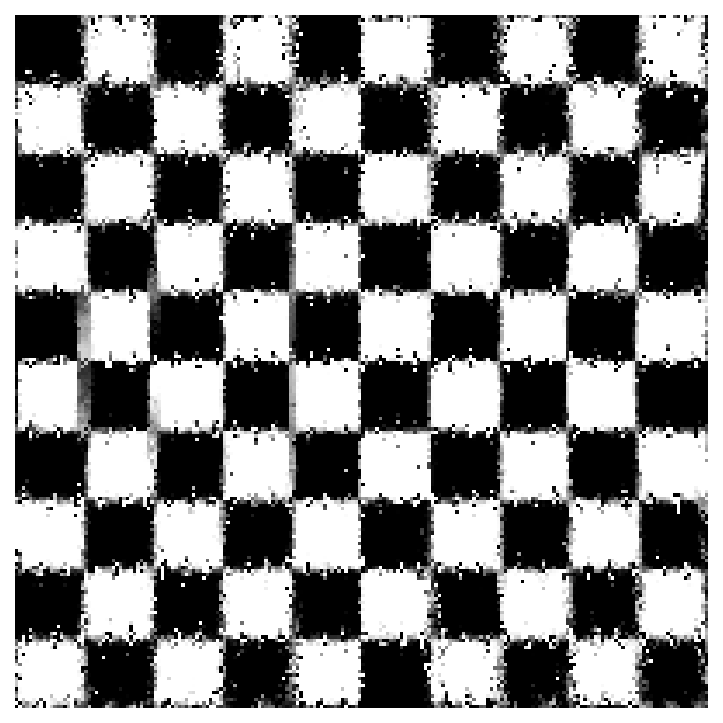}}
\subfigure[Franke]{\includegraphics[width=0.35\textwidth]{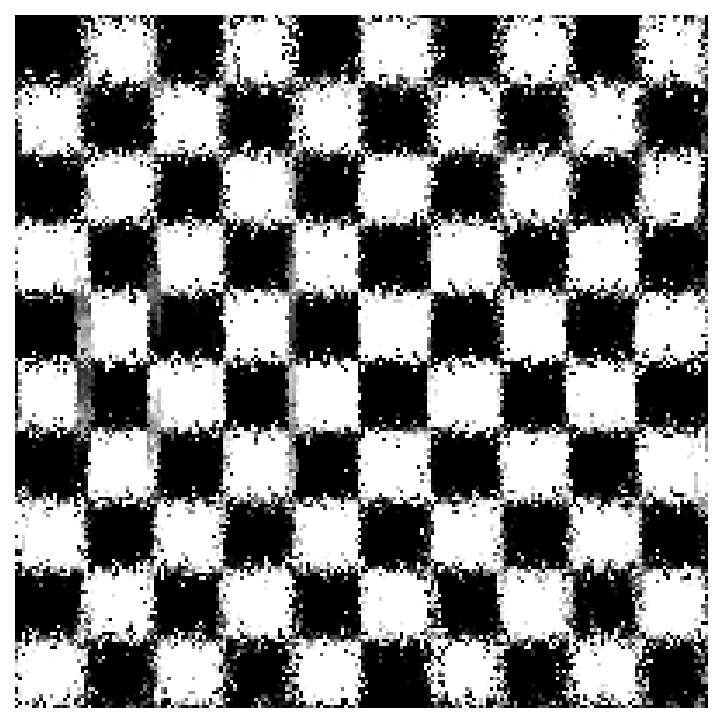}}
\subfigure[Modified Franke]{\includegraphics[width=0.35\textwidth]{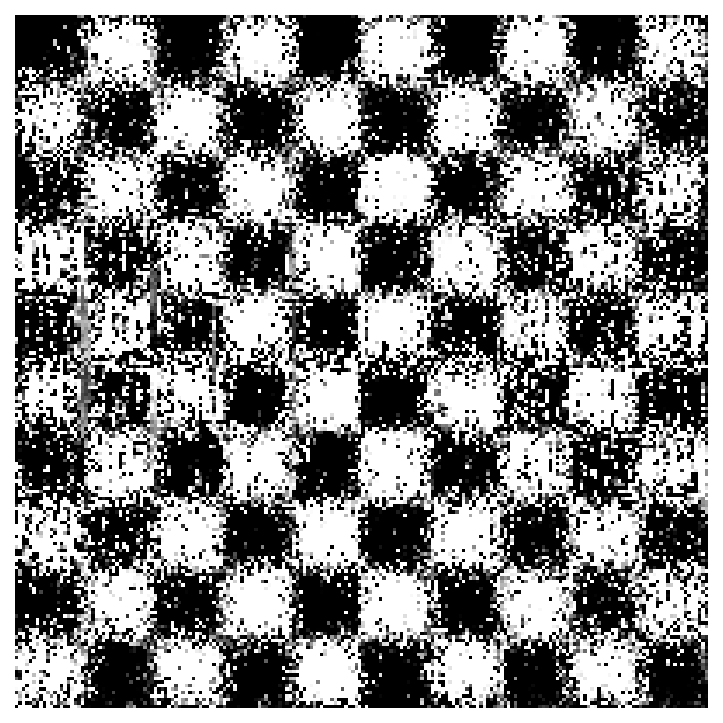}}
\subfigure[NN]{\includegraphics[width=0.35\textwidth]{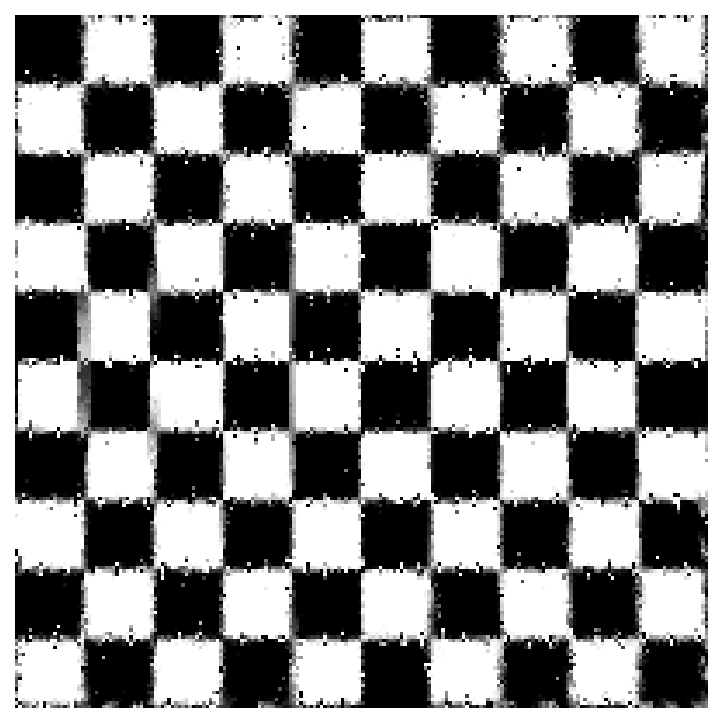}}
\end{center}
\caption{Corrected images using different strategies.}
\label{fig:image1_corrected}
\end{figure}

\begin{table}[h]
\centering
\caption{MSE and PSNR with diﬀerent strategies for image distortion correction for Figure \ref{fig:image_d1}.}
\begin{tabular}{|c| c c|} 
\hline
Scheme & MSE & PSNR\\
\hline
Rippa & 1.8277e-01&7.3809 \\
MLE & 3.0032e-01& 5.2241 \\
Hardy & 5.5457e-02 &12.5605 \\
Franke & 8.8646e-02  &10.5234 \\
Modified Franke &1.9115e-01 & 7.1863\\
NN & 3.8870e-02& 14.1038\\
\hline
\end{tabular}
\label{tab:mse_corrected1}
\end{table}

\begin{figure}[h]
\begin{center}
{\includegraphics[width=0.35\textwidth]{original2.png}}
{\includegraphics[width=0.35\textwidth]{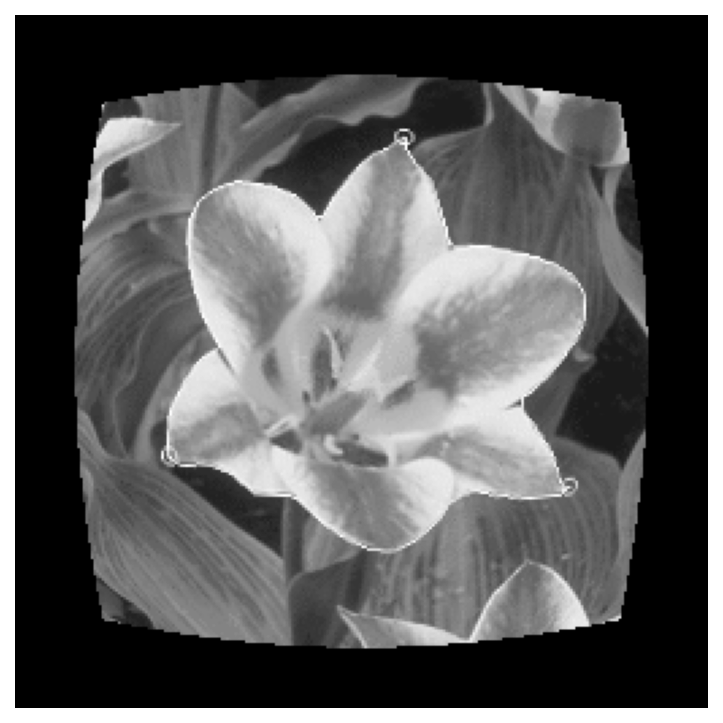}}
\end{center}
	\caption{Left: Original image. Right: Distorted image with $k=0.3$.}
 \label{fig:image_d2}
\end{figure}

\begin{figure}[h]
\begin{center}
\subfigure[Rippa]{\includegraphics[width=0.35\textwidth]{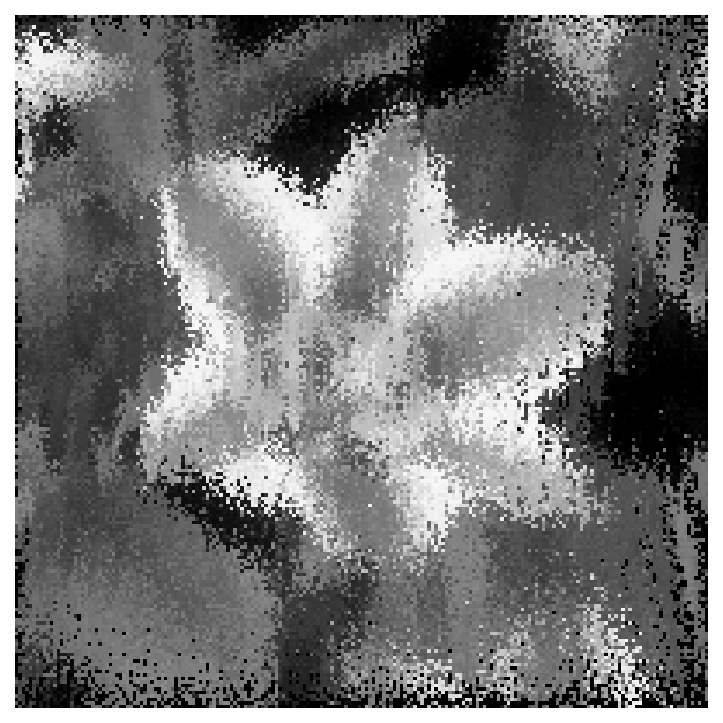}}
\subfigure[MLE]{\includegraphics[width=0.35\textwidth]{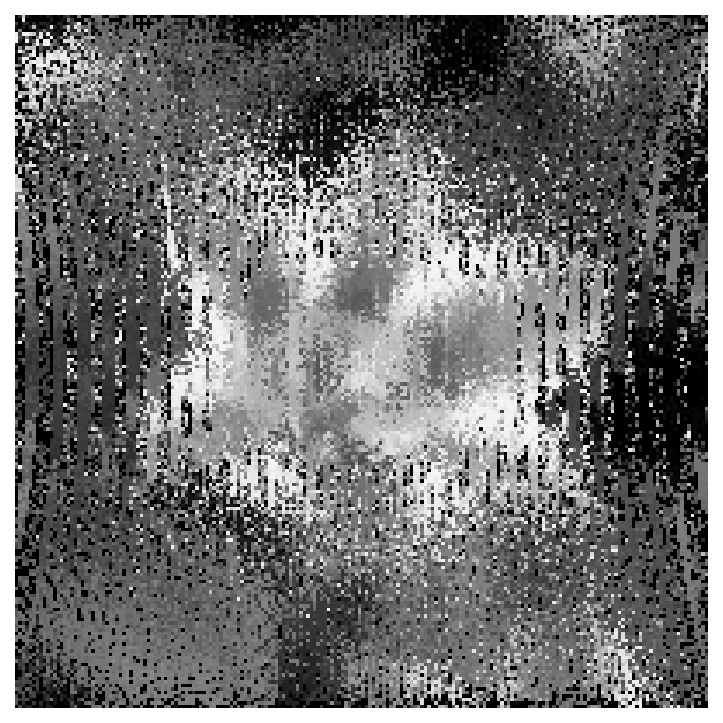}}
\subfigure[Hardy]{\includegraphics[width=0.35\textwidth]{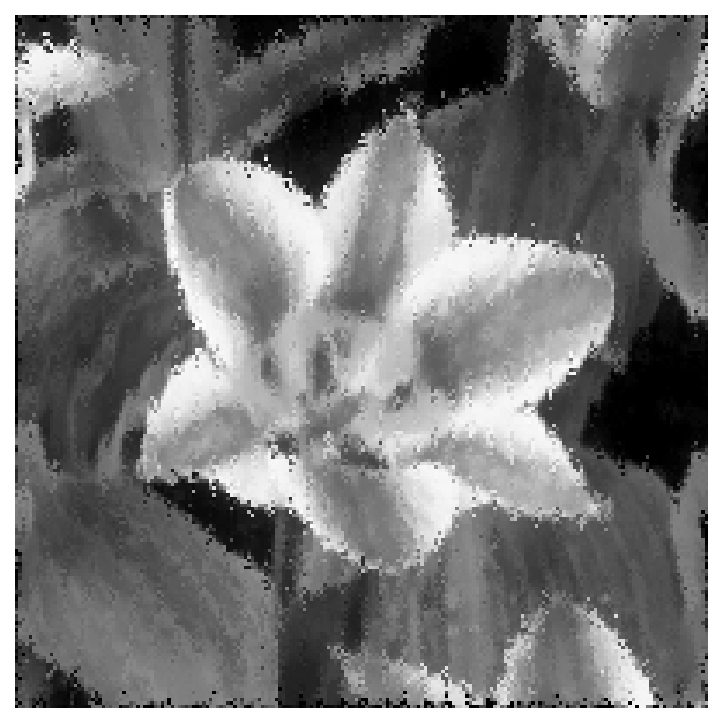}}
\subfigure[Franke]{\includegraphics[width=0.35\textwidth]{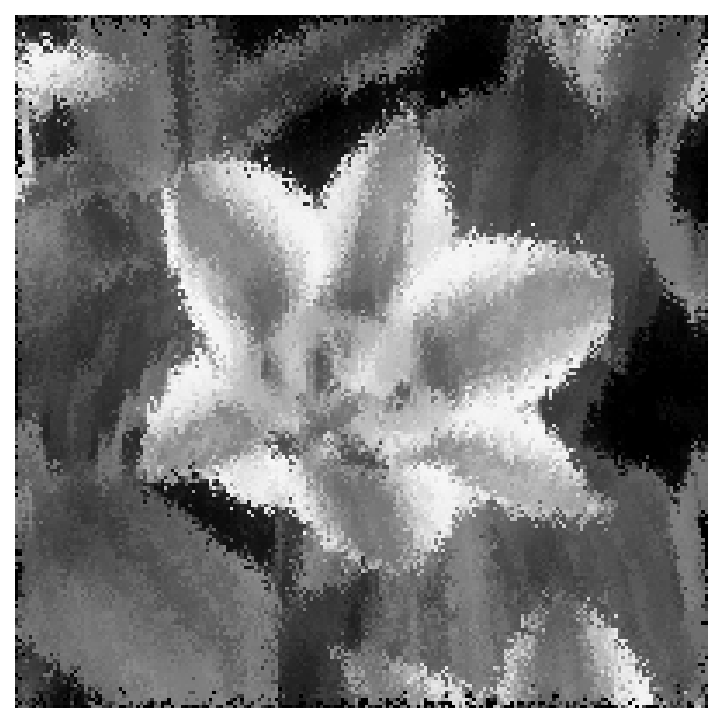}}
\subfigure[Modified Franke]{\includegraphics[width=0.35\textwidth]{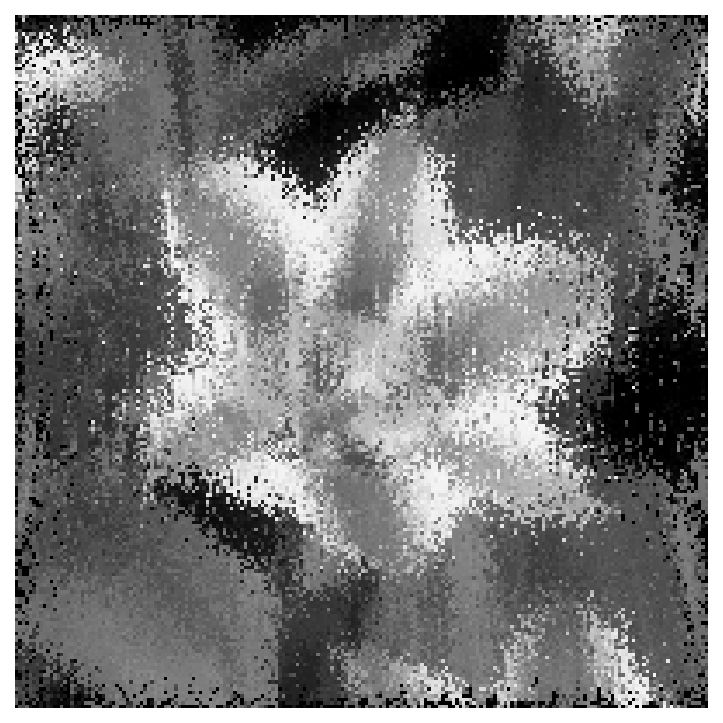}}
\subfigure[NN]{\includegraphics[width=0.35\textwidth]{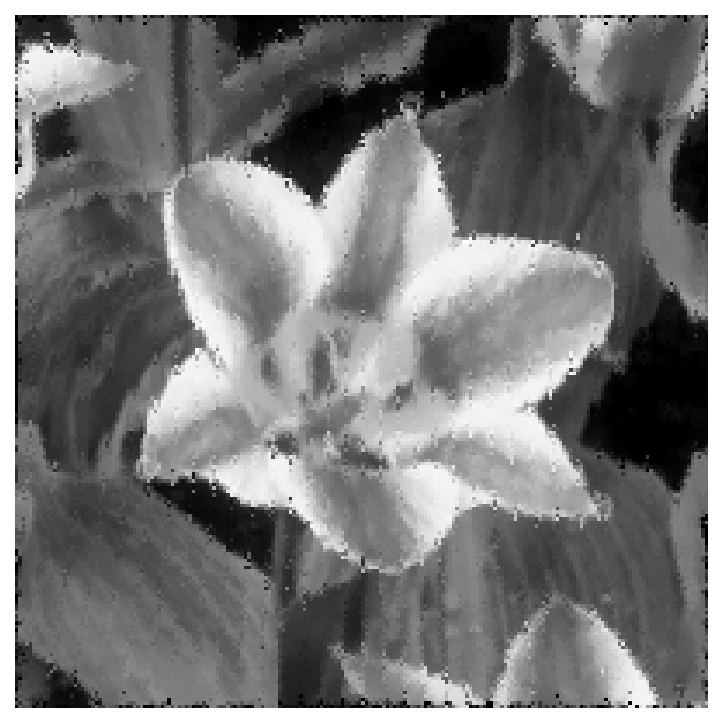}}
\end{center}
\caption{Corrected images using different strategies.}
\label{fig:image2_corrected}
\end{figure}

\begin{table}[h]
\centering
\caption{MSE and PSNR with different strategies for image distortion correction for Figure \ref{fig:image_d2}.}
\begin{tabular}{|c| c c|} 
\hline
Scheme & MSE & PSNR\\
\hline
Rippa & 2.2923e-02 & 16.1556\\
MLE & 4.8292e-02 & 12.9195\\
Hardy & 6.4938e-03 & 21.6332\\
Franke & 1.0203e-02 & 19.6710\\
Modified Franke &2.2342e-02 & 16.2670 \\
NN &4.2533e-03 & 23.4710\\
\hline
\end{tabular}
\label{tab:mse_corrected2}
\end{table}
\subsection{Solution of PDEs}
\label{sec:pdes}

We solve  time dependent and steady-state differential equations by the RBF-FD method, which is a meshless numerical technique used to solve PDEs generated from RBF interpolation over local sets of nodes on the surface \cite{Tolstykh2003,igor-rbf}.  In this section we show 1-dimensional and 2-dimensional problems. We set the oversampling parameter to $1$, thus considering a collocation method.

\subsubsection{1-dimensional heat equation}
\label{sec:time_1d}
In this section, we present the results of our method for numerically approximating a 1D heat equation with two initial conditions.
The results of RBF-FD method with the shape parameters derived from the adaptive NN strategy are compared   with the RBF-FD method with shape parameters obtained with Hardy, Franke and modified Franke strategies. We consider the collocation method. 
For all test problems presented in this article, 
the BDF2 \cite{Gear1967} for time stepping is implemented.

Let us consider the initial-boundary value problem
\begin{equation}
\label{p1}
\begin{split}
\frac{\partial u}{\partial t}&=\frac{\partial ^2 u}{\partial x^2}, ~x  \in [0,1],  ~0 \leq t\leq T,\\
u(x,0)&=-x^2+x,~x  \in (0,1) ,
\\
u(0,t)&=0,~ u(1,t)=0,~0\leq t\leq T.
\end{split}
\end{equation}
The exact solution of \eqref{p1} takes the following form
$$u(x,t)=\sum_{n=1}^{\infty} \left(\frac{-4}{n^3 \pi^3}\right)\left((-1)^n-1\right)\sin\left(n\pi x\right)\exp(-n^2 \pi ^2 t).$$
We set the time step to $\Delta t=0.001$ and the final time to $T=1$. The error between the analytical solution and the numerical solutions are shown in Figure \ref{fig:heat1d_test1_u} and Figure \ref{fig:heat1d_test1_nu} for equispaced and non-equispaced interpolation nodes, respectively. Our method has a clearly superior performance. On the contrary, the errors do not decrease in the other strategies.  Note that as $M$ increases, the decay of the error is slower and it seems to stagnate. This might be due to the fact that we restrict the interpolation matrix condition number to be bounded.

\begin{figure}[h]
\begin{center}
\subfigure[Equispaced]
{\includegraphics[width=0.45\textwidth]{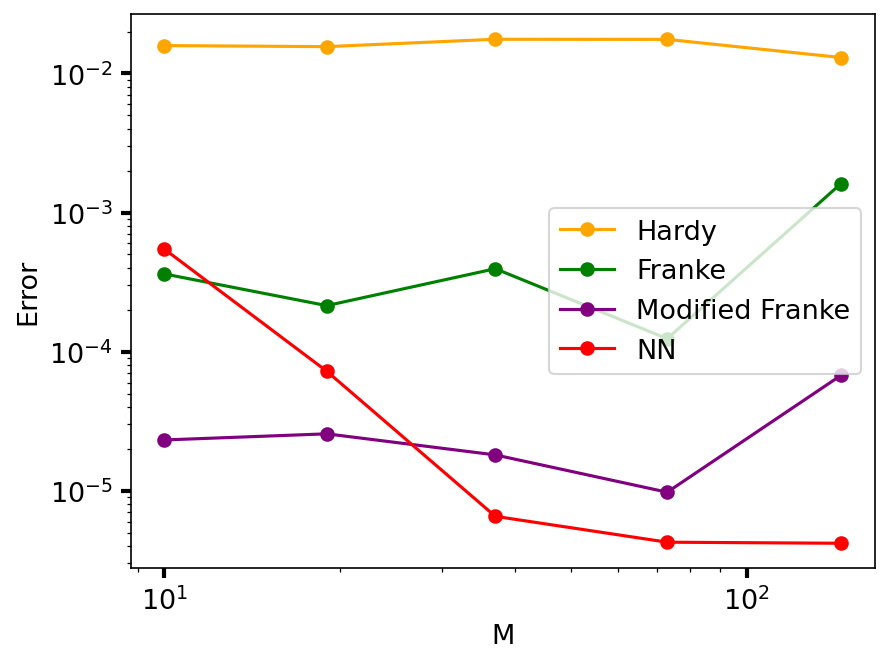}
 \label{fig:heat1d_test1_u}}
\subfigure[Non-equispaced]
{\includegraphics[width=0.45\textwidth]{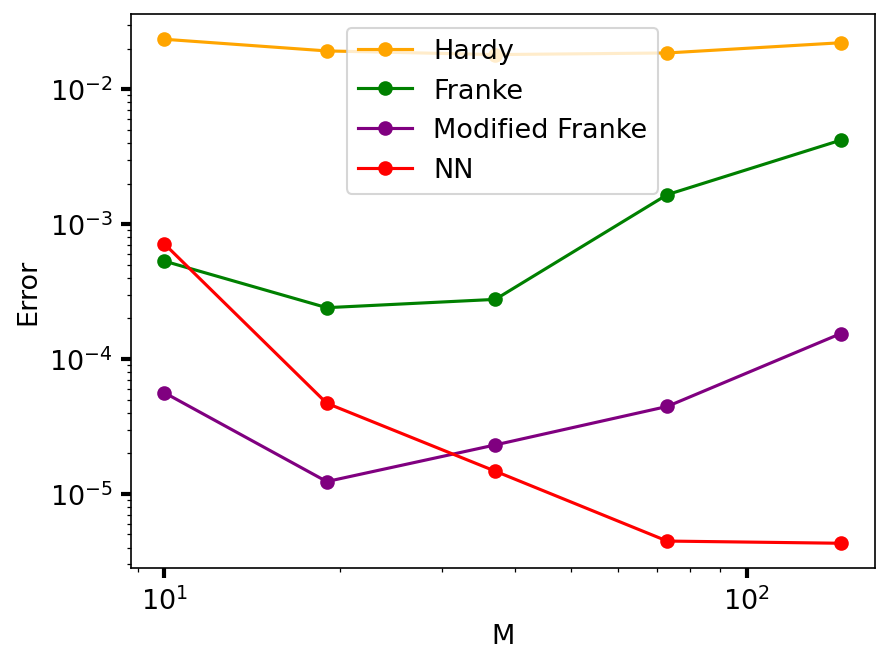}
 \label{fig:heat1d_test1_nu}}
\end{center}
	\caption{Plot of error convergence for \eqref{p1}.}
\end{figure}

Again we consider \eqref{p1}   with the different initial condition
\begin{equation}
\label{p4}
\begin{split}
\frac{\partial u}{\partial t}&=\frac{\partial ^2 u}{\partial x^2}, ~x  \in [0,1],  ~0 \leq t\leq T,\\
u(x,0)&=6\sin(\pi x),~x  \in (0,1),
\\
u(0,t)&=0,~ u(1,t)=0,~0\leq t\leq T.
\end{split}
\end{equation}
The exact solution can be written as follows
$$u(x,t)=6\sin(\pi x)\\exp(-\pi^2 t).$$
We again use  $\Delta t=0.001$ and  $T=1$. In Figure \ref{fig:heat1d_test2_u} and Figure \ref{fig:heat1d_test2_nu} the  errors are shown  for equispaced and non-equispaced points, respectively. It can be easily seen that only our approach converges to the exact solution. These figures show that the other methods are unable to appropriately approximate the solution of \eqref{p4}.

\begin{figure}[h]
\begin{center}
\subfigure[Equispaced]
{\includegraphics[width=0.45\textwidth]{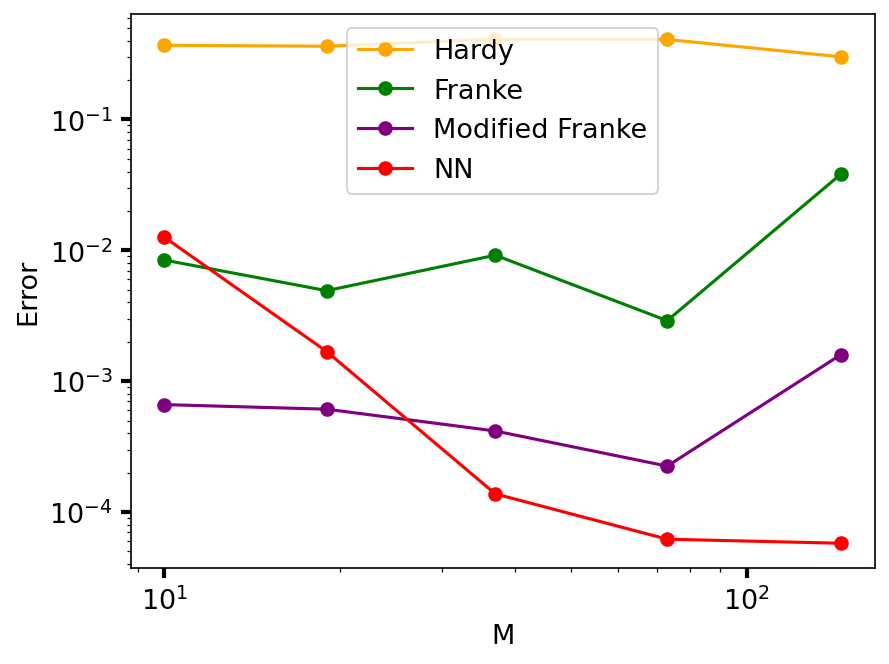}
\label{fig:heat1d_test2_u}}
\subfigure[Non-equispaced]
{\includegraphics[width=0.45\textwidth]{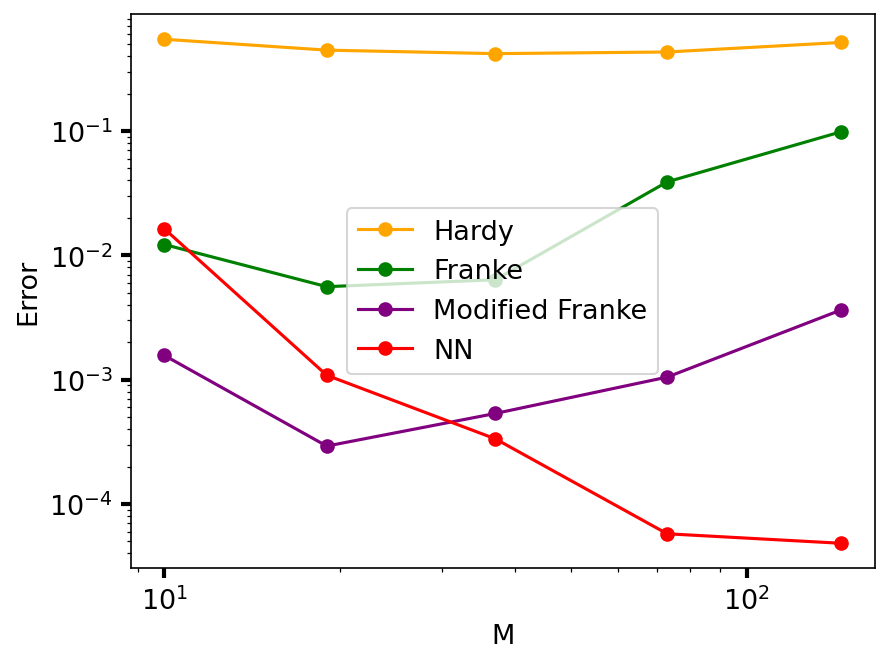}
\label{fig:heat1d_test2_nu}}
\end{center}
	\caption{Plot of error convergence for \eqref{p4}.}
 
\end{figure}

\subsubsection{2-dimensional boundary value problem}
\label{sec:bvp_2d}
Let us consider the 2-dimensional
linear elliptic boundary value problem
\begin{equation}
\label{poisson_eq_2d}
u_{xx}+u_{yy}= -4\pi^2\sin(2\pi x y)(x^2+y^2),~ (x,y) \in \Omega,
\end{equation}
with boundary data
\begin{equation}
\label{eq:poisson_eq_2d_bc}
u(x,y)= \sin(2\pi x y),~(x,y)  \in \partial \Omega.
\end{equation}
where the domain is a square $\Omega=[0,1]\times [0,1]$. The
exact solution of \eqref{poisson_eq_2d} with boundary data \eqref{eq:poisson_eq_2d_bc} is
$u(x,y)= \sin(2\pi x y)$.
The error between the numerical solution and the analytical solution for this 2D linear elliptic boundary value problem is displayed in Figure \ref{fig:poisson2d}. Again, we observe that the performance of the NN method superior to the performance of the other considered algorithms. In all other cases, the errors do not diminish (and some increase significantly) as the mesh is refined. We observe, however, that the NN method seems to stagnate in the finer meshes, and this can be due to the restriction on the matrix condition number.

\begin{figure}[h]
\begin{center}
\includegraphics[width=0.45\textwidth]{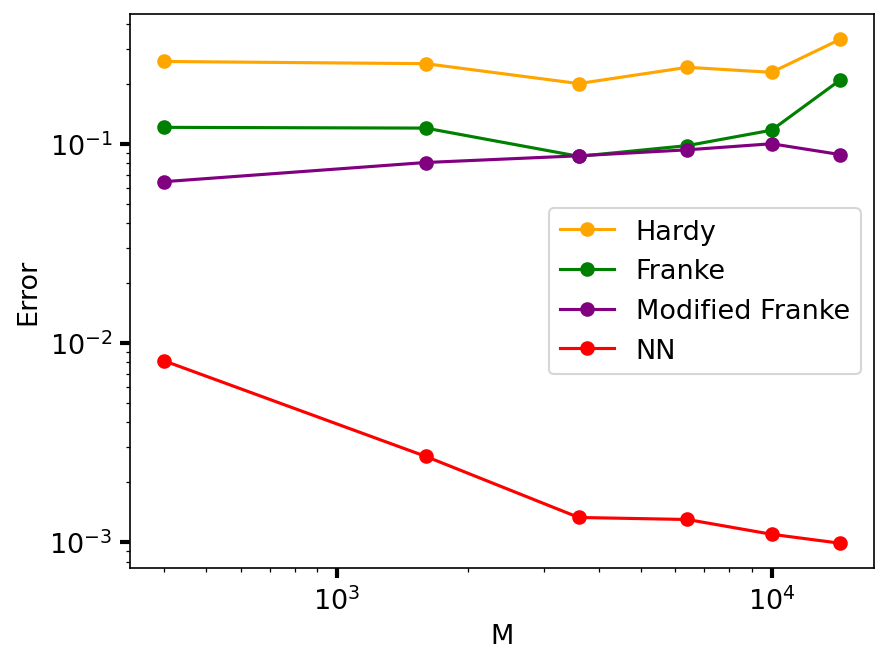}
\end{center}
	\caption{Convergence plot for \eqref{poisson_eq_2d}.}
 \label{fig:poisson2d}
\end{figure}

\subsubsection{2-dimensional heat equation}
\label{sec:time_2d}
We consider the 2D heat equation as follows:
\begin{equation}
\label{eq:heat2d}
  \frac{\partial u}{\partial t} = \mu\left(\frac{\partial ^2 u}{\partial x^2}+\frac{\partial ^2 u}{\partial y^2}\right),~t \in [0,T],
\end{equation}
with the initial condition
\begin{equation}
  u(x,y,0)=\sin(\pi x)\sin(\pi y),
\end{equation}
and the Dirichlet boundary conditions
\begin{equation}
  u(x,y,t)=0,
\end{equation}
Again, as in the previous tests in 2D, the domain is assumed to be a square $[0,1]\times [0,1]$. The exact solution of the above problem is
\begin{equation}
  u(x,y,t)=\sin(\pi x)\sin(\pi y) \exp(-\mu 2 \pi^2 t),~t \in [0,T].
\end{equation}
We assume a final time $T=0.5$, $\Delta t=0.005$ and $\mu =0.01$. Figure \ref{fig:heat2d} shows the convergence plot for 2D heat equation. The results of the adaptive NN for this test are qualitatively similar to the ones seen in the previous tests.
The Hardy, Franke and modified Franke methods produce constant errors, whereas the NN method produces solutions that converge as the number of interpolation points is augmented.

\begin{figure}[h]
\begin{center}
\includegraphics[width=0.45\textwidth]{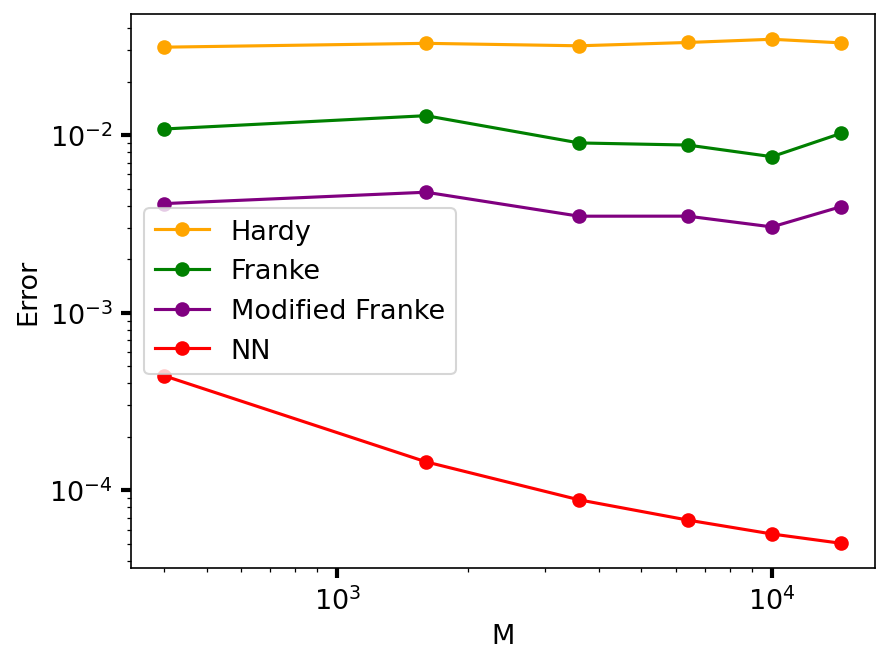}
\end{center}
	\caption{Convergence plot for the 2D heat problem.}
 \label{fig:heat2d}
\end{figure}

\subsubsection{Computational cost comparison}
We finish this section by presenting the computational cost of our method compared to Hardy's, Franke's, modified Franke's, Rippa's and MLE methods. Tables \ref{tab:timings_1d} and \ref{tab:timings_2d} show the results for one of the interpolation tasks, in 1-dimension and 2-dimensions, respectively. The simulation timings reported are from simulations run on the Ohio Supercomputer Center on AMD EPYC 7643 (Milan) CPU processors (Ascend Cluster). We present only the results for the interpolation of $f_1$ and $f_4$ as there is not much difference between the problems when the dimension and $N$ are fixed. The Hardy, Franke and modified Franke methods are much cheaper to evaluate, whereas the NN method requires the evaluation of the trained\footnote{In the reported timings, the time used to train the neural network was not included. In the same machine used to time the simulations, it took 1 hour and 14 minutes to train (without GPU acceleration). This is a one-time cost.} neural network for each stencil of size $N=10$. The Rippa and MLE algorithms are the most expensive methods as they requires the computation of the error vector $E$ over a set of candidate shape parameters $\varepsilon$. In the PDE case, the overhead added by the NN method is only at the beginning of the method, as we consider static meshes. In this case, the computational performance of the methods is very similar for sufficiently large number of interpolation points, as most of the computation is on the evolution of the PDE.

\begin{table}[h]
\begin{center}
\caption{Comparing computational efficiency between different methods in the 1-dimensional interpolation task (interpolating $f_1$). The results are reported in seconds. \label{tab:timings_1d}}
\begin{tabular}{|c | c c c c c c|} 
 \hline
M & Hardy & Franke & Modified Franke & Rippa & MLE & NN\\ [0.5ex] 
 \hline\hline
10 & 1.5e-05 & 9.8e-07 & 7.8e-07                    & 1.3e-01 & 1.3e-01 & 1.2e-03  \\ 
19 & 2.6e-05 & 2.0e-06 & 1.6e-06                    & 2.5e-01 & 2.5e-01 & 1.4e-03  \\ 
37 & 5.2e-05 & 4.0e-06 & 3.2e-06                    & 5.1e-01 & 5.1e-01 & 2.8e-03  \\ 
73 & 1.0e-04 & 8.4e-06 & 6.4e-06                    & 1.0e+00 & 1.0e+00 & 5.9e-03  \\ 
145 & 2.1e-04 & 1.7e-05 & 1.3e-05                    & 2.0e+00 & 2.0e+00 & 1.1e-02  \\ 
289 & 4.2e-04 & 3.2e-05 & 2.6e-05                    & 4.0e+00 & 4.0e+00 & 2.3e-02  \\ 
577 & 8.4e-04 & 6.2e-05 & 5.0e-05                    & 7.8e+00 & 7.8e+00 & 4.5e-02  \\ 
1153 & 1.7e-03 & 1.2e-04 & 1.0e-04                    & 1.6e+01 & 1.6e+01 & 8.9e-02  \\ 
2305 & 3.3e-03 & 2.5e-04 & 2.1e-04                    & 3.1e+01 & 3.1e+01 & 1.8e-01  \\ 
4609 & 6.8e-03 & 5.0e-04 & 4.1e-04                    & 6.3e+01 & 6.3e+01 & 3.6e-01  \\ 
 \hline
\end{tabular} 
\end{center}
\end{table}

\begin{table}[h]
\begin{center}
\caption{Comparing computational efficiency between different methods in the 2-dimensional interpolation task (interpolating $f_4$). The results are reported in seconds. \label{tab:timings_2d}}
\begin{tabular}{|c | c c c c c c|} 
 \hline
 M & Hardy & Franke & Modified Franke & Rippa & MLE & NN\\ [0.5ex] 
 \hline\hline
$20 \times 20$ & 2.4e-01 & 6.7e-01 & 6.7e-01                & 1.2e+02  & 1.2e+02 & 5.3e+00  \\ 
$40 \times 40$ & 9.4e-01 & 2.7e+00 & 2.7e+00                & 5.0e+02  & 5.0e+02 & 1.9e+01  \\ 
$60 \times 60$ & 2.2e+00 & 6.2e+00 & 6.1e+00                & 1.1e+03  & 1.1e+03 & 4.0e+01  \\ 
$80 \times 80$ & 3.8e+00 & 1.1e+01 & 1.1e+01                & 2.0e+03  & 2.0e+03 & 7.0e+01  \\ 
$100 \times 100$ & 5.7e+00 & 1.7e+01 & 1.7e+01                & 3.1e+03  & 3.1e+03 & 1.1e+02  \\ 
$120 \times 120$ & 8.3e+00 & 2.4e+01 & 2.4e+01                & 4.4e+03  & 4.4e+03 & 1.6e+02  \\ 
 \hline
\end{tabular} 
\end{center}
\end{table}

\subsection{Fallback procedure}
In this section, we evaluate numerically the impact of the fallback procedure introduced in Section \ref{sec:fallback}. In Figure \ref{fig:fallback-1d-uniform}, we show the error curves for the interpolation of function $f_1$\footnote{The error curve for function $f_2$ is very similar to the one of $f_1$ and thus is omitted.} and $f_3$ when considering different thresholds, namely $\theta = 12, 16, \infty$. This means that if the NN's prediction leads to a shape parameter larger than $\theta$, we correct the shape parameter by solving optimization problem as in Section \ref{sec:optimization}. We note that after the third point, the fallback procedure is activated when the threshold $\theta=12$. However, when we increase the threshold to $\theta=16$, the fallback procedure is no longer activated, meaning that the NN prediction produces shape parameters that lead to interpolation matrices with condition numbers smaller than $10^{16}$. We also note that the error is degraded when we keep the condition number of the interpolation matrix to be smaller than $11.5$ as the interpolation points get closer together, this is expected as the condition number of the interpolation matrix tends to increase as the points get closer to each other.

To try to answer the question of why is the network producing shape parameters that lead to condition numbers outside of the expected range $[10^{11},10^{11.5}]$, we hypothesize the flagged data points are perhaps out-of-distribution. We plot the distribution of the average of the distance matrix for each input $\vec{x}$ that requires correction, against the average of the distance matrix in the training set, which can be seen in Figure \ref{fig:fallback-1d-uniform} (right) and postulate that these points might be far from the training set points. In Figure \ref{fig:fallback-1d-not-uniform}, we show the same figure for non-equispaced axis and reach similar conclusions.

\begin{figure}[h]
\begin{center}
\includegraphics[width=0.33\textwidth]{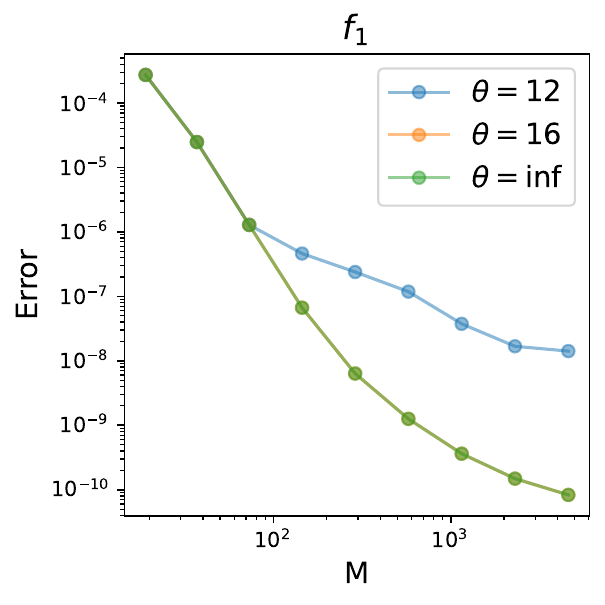}\includegraphics[width=0.33\textwidth]{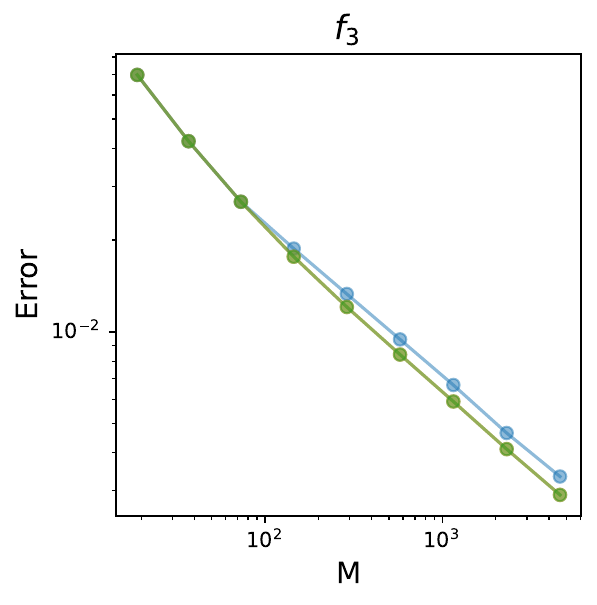}\includegraphics[width=0.31\textwidth]{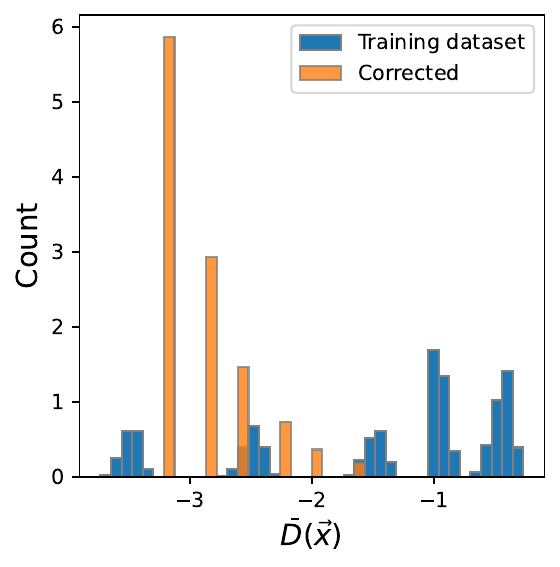}
\end{center}
	\caption{Fallback performance for equispaced interpolation nodes in 1d interpolation tasks. Left: Error convergence plots for functions $f_1$ and $f_3$ subject to different thresholds $\theta$. Right: Distribution of the average distance between the corrected points and the training dataset.}
 \label{fig:fallback-1d-uniform}
\end{figure}

\begin{figure}[h]
\begin{center}
\includegraphics[width=0.33\textwidth]{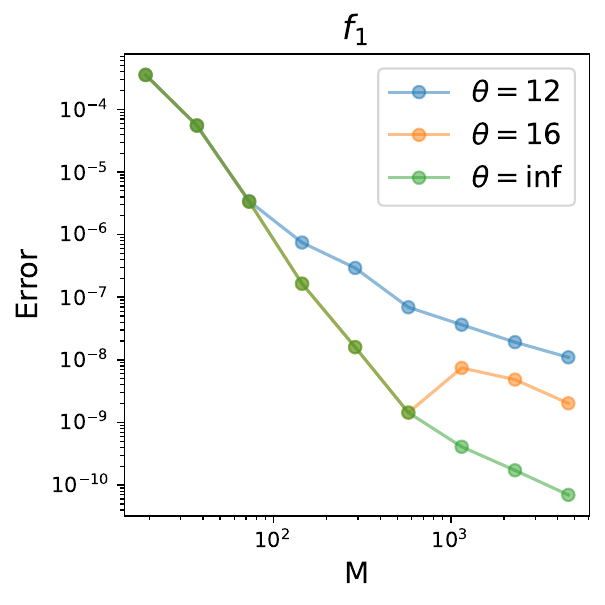}\includegraphics[width=0.33\textwidth]{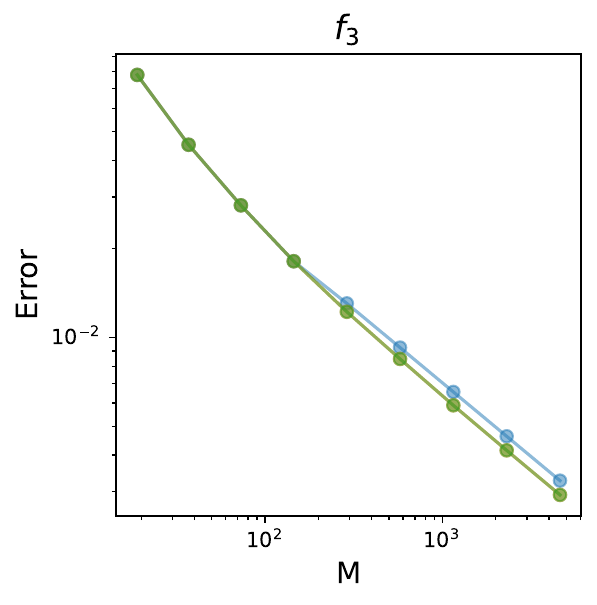}\includegraphics[width=0.33\textwidth]{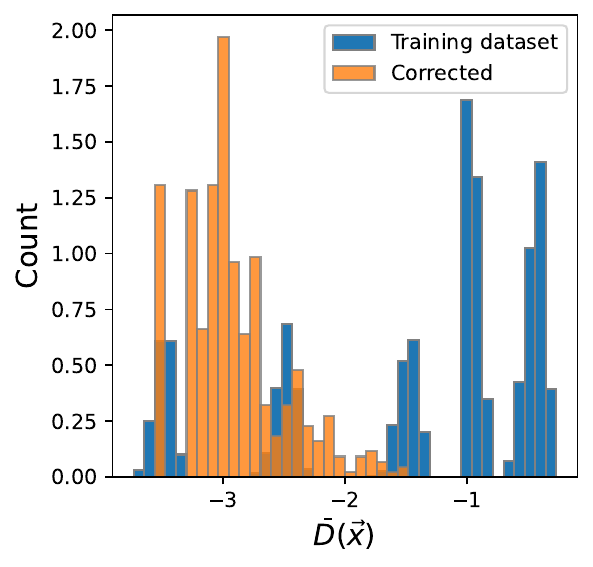}
\end{center}
	\caption{Fallback performance for non-equispaced interpolation nodes in 1d interpolation tasks. Left: Error convergence plots for functions $f_1$ and $f_3$ subject to different thresholds $\theta$. Right: Distribution of the average distance between the corrected points and the training dataset.}
 \label{fig:fallback-1d-not-uniform}
\end{figure}

With this fallback procedure, we have a strong guarantee that the generated interpolation matrix has a controlled condition number which can be useful because it places less importance on the creation of the training dataset. For general tasks, we propose to use $\theta =\infty$, which means that the fallback procedure is only activated when the generated interpolation matrix is truly ill-conditioned. The downside to this strategy is that there is significant overhead if the optimization is required. The performance depends heavily on the initial guess for the optimization procedure. We currently set the initial guess $\varepsilon_{init}$ to be the output of the NN to avoid fine tuning the optimization procedure, however, the most efficient results we found for the more refined meshes (large $M$), was to set $\varepsilon$ to be high, for example, $\varepsilon_{init}=400$.

\subsection{Retraining on unseen data}

The three 1-dimensional interpolation problems $f_i,~  i=1,2,3$, are examined to show the numerical influence of the retraining strategy (Section \ref{sec:retraining}). 

In Figures \ref{fig:retrain_u} and \ref{fig:retrain_nu}, we have displayed the performance of the trained NN, before and after retraining, namely, naive retraining and model merging strategies,  when considering equispaced and non-equispaced interpolation nodes, respectively. 
In the left picture of these figures, we show the error curves before and after retraining strategy. We can see that the results of before retraining strategy are more precise than after retraining strategy,  despite very similar results for $f_3$.
The cause for this is clear,
as the retraining strategy ensures that the logarithm of the condition number of interpolation matrices are less than the threshold $\theta=14$.
Also, the right pictures show  the distribution of the logarithm of the condition number of the interpolation matrices. 
We can see an overall better performance when using model merging strategy for $f_1$.

\begin{figure}[h]
\begin{center}
{\includegraphics[width=0.45\textwidth]{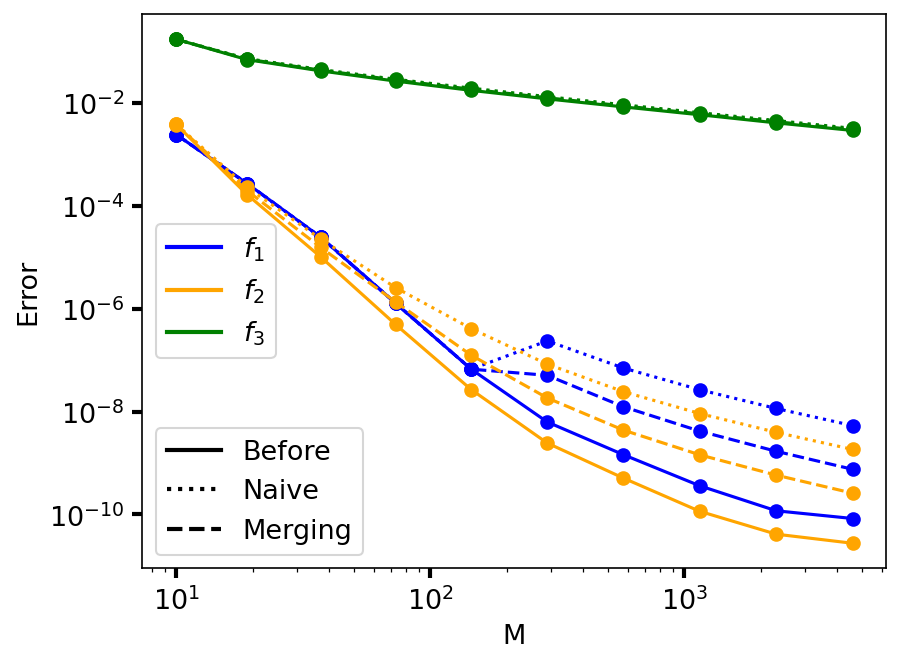}}
{\includegraphics[width=0.45\textwidth]{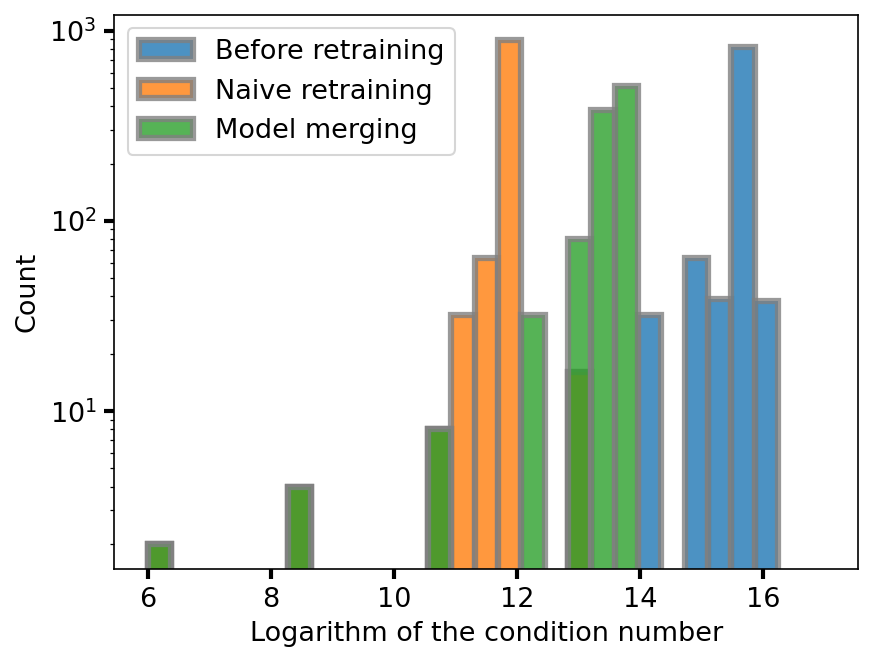}}
\end{center}
	\caption{Retraining performance using naive retraining and model merging strategies for equispaced interpolation nodes in 1d interpolation tasks. Left: Error convergence plots for functions $f_1$, $f_2$ and $f_3$ before retraining and after retraining using  naive retraining and model merging strategies. Right: Distribution of the logarithm of the condition number.}
	\label{fig:retrain_u}
\end{figure}

\begin{figure}[h]
\begin{center}
{\includegraphics[width=0.45\textwidth]{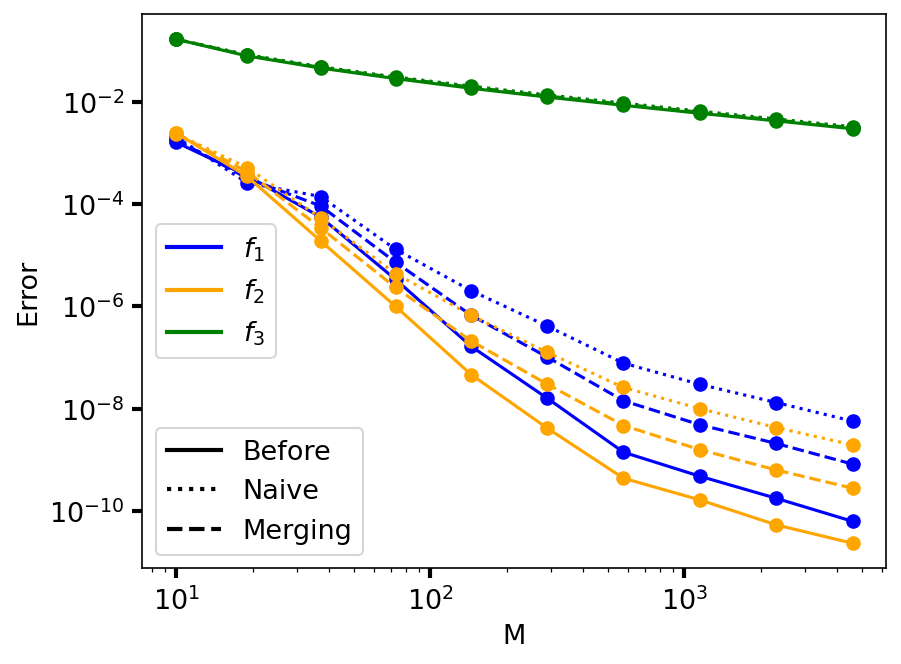}}
{\includegraphics[width=0.45\textwidth]{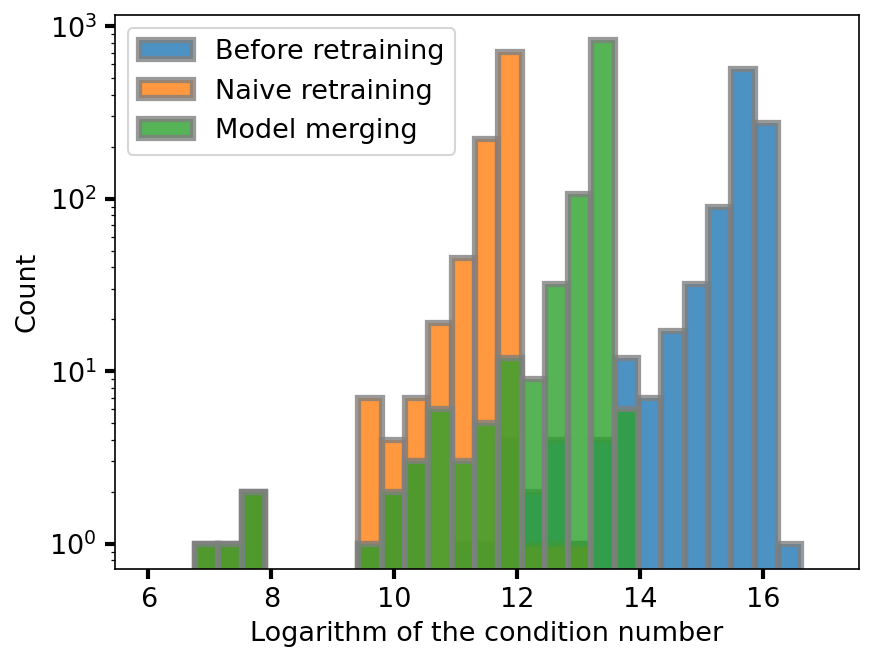}}
\end{center}
	\caption{Retraining performance using naive retraining and model merging strategies for non-equispaced interpolation nodes in 1d interpolation tasks. Left: Error convergence plots for functions $f_1$, $f_2$ and $f_3$ before retraining and after retraining using  naive retraining and model merging strategies. Right: Distribution of the logarithm of the condition number.}
	\label{fig:retrain_nu}
\end{figure}

\section{Conclusion}
\label{sec:conclusion}
In this work, we propose a novel approach to predict a shape parameter for the inverse multiquadric RBFs. The method is based on two parts: i) we derive an optimization problem that attains a suitable shape parameter for any distribution of points $\vec{x}$, ii) we use the optimization problem to generate a dataset and train a NN to predict a suitable shape parameter given any distribution of interpolation points. We focus on a setting where the size of the stencil is fixed and provide a unified strategy for 1-dimensional and 2-dimensional clouds of points. Furthermore, we propose a fallback procedure that guarantees that the generated interpolation matrices (using the predicted shape parameter) remain well-conditioned.
The proposed approach is tested on interpolation tasks and integrated with a RBF-FD method and shows promising results in comparison to other adaptive strategies. Although we are able to strictly guarantee the well-posedness of the generated interpolation matrices, there is a computational overhead when further optimization is necessary. To overcome this, we introduce a retraining strategy that improves the learned model. We were able to show that the retraining strategy does control the condition number of the generated interpolation matrices.

While all our results focused on inverse multiquadratic RBFs, the adaptation of these proposed methods to other positive definite RBFs is straightforward. The optimization problem changes only with respect to the changed interpolation matrix, thus, the data generation, as presented in Algorithm \ref{alg:data_generation_1d}, can be re-used to generate a new dataset. The caveat is that for each different RBF basis and stencil size $N$, a new neural network has to be trained.

There are several future directions for this work: one, we would like to explore is the extension of this method to clouds of points embedded in higher dimensions. Secondly, with the recent advances in the theory of RBFs in the numerical solution of conservation laws (see \cite{glaublitz2023,glaublitz2024}), the integration of our novel method with state-of-the-art energy-stable RBF methods could lead to powerful novel meshless methods for conservation laws.

\section*{Acknowledgments}
This work was supported in part by an allocation of computing time from the Ohio
Supercomputer Center. We wish to thank the anonymous reviewers for their input and feedback. 
\appendix

\section{Computational implementations}
\label{app:dataset}

\subsection{Dataset generation}
The generated training set contains the pairs $(\vec{x},\varepsilon)$, where $\vec{x}$ denotes a set of interpolation nodes and $\varepsilon$ the corresponding shape parameter, that leads to an interpolation matrix with controlled condition number. We randomly sample sets of interpolation nodes $\vec{x}$ and use the optimization procedure described in Section \ref{sec:optimization} to obtain a suitable $\varepsilon$. Algorithm \ref{alg:data_generation_1d} describes the dataset generation: given a domain $I$ (for simplicity, we assume it is $1$-dimensional interval $[a,b]$ up to a $n$-dimensional interval $[a,b]^n$ hypercube), we uniformly sample $N_I$ points, creating a sorted list of points $X=(x_1,...,x_{N_I})$. Then, iterating through the list $X$ $R$ times: we start with $x_1$, find its  $N-1$ nearest neighbors and form $\vec{x}=(x_1,x_1^{(1)},...,x_1^{(N-1)})$. Using $\vec{x}$, we solve \eqref{eq:varepsilon-optimization} to find the corresponding $\varepsilon$. This is repeated $D$ number of times. Then, we obtain a dataset of size\footnote{It is possible that the optimization does not converge, leading to a smaller number of training points. This did not occur in our training.} $R\times D$ with pairs $(\vec{x},\varepsilon)$ that will be used as training set for supervised learning algorithm.

\begin{algorithm}[h]
\label{alg:data_generation_1d}
\textbf{Input:} Domain $I$, the number of points in the domain $N_I$, size of the interpolation node set $N$, number of repeats $D$, use ratio $R$\\
\textbf{Output:} a set of size $R\times D$, where each element is the training data pair $(\vec{x},\varepsilon)$, with $\vec{x}=(x_1,\cdots,x_N)$\\
\SetAlgoLined
\For{$\text{k}\leftarrow 1$ to $D$}{
- Sample $N_I$ points uniformly from domain $I$, generating sample $X = (x_1,...,x_{N_I})$\\
- Sort $X$\\
\For{$i\leftarrow 1$ to $R$}{
- Find the $N-1$ nearest neighbors of $x_i$, and form $\vec{x} = (x_i, x_i^{(1)},x_i^{(2)},...,x_i^{(N-1)})$ where $x_{i}^{(k)}$ denotes the $k$-th nearest neighbour to $x_i$\\
- Initialize the shape parameter with $\varepsilon=\frac{1}{0.815 N \sum\limits_{i=1}^{N} d_i}$, where $d_i$ is the distance of  $x_i$ from the closest point to it\\
- Set $\text{trials}=0$\\
- Set $\text{unsuccessful}=0$
Compute loss using \eqref{eq:loss}\\
\While{$\text{loss} > 10^{-3}$}{
- Solve \eqref{eq:varepsilon-optimization} using Adam stochastic optimizer
\\
- trials = trials + 1\\
\If{$\text{trials} > 40$}
{- unsuccessful = 1\\
- break}
}
\If {$\text{unsuccessful} = 0$}
{Add \vec{x} with corresponding shape parameter $\varepsilon$ to the dataset}
\caption{Dataset generation}
}
}
\end{algorithm}

The 1-dimensional dataset generated is described in Table \ref{table:dataset_1d}. We consider the domains $I=[0,0.01],[0,0.1],[0,1]$. We set $N_I = 10$, $R=1$ and $D=700$. We additionally used three sample functions to evaluate the interpolation error from the generated interpolation matrix as a consistency check -- the interpolation error does not influence the point generation. Thus, we generate $3\times |I|\times R \times D = 6300$ training points, where $|I|$ denotes the number of considered intervals.   
The 2-dimensional dataset is described in Table \ref{table:dataset2d}. We consider the domains $I=[0,0.001]^2,[0,0.01]^2,[0,0.1]^2,[0,1]^2$. We set $N_I = 20$, $R=20$ and $D=50$. We additionally used three sample functions to evaluate the interpolation error from the generated interpolation matrix as a consistency check -- the interpolation error does not influence the point generation. Thus, we generate $3\times |I|\times R \times D = 12000$ training points.
\begin{table}[h]
  \centering
  \caption{One-dimension dataset, with use ratio $R=1$} 
\begin{tabular}{|c|c |c | c|} 
\hline
$u(x)$ & Domain $I$ & $N_I $& $D$\\
\hline
 & $[0,0.01]$ &10&700\\
$\exp(\sin(\pi x))$ & $[0,0.1]$ & 10 & 700 \\
& $[0,1]$ & 10 & 700\\
\hline
 & $[0,0.01] $&  10 & 700\\
$\frac{1}{16x^2+1}$&$[0,0.1]$&10&700\\
& $[0,1]$ & 10 & 700  \\
\hline
$\cos(200\pi x)$ &$[0,0.01]$&10& 700\\
$\cos(20\pi x)$&$[0,0.1]$&10&700\\
$\cos(2\pi x)$&$[0,1]$&10&700\\
\hline
\end{tabular}
	\label{table:dataset_1d}
\end{table}

\begin{table}[h]
\scriptsize
  \centering
  \caption{Two-dimension dataset, with use ratio $R=N_I$. 
	\label{table:dataset2d}} 
\begin{tabular}{|c|c|c|c|} 
\hline
$u(x,y)$&Domain& $N_I$ & $D$\\
\hline
 & $[0,0.001]^2$ &20&50\\
$(1+ \exp(-\frac{1}{\kappa})-\exp(-\frac{x}{\kappa})-\exp(\frac{x-1}{\kappa}))(1+ \exp(-\frac{1}{\kappa})-\exp(-\frac{y}{\kappa})-\exp(\frac{y-1}{\kappa}))$&$[0,0.01]^2$&20&50\\
$\kappa=0.1$ & $[0,0.1]^2$ &20&50\\
& $[0,1]^2$&20&50\\
\hline
 & $[0,0.001]^2$&20&50\\
$(1+ \exp(-\frac{1}{\kappa})-\exp(-\frac{x}{\kappa})-\exp(\frac{x-1}{\kappa}))(1+ \exp(-\frac{1}{\kappa})-\exp(-\frac{y}{\kappa})-\exp(\frac{y-1}{\kappa}))$ & $[0,0.01]^2$&20&50\\
 $\kappa=1$ & $[0,0.1]^2$&20&50\\
& $[0,1]^2$&20&50\\
\hline
 & $[0,0.001]^2$&20&50\\
$\frac{3}{4} \exp\big(-( \frac{(9x-2)^2+(9y-2)^2}{4})\big)+\frac{3}{4} \exp\big(-( \frac{(9x+1)^2}{49}+\frac{(9y+1)^2}{10})\big)$   & $[0,0.01]^2$ &   20 &50 \\
$+\frac{1}{2} \exp\big(-( \frac{(9x-7)^2+(9y-3)^2}{4})\big)
 -\frac{1}{5} \exp\big(-( (9x-4)^2+(9y-7)^2)\big)$ & $[0,0.1]^2$&20&50\\
& $[0,1]^2$&20&50\\
\hline
\end{tabular}
\end{table}

\subsection{Fall-back scheme}
Algorithm \ref{alg:fallback} describes the fallback procedure, to guarantee that the interpolation matrix generated has a bounded condition number, smaller than the user-defined threshold $\theta$.

\begin{algorithm}[h]
\caption{Shape prediction with fallback 
\label{alg:fallback}} 
\textbf{Input:} Set of interpolation notes $\vec{x}\in\mathbb{R}^{N\times n}$, threshold $\theta$ \\
\textbf{Output:} $\varepsilon$ \\
\SetAlgoLined
 \vec{d} $\leftarrow$ generate features from $\vec{x}$ \\
 $\varepsilon \leftarrow \mathcal{F}(\vec{d})$\\
\If{$ \mbox{cond}(A(\vec{x},\varepsilon)) > \theta$}
{ $\varepsilon \leftarrow$ solve minimization as in \eqref{eq:varepsilon-optimization} 
}
\end{algorithm}

\revisionthree{
\subsection{Reproducibility}
The numerical experiments presented in Section \ref{sec:numerical} and in \ref{ap:rippa_mle} can be reproduced by running the jupyter notebooks provided in our GitHub \cite{ourgithub}.
}

\section{Rippa and MLE setups}
\label{ap:rippa_mle}
In this section, we provide some experiments to justify some engineering choices adopted in this paper for the Rippa and MLE methods. 

One source of variability for the performance of both the Rippa and MLE methods is the necessity to establish a set of candidate $\varepsilon$'s (more recent works, such as \cite{sergey} consider an adaptive candidate set). We consider the following candidate sets:
\begin{align*}
     A&:= \{0.001, 0.002, 0.005, 0.0075, 0.01, 0.02, 0.05, 0.075, 0.1, 0.2, 0.5, 0.75,1,2, 5, 7.5,\\ & \quad \quad \quad 10, 20.0, 50, 75, 100, 200, 500, 1000\} \\
     B&:= \{0.001, ..., 30 \} \quad \mbox{200 equidistant points} \\
     C&:=  B \cup \{ 50, 75, 100, 200, 500, 1000 \}.
\end{align*}

In Figure \ref{fig:candidate_sets}, we show the approximation of the simple function 
\[f(x) = \cos(2/\delta \pi x) + x^2 + x,\]
where $\delta = 0.001,0.01,0.1$. The candidate sets that appear to perform the best are sets $A$ and $C$. In particular, as the interpolation interval diminishes, the chose $\varepsilon$ can lead to an unstable interpolation, as seen in the second and third rows, when $\delta = 0.01 $ and $\delta=0.001$. 

\begin{figure}
\begin{center}
\includegraphics[width=0.45\textwidth]{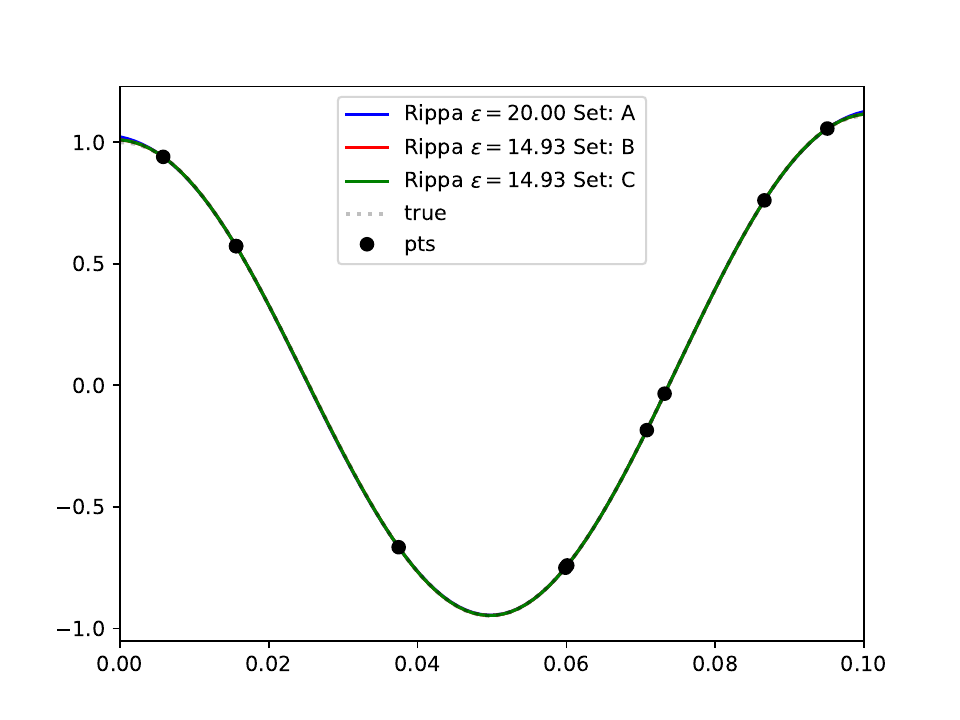}\includegraphics[width=0.45\textwidth]{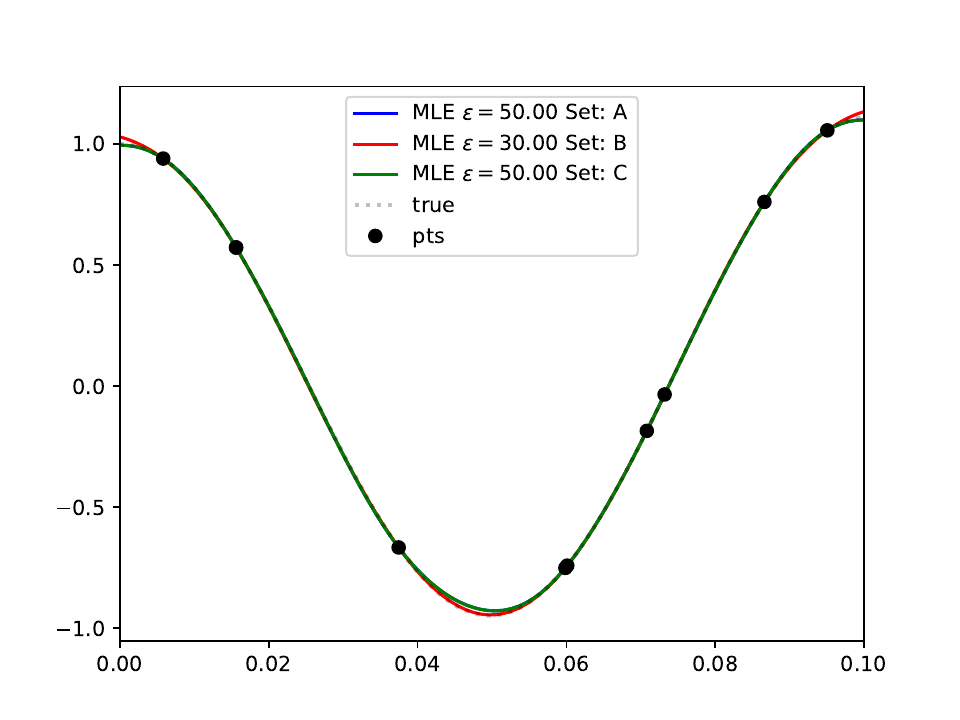}\\
\includegraphics[width=0.45\textwidth]{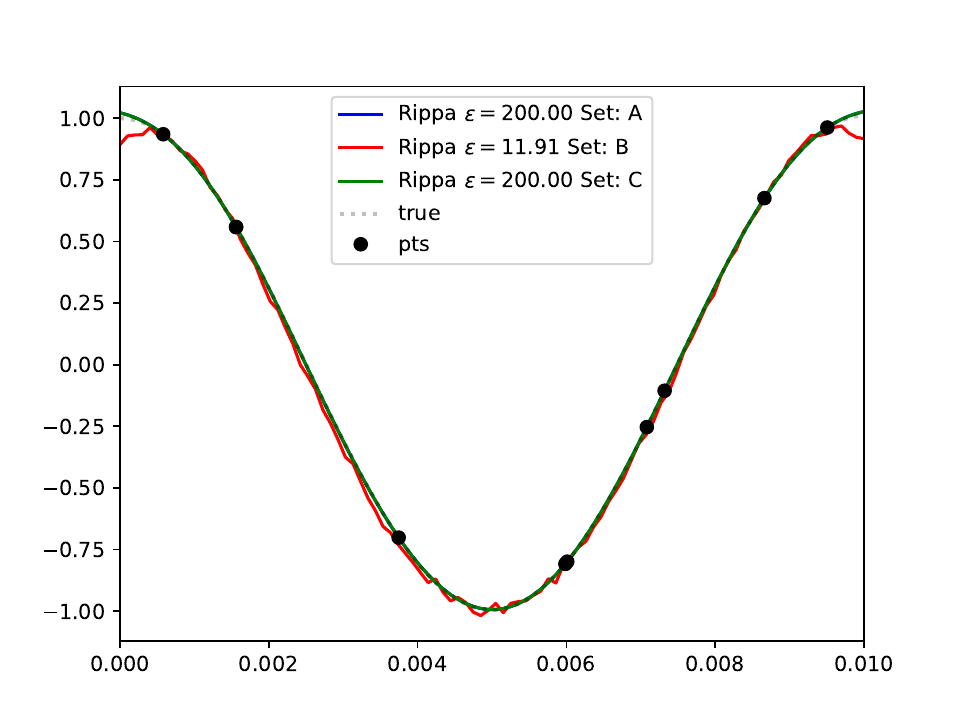}
\includegraphics[width=0.45\textwidth]{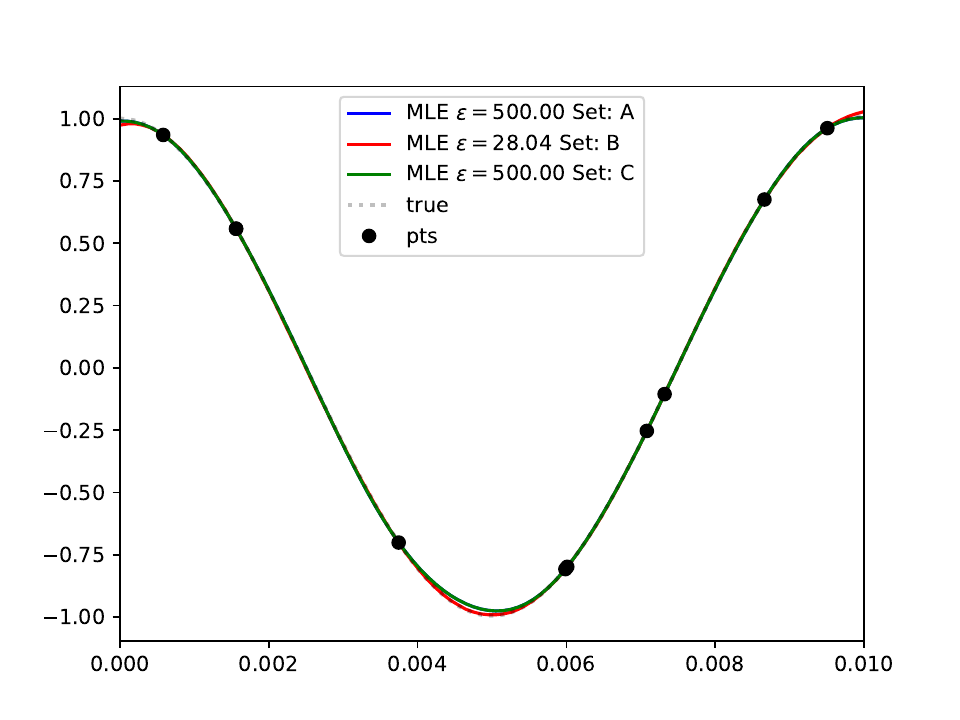}\\
\subfigure[Rippa]{\includegraphics[width=0.45\textwidth]{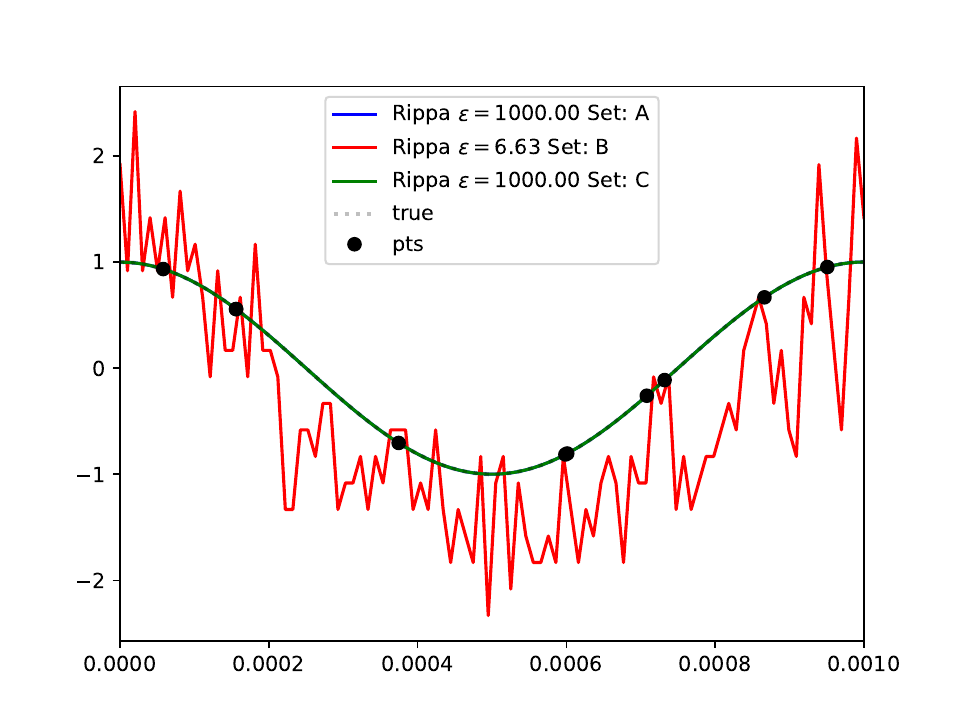}}
\subfigure[MLE]{\includegraphics[width=0.45\textwidth]{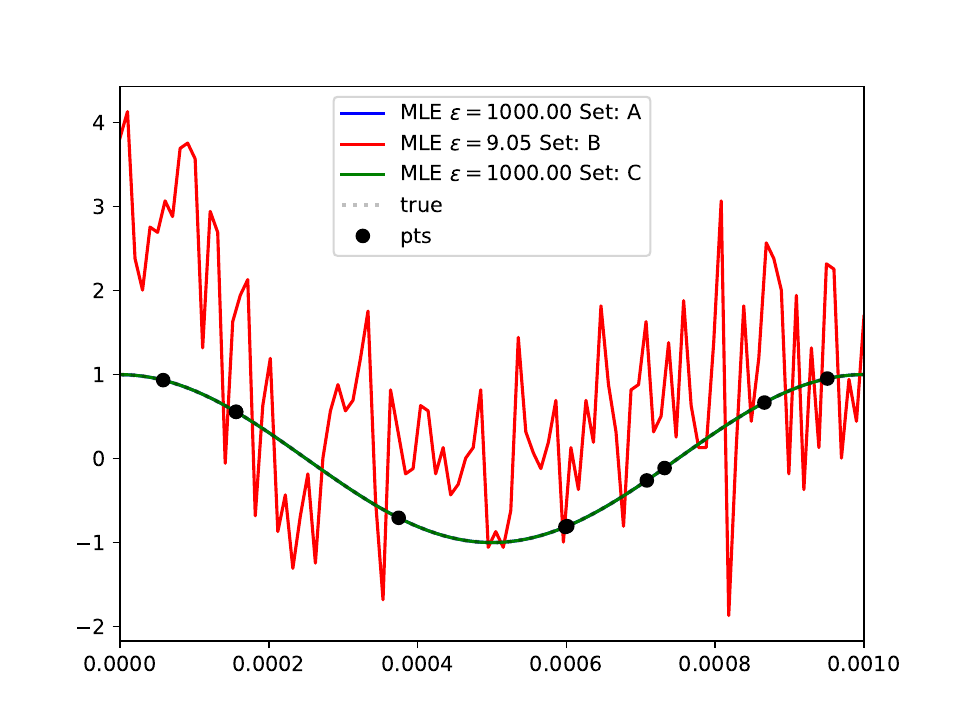}}
\end{center}
\caption{Comparison of the performance of the Rippa and MLE shape parameter selection, varying the candidate sets and the total length of the approximation interval. Row 1: $\delta = 0.1$, row 2:  $\delta = 0.01$ and row 3:  $\delta = 0.001$.}
\label{fig:candidate_sets}
\end{figure}

Additionally, it was noted that the interpolation matrix can become numerically ill-conditioned for some $\varepsilon$, which affects the approximation error of the generated approximator. In Figures \ref{fig:stability_rippa}, we show the error curves for the approximations using Rippa and MLE methods on the 1D interpolation problems using uniform grids (Subsection \ref{sec:interpolation_1d}), while limiting the maximum interpolation matrix condition. While the MLE method is unchanged, the Rippa method appears to be more susceptible to this stability criteria. This can also be seen in Figure \ref{fig:candidate_sets}, left figure in row 2, we did not disregard $\varepsilon$ that lead to an interpolation matrix with a large condition number and the interpolation has some oscillations. During our experiments we limit the maximum of the condition of the interpolation matrix to be $10^{16}$.

\begin{figure}
\begin{center}
\includegraphics[width=0.33\textwidth]{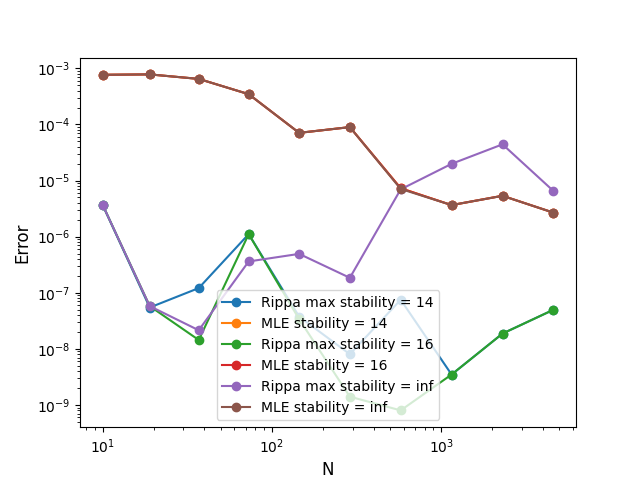}\includegraphics[width=0.33\textwidth]{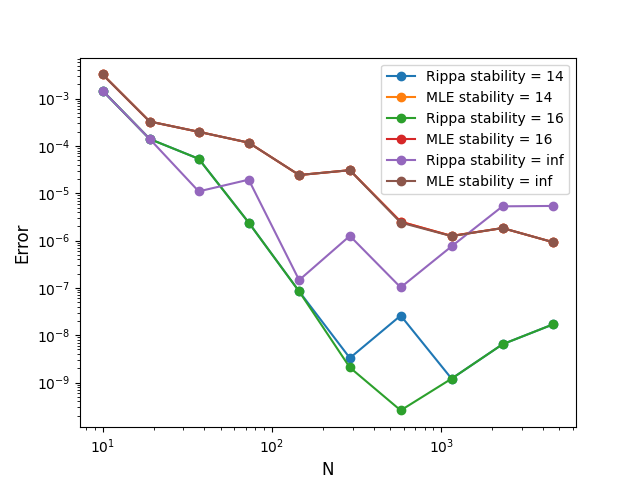}
\includegraphics[width=0.33\textwidth]{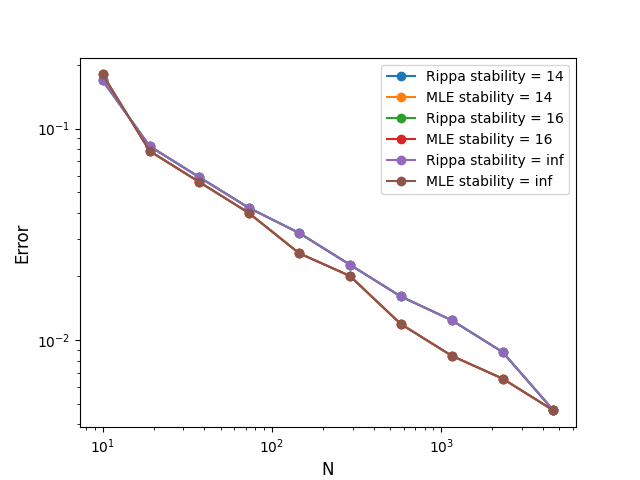}
\end{center}
\caption{Comparison of the performance of the Rippa and MLE methods varying the maximum acceptable condition for the generated interpolation matrix.}
\label{fig:stability_rippa}
\end{figure}

Lastly, excluding $\varepsilon$ that lead to interpolation matrices that have a condition number larger than $10^{16}$, we also evaluated the effect of the candidate sets on the error convergence for the 1-dimensional problems. In Figure \ref{fig:convergence_rippa}, we can note that the candidate sets $A$ and $C$ performance's are quite similar in some regimes (for small $M$ or very large $M$), while the candidate set $B$ does not find a suitable $\varepsilon$ when $M$ increases.

\begin{figure}
\begin{center}
\includegraphics[width=0.45\textwidth]{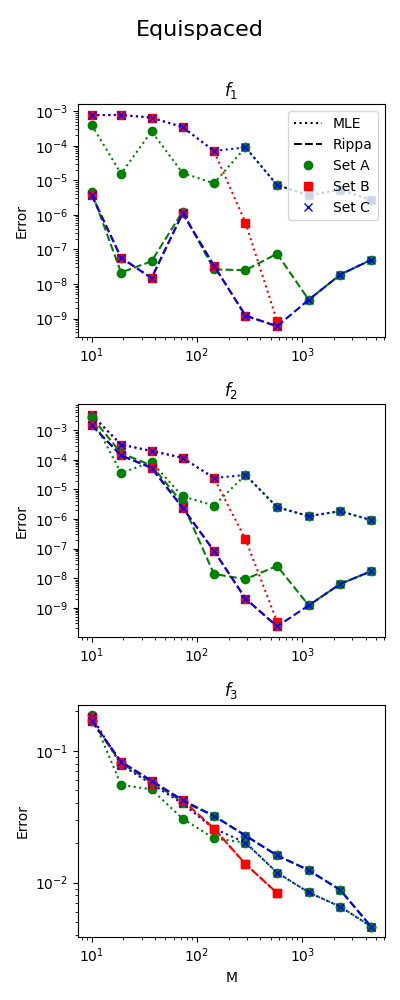}\includegraphics[width=0.45\textwidth]{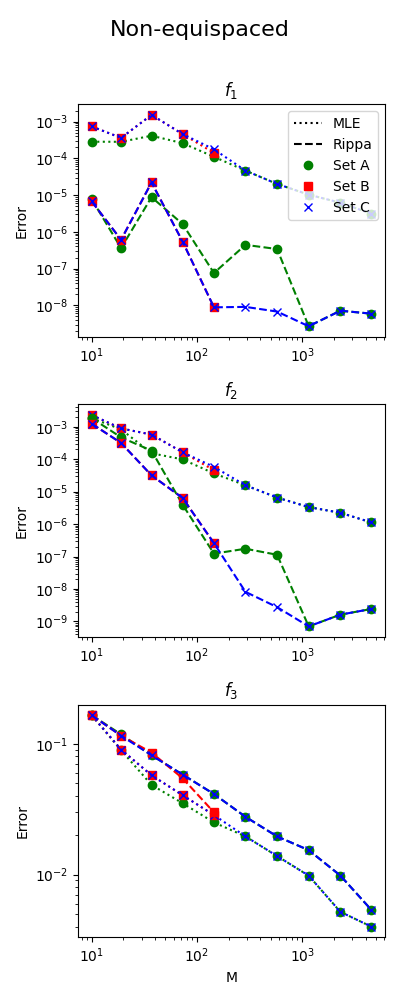}
\end{center}
\caption{Comparison of the performance of the Rippa and MLE shape parameter selection, varying the candidate sets on the 1-dimensional interpolation problems.}
\label{fig:convergence_rippa}
\end{figure}

\bibliographystyle{elsarticle-num}

\end{document}